\documentclass{article}
\usepackage[utf8]{inputenc}
\usepackage{amsmath}
\usepackage{amsthm}
\usepackage{amsfonts}
\usepackage{mathtools}
\usepackage{wasysym}
\usepackage{MnSymbol}
\usepackage{thmtools}
\usepackage{stmaryrd}
\usepackage[letterpaper,margin=1in]{geometry}  
\usepackage{slashed}
\usepackage[english]{babel}				
\usepackage[pdfencoding=auto, psdextra, draft=false]{hyperref}
\usepackage{bookmark}
\usepackage{url}					
\usepackage[T1]{fontenc}
\usepackage{xspace}		
\usepackage{fancyhdr}
\usepackage{enumerate}
\usepackage{mathrsfs}
\usepackage{graphicx}
\usepackage{soul,color}
\usepackage{mathtools}
\usepackage{tikz-cd}
\usepackage[maxbibnames=99]{biblatex}
\usepackage{csquotes}
\usepackage{chngcntr}
\usepackage[bbgreekl]{mathbbol}
\counterwithin{equation}{section}
\DeclareSymbolFontAlphabet{\mathbb}{AMSb}
\DeclareSymbolFontAlphabet{\mathbbl}{bbold}

\renewcommand{\epsilon}{\varepsilon}
\renewcommand{\rho}{\varrho}
\renewcommand{\phi}{\varphi}
\newcommand{\NN}{\ensuremath{\mathbb{N}}\xspace}
\newcommand{\ZZ}{\ensuremath{\mathbb{Z}}\xspace}

\newcommand{\CC}{\ensuremath{\mathbb{C}}\xspace}
\newcommand{\FF}{\ensuremath{\mathbb{F}}\xspace}
\newcommand{\TT}{\ensuremath{\mathbb{T}}\xspace}

\newcommand{\DD}{\ensuremath{\mathbbl{\Delta}}\xspace}
\newcommand{\Sp}{\mathcal{S}p}
\newcommand{\tc}{\ensuremath{\mathrm{TC}}}
\newcommand{\thh}{\ensuremath{\mathrm{THH}}}
\newcommand{\tp}{\ensuremath{\mathrm{TP}}}
\newcommand{\tr}{\ensuremath{\mathrm{TR}}}
\DeclareMathOperator*{\colim}{\ensuremath{\operatorname{colim}}}
\DeclareMathOperator{\fdim}{\operatorname{dim}_F}
\DeclareMathOperator{\ndim}{\operatorname{dim}_{\mathcal{N}}}
\newtheorem{thm}{Theorem}[section]
\newtheorem{prop}[thm]{Proposition}
\newtheorem{lem}[thm]{Lemma}
\newtheorem{cor}[thm]{Corollary}
\newtheorem{defn}[thm]{Definition}

\theoremstyle{remark}
\newtheorem{rem}[thm]{Remark}
\newtheorem{note}[thm]{Notation}
\newtheorem{ex}[thm]{Example}
\newtheorem{con}[thm]{Construction}

\title{\texorpdfstring{$K$-Theory of Truncated Polynomials}{K-Theory of Truncated Polynomials}}

\author{Noah Riggenbach}

\begin{document}
\maketitle
\begin{abstract}
    We study the algebraic $K$-theory of rings of the form $R[x]/x^e$. We do this via trace methods and filtrations on topological Hochschild homology and related theories by quasisyntomic sheaves. We produce computations for $R$ a perfectoid ring in terms of the big Witt vectors of $R$, for $R$ a smooth curve over a perfectoid ring in terms of the prismatic cohomology of $R$, and for $R$ a complete mixed characteristic discrete valuation rings with perfect residue field in terms of the prismatic cohomology and Hodge-Tate divisor of $R$.
\end{abstract}
\tableofcontents

\section{Introduction}

The algebraic $K$-theory of truncated polynomials has been studied by many authors. As the simplest singular family of rings and a prototypical example of nilpotent thickenings it serves as a fantastic test bed for advances in trace methods. In addition, it also plays an important role among nilpotent extensions. For example for any ring $R$ and nonunit $\pi\in R$ there is a filtration on $R/\pi^n$ which has the same associated graded terms as the $x$-adic filtration on $(R/\pi[x])/x^n$. 

The problem of computing the algebraic $K$-theory of truncated polynomial algebras in all degrees was first successfully done by Hesselholt and Madsen in \cite[Theorem E]{HM_Witt_vector_K_theory} where they compute $K_*(k[x]/x^2, (x))$ for $k$ a perfect field of characteristic $p>0$. This was extended by the same authors to a calculation of $K_*(k[x]/x^e,(x))$ in \cite{Cyclic_Polytopes}. Both of these calculations were revisited by Speirs in \cite{Speirs_truncated} using the reformulation of topological cyclic homology due to Nikolaus and Scholze in \cite{NS}. In the case of $e=2$ these results were then revisited again by Mathew in \cite[Theorem 10.4]{Mathew_recent_advances} using the quasisyntomic filtration of \cite{BMS2}. The case $e>2$ was then computed by Sulyma in \cite{Sulyma_truncated} using these methods. 

These computations are given in terms of the big Witt vectors of $k$ and the Verschiebung maps and is surprisingly simple. Generalizing these results to different base rings has been done in several directions. The case of $A/k$ a smooth algebra was considered next. In \cite{Hesselholt-Madsen} it is shown that the groups $K_*(A[x]/x^e, (x))$ fit into a long exact sequence with truncated de Rham-Witt complexes and Verschiebung maps in a similar fashion as their answer when $A=k$. When the base ring is the tensor product of other truncated polynomial algebras this was studied by Angeltveit, Gerhardt, Hill, and Lindenstrauss in \cite{AGHL} where they calculate the $K$-groups explicitly when the characteristic does not divide the exponents and prove an inductive formula when they do. Moving away from the positive characteristic case, the $K$-theory of truncated polynomials over the integers has also been studied by Angeltveit, Gerhardt, and Hesselholt in \cite{AGH} where they show that the even groups are finite, the odd groups are free, and compute the cardinality and ranks of these groups respectively. 

Using truncated polynomials it was then shown by Betley and Schlichtkrull in \cite{Betley_Schlichtkrull} that profinitely one can recover the topological cyclic homology of the base. Specifically one has an equivalence $\tc(A;\widehat{\ZZ})\simeq \lim \Omega K(A[x]/x^e, (x);\widehat{\ZZ})$. Here the limit is over both the map $A[x]/x^{ef}\to A[x]/x^e$ given by $x\mapsto x$ and the maps $A[x]/x^e\to A[x]/x^{ef}$ given by $x\mapsto x^f$. These maps were then identified in terms of the calculations done above when $A/k$ is smooth in \cite{Hesselholt_tower}. In addition it was also shown by Hesselholt in \cite{Hesselholt_curves} that the result of \cite{Betley_Schlichtkrull} could be refined to a computation of $\tr(A)$ when $A$ is a commutative $\ZZ/p^j$-algebra for some $j\geq 1$. The conditions on $A$ were removed by McCandless in \cite[Theorem A]{McCandless_curves}. 

The goal of this paper is to use the quasisyntomic topology to extend these computations and to study topological restriction homology. Our first result is to extend the computations for $k$ a perfect $\FF_p$-algebra to perfectoid rings.

\begin{thm}\label{thm: main computation for perfectoids}
Let $R$ be a perfectoid ring. Then \[K_{2r-1}(R[x]/x^e, (x);\ZZ_p)\cong \mathbb{W}_{re}(R)/V_e\mathbb{W}_r(R)\] and the even groups are trivial.
\end{thm}

We will also consider the results for $\ZZ$ (and other discrete valuation rings) and for smooth algebras over perfectoid rings. The recent advances in topological cyclic homology will allow us to consider both of these cases at the same time. Specifically, these rings have bounds on their F-dimensions and $\mathcal{N}$-dimensions, as defined in Subsection~\ref{ssec:f and n dimension}.

In order to state the next Theorem it will be helpful to first fix some notation. First define $J_p:=\NN\setminus p\NN$ and for any $e\in \ZZ$ let $e'\in J_p$ be the integer $e':=e/p^{v_p(e)}$. As in \cite[Section 2]{Speirs_truncated} let $s:=s(p,r,u)$ be the unique positive integer such that \[up^{s-1}\leq r<up^s\] if such an integer exists and $s=0$ otherwise. Define the function $t:=t(u,p,s,e)$ to be the function $t=\lfloor \frac{up^{s-1}-1}{e}\rfloor$.

\begin{thm}\label{thm: computation for F-curves}
Let $S$ be an F-smooth ring with $\ndim(S)\leq 1$(see Definition~\ref{defn: n-dim}) and $r\geq 2$. Then there are isomorphisms
\begin{align*}
&K_{2r-1}(S[x]/x^e, (x);\ZZ_p)\simeq \\
&\prod_{u\in J_p\setminus e'J_p}
{H^0\left(S, \mathcal{I}^{-\lfloor t/p\rfloor}\otimes \mathcal{N}^{\geq r-t-1}\mathcal{O}_{\widehat{\DD}}/(\mathcal{N}^{\geq r-t}\mathcal{O}_{\widehat{\DD}}\otimes \mathcal{I}_{s-1})\{r-1\}\right)} \\ 
&\times \prod_{u\in e'J_p}\begin{cases}{H^0\left(S, \mathcal{I}^{-\lfloor t/p\rfloor}\otimes \mathcal{N}^{\geq r-t-1}\mathcal{O}_{\widehat{\DD}}/(\mathcal{N}^{\geq r-t}\mathcal{O}_{\widehat{\DD}}\otimes \mathcal{I}_{s-1})\{r-1\}\right)} & \textrm{if }up^{v_p(e)}\geq er\\
H^0\left(S, \mathcal{I}^{-\lfloor t/p\rfloor}\otimes \mathcal{N}^{\geq r-t-1}\mathcal{O}_{\widehat{\DD}}/(\phi^{s-1-v_p(e)})^*(\mathcal{I}_{v_p(e)})\{r-1\}\right) & \textrm{otherwise}\end{cases} \end{align*}
and
\begin{align*}
&K_{2r-2}(S[x]/x^e, (x);\ZZ_p)\simeq \\
&\prod_{u\in J_p\setminus e'J_p}
{H^1\left(S,\mathcal{I}^{-\lfloor t/p\rfloor}\otimes \mathcal{N}^{\geq r-t-1}\mathcal{O}_{\widehat{\DD}}/(\mathcal{N}^{\geq r-t}\mathcal{O}_{\widehat{\DD}}\otimes \mathcal{I}_{s-1})\{r-1\}\right)} \\ 
&\times \prod_{u\in e'J_p}\begin{cases}{H^1\left(S, \mathcal{I}^{-\lfloor t/p\rfloor}\otimes \mathcal{N}^{\geq r-t-1}\mathcal{O}_{\widehat{\DD}}/(\mathcal{N}^{\geq r-t}\mathcal{O}_{\widehat{\DD}}\otimes \mathcal{I}_{s-1})\{r-1\}\right)} & \textrm{if }up^{v_p(e)}\geq er\\
H^1\left(S, \mathcal{I}^{-\lfloor t/p\rfloor}\otimes \mathcal{N}^{\geq r-t-1}\mathcal{O}_{\widehat{\DD}}/(\phi^{s-1-v_p(e)})^*(\mathcal{I}_{v_p(e)})\{r-1\}\right) & \textrm{otherwise}\end{cases} \end{align*}
with $s=s(p,re,u)$. When $s=0$ we define the above groups to be zero.
\end{thm}

\begin{rem}
It is helpful for \cite{Speirs_truncated} to introduce a function $h:=h(p,r, e,u)$ similar to the $s$ function above. This function allows them to avoid the seperate cases in their theorem statement. This will not work in our case since the Frobenius pullback $(\phi^s)^*(\mathcal{I})$ will in general not be equivalent to $\mathcal{I}$. In \cite{Speirs_truncated} the only case considered is when $\mathcal{I}=p\mathcal{O}_{\widehat{\DD}}$ and so the Frobenius pullback will have no effect. 
\end{rem}

The strategy to prove this is to first construct a spectral sequence for general $p$-complete rings $S$ with $E_2$ page described in terms of quasisyntomic sheaves. We do this in Section~\ref{sec: bms spectral sequence for truncated polynomial algebras}. While the spectral sequence we construct is in terms of quasisyntomic sheaves, we emphasize here that the assumption on $\ndim(S)$ is not just to get the spectral sequence to collapse but is used in significantly simplifying the $E_2$ page. The statement on the $E_2$ page works when $r> \frac{p}{p-1}(\fdim(S)-1)+1$ without any assumptions on $\ndim(S)$. The exact statement (with $i=r-1$) is Theorem~\ref{thm: tc filtration simplification segal}.

We include two applications of our methods. The first application is to topological restriction homology. In \cite[Lemma 7.11]{k1_local_tr} Mathew computed $\tr(\mathcal{O}_C)$ where $\mathcal{O}_C\subseteq C$ is the ring of integers in a spherically complete algebraically closed nonarchemidian field of mixed characteristic $(0,p)$. One of the key steps in the proof is showing that the odd groups $\tr_{2i+1}(\mathcal{O}_C)$ vanish which crucially uses spherical completeness. Note that $C$ is also a perfectoid field in the sense of Tate, and so we have another measure of spherical completeness given by the cokernel of the map $\theta_\infty: A_{inf}(\mathcal{O}_C)\to W(\mathcal{O}_C)$ by \cite[Lemma 3.23]{BMS1}. 

\begin{thm}\label{thm: main calculation for tr}
Let $S$ be a F-smooth ring with $\ndim(S)\leq 1$(see Definition~\ref{defn: n-dim}). Then for all $i\geq 1$ there are functorial in $S$ chain complexes $\theta^{\widehat{\DD}_S}_{\infty}(i)$ such that
\begin{enumerate}
    \item when $R$ is a perfectoid ring there is a quasi-isomorphism \[\theta^{\widehat{\DD}_R}_{\infty}(1)\simeq ( \ldots\to 0\to A_{inf}(R)\xrightarrow{\theta_{\infty}} W(R)\to 0\to \ldots)\] with $A_{inf}(R)$ is degree zero and $\theta^{\widehat{\DD}_R}_\infty(i)\cong \operatorname{fib}(A_{inf}(R)\to \lim_{n\to \infty} A_{inf}(R)/(d_n)^i)$ more generally;
    \item there are short exact sequences  \[0\to H^2(\theta^{\widehat{\DD}_S}_\infty(i+1)) \to  \tr_{2i}(S;\ZZ_p) \to  H^0(\theta^{\widehat{\DD}_S}_{\infty}(i))\to 0 \] and isomorphisms \[\tr_{2i-1}(S)\cong H^1(\theta^{\widehat{\DD}_S}_\infty (i))\] for all $i\geq 1$, and in particular the groups $\mathrm{TR}_{2i-1}(\mathcal{O}_{C})$ vanish for $C$ a spherically complete perfectoid field; and
    \item for all $i\geq 1$,  $\theta^{\widehat{\DD}_S}_\infty (i)\in D^{[0,\ndim(S)+1]}(\ZZ_p)$.
\end{enumerate}
\end{thm}

Without the assumption that $\ndim(S)\leq 1$ there is a spectral sequence converging to $\tr(S)$ and $\mathrm{gr}^i\tr(S)\simeq \theta^{\widehat{\DD}_S}_\infty (i)$ for $i\geq \fdim(S)$. These complexes should be the cohomology of a quasisyntomic sheaf and should exist for all $i\in \NN$ without the assumption on $\fdim(S)$, but for this paper we only consider the above case. One could take $\theta^{\widehat{\DD}_S}_\infty(i):=\operatorname{R\Gamma}_{\mathrm{QSyn}}(S;\mathrm{gr}^i\tr(-)[-2i])$ where the filtration is the one constructed in Section~\ref{sec: tr stuff} but it takes some work to connect this definition to spherical completeness for perfectoid rings. 

For the next application of our results, we will change our notation slightly. We have thus far used $e$ for the exponent of $x$ which we are modding out by. This has been in an effort to match with the notation used by previous authors who have studied the $K$-theory of truncated polynomial algebras. In what follows, we will wish to talk about the ramification degree which is usually denoted by $e$. Thus for the rest of this section and Section~\ref{sec: agh proof} we will use $e$ for the ramification degree and $n$ for the exponent of $x$.

As another application of our ideas, Theorem~\ref{thm: computation for F-curves} can be used to recover and extend an interesting result of Angeltveit, Gerhardt, and Hesselholt. Specifically, in \cite{AGH} the following theorem is proven.

\begin{thm}[Theorem A, \cite{AGH}]\label{thm: agh}
Let $n$ be a positive integer and $i$ a non-negative integer. Then:
\begin{enumerate}
    \item The abelian group $K_{2i+1}(\ZZ[x]/x^n, (x))$ if free of rank $n-1$.
    \item The abelian group $K_{2i}(\ZZ[x]/x^n, (x))$ is finite of order $(ni)!(i!)^{n-2}$.
\end{enumerate}
\end{thm}
This result builds on earlier work of Soul{\'e} in \cite{Soule} and Staffeldt in \cite{Staffeldt} which computed the ranks of these groups. In addition, the results of Soul{\'e} and Staffeldt apply not just to the integers but to any number ring.

One of the surprising things about the above result is that there is no torsion group appearing in the odd relative $K$-groups. While we know from the results of Soul{\'e} and Staffeldt that the rational $K$-theory agrees with the rationalization of the above, there is no a priori reason to expect no $p$-torsion to appear in the odd homotopy groups.

Our results allow us to extend the above to more general discrete valuation rings. 

\begin{thm}\label{thm: AGH theorem}
Let $n$ be a positive integer, $i$ a non-negative integer, and $A$ a CDVR with of mixed characteristic $(0,p)$, uniformizer $\pi$ with Eisenstein polynomial $E(x)$, ramification index $e=\deg(E(x))$, and perfect residue field $k$. Then there are isomorphisms \[K_{2i+1}(A[x]/x^n, (x);\ZZ_p)\simeq A^{n-1}.\] If in addition $k$ is finite then \[v_p(|K_{2i}(A[x]/x^n, (x);\ZZ_p)|)=ev_p(|k|)v_p((ni)!(i!)^{(n-2)})+v_p(|k|)v_\pi(E'(\pi))(ni-i)\] and since these groups are $p$-torsion this determines the cardinality.
\end{thm}

\begin{rem}
    When $A$ is tamely ramified we have that $v_\pi(E'(\pi))=e-1$, and in particular when $e=1$ the final term in the above formula vanishes.
\end{rem}

\begin{rem}
    If instead we take $A=\mathcal{O}_K$ where $K/\mathbb{Q}$ is a finite extension, then \[K(A[x]/x^n, (x);\ZZ_p)\simeq K(A^{\wedge}_p[x]/x^n,(x);\ZZ_p)\] and so we may apply the above result. In particular \[v_p(|K_{2i}(A[x]/x^n,(x);\ZZ_p)|)=[K:\mathbb{Q}]v_p((ni)!(i!)^{n-2})+v_p(|A^\wedge_p/\pi|)v_\pi(E'(\pi))(ni-i)\] where $\pi\in A^\wedge_p$ is a uniformizer with Eisenstein polynomial $E(x)$. Applying a fracture square argument, where the rational input is coming from \cite{Soule, Staffeldt}, we find that \[|K_{2i}(A[x]/x^n,(x))|=[(ni)!(i!)^{n-2}]^{[K:\mathbb{Q}]}\times|A/\mathfrak{D}_{K/\mathbb{Q}}|^{ni-i}\] where $\mathfrak{D}_{K/\mathbb{Q}}$ is the different, ie the annihilator of $\Omega_{A/\ZZ}$ which after $p$-completion is given by the ideal $(\mathfrak{D}_{K/\mathbb{Q}})^\wedge_p=(E'(\pi))$. The exponent of $[K:\mathbb{Q}]$ is appearing in the $p$-adic valuation as $ef$ where $f=v_p(|A^\wedge_p/\pi_p|)=[A^\wedge_p/\pi : \FF_p]$.
\end{rem}

To the author's knowledge both parts of this computation are new for $\mathcal{O}\neq \ZZ$.

\textbf{Acknowledgements.}
The author would like to thank Benjamin Antieau, Micah Darrell, Sanath Devalapurkar, Jeremy Hahn, Ayelet Lindenstrauss, Akhil Mathew, Catherine Ray, and Yuri Sulyma for many helpful conversations. The author would also like to thank Martin Speirs for their many helpful conversations, guidance, and collaboration on the early stages of this project. Finally, the author would like to thank an anonymous referee who's comments and suggestions greatly improved the readability of this paper.

\section{The sheaves and definitions appearing in Theorem~\ref{thm: computation for F-curves}}\label{sec: Background}
This Section is dedicated to unwinding the objects appearing in Theorem~\ref{thm: computation for F-curves} and Theorem~\ref{thm: main calculation for tr}, and for recording results and constructions which will be used throughout the paper. In Subsection~\ref{ssec: background info} we will mostly be discussing the sheaves appearing in Theorem~\ref{thm: computation for F-curves}. At the end we will also record a theorem in Borel equivariant homotopy theory which we will use repeatedly. None of the constructions or results in Subsection~\ref{ssec: background info} are due to us. In Subsection~\ref{ssec:f and n dimension} we will define the notation of F-dimension and $\mathcal{N}$-dimension.

\subsection{Background information}\label{ssec: background info}
We begin by recalling the quasisyntomic site. 

\begin{defn}[Definition 4.10, \cite{BMS2}] 

A ring $A$ is quasisyntomic if the following conditions hold: 

\begin{enumerate} 

    \item $A$ is $p$-complete and has bounded $p^\infty$-torsion; 

    \item $L_{A/\ZZ_p}\in D(A)$ has $p$-complete Tor amplitude in $[-1,0]$, i.e. $L_{A/\ZZ_p}\otimes_A^{\mathbb{L}}A/p\in D(A/p)$ has Tor amplitude in $[-1,0]$. 

\end{enumerate} 

Let $\mathrm{QSyn}$ denote the category of such rings. We say that a map $A\to B$ in $\mathrm{QSyn}$ is a quasisyntomic map (respectively quasisyntomic cover) if 

\begin{enumerate} 

    \item $B$ is a $p$-completely (faithfully) flat $A$-module, i.e. $B\otimes_A^{\mathbb{L}} A/p$ is a (faithfully) flat $A/p$-module in degree $0$; 

    \item and $L_{B/A}\in D(B)$ has $p$-complete Tor amplitude in $[-1,0]$. 

\end{enumerate} 

\end{defn} 

The category of quasisyntomic rings contains two subcategories of particular importance for us. The first is the category of perfectoid rings which we now define.

\begin{defn}[Definition 3.5, \cite{BMS1}]
    A ring $R$ is perfectoid if
    \begin{enumerate}
        \item $R$ is $\pi$-adically complete for some $\pi\in R$ such that $\pi^p\mid p$;
        \item the Frobenius map $\phi:R/p\to R/p$ is surjective;
        \item and the kernel of the map $\theta:A_{inf}(R)\to R$ is principle.
    \end{enumerate}
    following \cite[Definition 3.2(2)]{Bhatt_Scholze} we will call a choice of generator $\ker(\theta)=(d)$ and orientation of $R$. We will also frequently use the notation $\Tilde{d}=\phi(d)$ and $\Tilde{d}_n = \phi(d)\phi^2(d)\ldots \phi^n(d)$.
\end{defn}

The second subcategory is the category of quasiregular semiperfectoid rings.  

\begin{defn}[Definition 4.20, \cite{BMS2}] 

A ring $S$ is quasiregular semiperfectoid if: 

\begin{enumerate} 

    \item $S\in \mathrm{QSyn}$; 

    \item there is a map $R\to S$ where $R$ is a perfectoid ring; 

    \item and $S/p$ is semiperfect, i.e. the Frobenius map $x\mapsto x^p$ is surjective. 

\end{enumerate} 

\end{defn}

There are several equivalent ways of defining when $S$ is quasiregular semiperfectoid. For instance, we can replace items $(i)$ and $(ii)$ in the above definition with the requirement that there is some surjective map $R\to S$ where $R$ is perfectoid.

We are now ready to define the relevant sites. 

\begin{defn}[Lemma 4.17, \cite{BMS2}; Lemma 4.27, \cite{BMS2}; Variant 4.33, \cite{BMS2}; Variant 4.35, \cite{BMS2}] 

Let $A\in \mathrm{QSyn}$. Define the sites $\mathrm{QSyn}_A$, $\mathrm{QRSPerfd}_A$, $\mathrm{qSyn}_A$, and $\mathrm{qrsPerfd}_A$ to be the sites of quasisyntomic rings with a map from $A$, quasiregular semiperfectoid rings with a map from $A$, quasisyntomic $A$-algebras, and quasiregular semiperfectoid $A$-algebras, respectively. All of the sites are topologized using quasisyntomic covers.  

\end{defn} 

As stated before the definition of quasiregular semiperfectoid rings, these rings form a basis for the above sites.

\begin{thm}[Proposition 4.31, \cite{BMS2}]\label{thm: qrsp are basis} 

The functors $U:\mathrm{QRSPerfd}_{A}^{op}\to \mathrm{QSyn}_A^{op}$ and $u:\mathrm{qrsPerfd}_{A}^{op}\to \mathrm{qSyn}_A^{op}$ induce equivalences  

\[\operatorname{Shv}_\mathcal{C}(\mathrm{QSyn}^{op})\simeq \operatorname{Shv}_{\mathcal{C}}(\mathrm{QRSPerfd}_A^{op})\] and \[\operatorname{Shv}_\mathcal{C}(\mathrm{qSyn}^{op})\simeq \operatorname{Shv}_{\mathcal{C}}(\mathrm{qrsPerfd}_A^{op})\] for any presentable $\infty$-category $\mathcal{C}$. We will refer to the inverse functor as unfolding. 

\end{thm}

We will often use the above Theorem to compute sheaves by reducing to the quasiregular semiperfectoid ring case. We can also use this equivalence to define several sheaves on the sites $\mathrm{QSyn}^{op}_A$ and $\mathrm{qSyn}^{op}_A$. We begin by defining $\mathcal{O}_{\widehat{\DD}}\in \operatorname{Shv}_{D(\ZZ_p)}(\mathrm{QSyn}^{op}_A)$. Note that we are using different notation than \cite{BMS2} in the following. Originally, we did follow their notation, but with some of the formulas appearing in our paper we found that their notation caused confusion between the sheaves and the derived global sections of the sheaves in question. This was not a danger in their work but would certainly cause issues in ours since we are often taking tensor products of sheaves and global sections is not symmetric monoidal.

\begin{note} 

Define the sheaf $\mathcal{O}_{\widehat{\DD}}\in \operatorname{Shv}_{D(\ZZ_p)}(\mathrm{QSyn}^{op}_A)$ to be the sheaf given under the above equivalence from the sheaf $\pi_0(\tc^{-}(-;\ZZ_p))$ on $\mathrm{QRSPerfd}^{op}_A$. We will use the same notation for the corresponding sheaf on $\mathrm{qSyn}_A^{op}$. This sheaf comes with a natural endomorphism $\phi:\mathcal{O}_{\widehat{\DD}}\to \mathcal{O}_{\widehat{\DD}}$ which on $\mathrm{QRSPerfd}_A$ is given by $\pi_0(\phi_p)$. In \cite{BMS2} these objects are denoted by $\widehat{\DD}_{-}$ and $\phi$, respectively 

\end{note}

This notation matches up with that of \cite{ALB_prismatic_Dieudonne}, as will the following notation. One of the main results of \cite{BMS2}[Construction 7.7] is that the above sheaf lifts to a sheaf valued in the complete derived filtered category $\widehat{\operatorname{DF}}(\ZZ_p)$. This is the full subcategory of $\operatorname{DF}(\ZZ_p):= \operatorname{Fun}(\ZZ^{op}, \operatorname{D}(\ZZ_p))$ spanned by those functors $F^{\geq *}:\ZZ^{op}\to \operatorname{D}(\ZZ_p)$ such that $F^{\geq -\infty}:=\lim_{\ZZ^{op}}F^{\geq *}=0$.

\begin{note} 

Let $\mathcal{N}^{\geq *}\mathcal{O}_{\widehat{\DD}}\in \operatorname{Shv}_{\widehat{\operatorname{DF}}(\ZZ_p)}(\mathrm{QSyn}^{op}_A)$ denote the Nygaard filtration on $\mathcal{O}_{\widehat{\DD}}$ constructed in \cite[Construction 7.7]{BMS2}. For each $i$ we can also describe $\mathcal{N}^{\geq i}\mathcal{O}_{\widehat{\DD}}$ as the sheaf corresponding under Theorem~\ref{thm: qrsp are basis} to the sheaf $\operatorname{Im}(v^i:\tc^-_{2i}(-;\ZZ_p)\to \tc^-_0(-;\ZZ_p))$ on $\mathrm{QRSPerfd}$. Here $v\in \tc^-_{-2}(-;\ZZ_p)$ is a lift of the first chern class along the map $\tc^{-}(-;\ZZ_p)\to (-)^{h\TT}$.\footnote{The element $v\in \tc^{-}(-;\ZZ_p)$ cannot be chosen consistently even on $\mathrm{QRSPerfd}$. Even so, the image of any given choice will always be the same and agree with the filtration produced by the homotopy fixed points spectral sequence.} 

\end{note} 

Note that for any $S\in \mathrm{QSyn}_A$ we have by definition that $\mathcal{N}^{\geq i}\mathcal{O}_{\widehat{\DD}}(S)=\mathcal{N}^{\geq i}\widehat{\DD}_S$ as defined in \cite[Construction 7.7]{BMS2} for any $i\in \ZZ$.

In addition to the above definitions there are a number of line bundles over $\mathcal{O}_{\widehat{\DD}}$ which we will need. We note here that there are some size issues which make some definitions difficult, and that these size issues are in some sense a fundamental part of the definition of these sites. More specifically, since the notion of covers in this site has no finiteness assumption it is not automatic that there is a sheafification functor, see for example \cite[Theorem 5.5]{Waterhouse} for an example of the lack of such a sheafification functor in the fpqc site. Thus to define sheaves such as the tensor product of two sheaves in the quasisyntomic site, or to calculate cohomology groups, we cannot use the standard definition of taking the tensor product or resolution in the category of presheaves and sheafify the result. While in general this can cause quite a few issues, throughout this paper we will be working with line bundles over $\mathcal{O}_{\widehat{\DD}}$ or $\mathcal{N}^{\geq *}\mathcal{O}_{\widehat{\DD}}$ and so there are no set theoretic issues with forming things like the tensor product of the sheaves we will be considering.

\begin{note} 

Let $\mathcal{I}=:\mathcal{I}_1\in \operatorname{Shv}_{D(\ZZ_p)}(\mathrm{QSyn}_A^{op})$ be the $\mathcal{O}_{\widehat{\DD}}$-line bundle which corresponds under Theorem~\ref{thm: qrsp are basis} to the sheaf given by $\operatorname{Ker}(\tc^{-}_0(-;\ZZ_p)\to \pi_0(\thh(-;\ZZ_p)^{tC_p}))$. 

\end{note} 

A few remarks are in order. First note that this is indeed a line bundle since on $\mathrm{QRSPerfd}_A$ every ring admits a map from a perfectoid ring by definition and so \cite[Proposition 6.4]{BMS2} and \cite[Theorem 7.2(5)]{BMS2} apply. We also note that this is not the same sheaf as in \cite{ALB_prismatic_Dieudonne} since our sheaf takes values in \textit{Nygaard complete} groups. Since we never use the non-complete version of $\mathcal{I}$ and we need to subscript for a different construction we decided to use the above.

\begin{note} 

Define $\mathcal{I}_m:=\mathcal{I}\otimes_{\mathcal{O}_{\widehat{\DD}}}\phi^*(\mathcal{I})\otimes_{\mathcal{O}_{\widehat{\DD}}}\ldots \otimes_{\mathcal{O}_{\widehat{\DD}}}(\phi^{m-1})^*(\mathcal{I})$. By \cite[Theorem 7.2(5)]{BMS2} This is the same as the sheaf which under Theorem~\ref{thm: qrsp are basis} corresponds to the sheaf $\mathcal{I}\cdot \phi^*(\mathcal{I})\cdot\ldots\cdot (\phi^{m-1})^*(\mathcal{I})$. 

\end{note}

We will need one more line family of line bundles. Just as in the previous case we will not distinguish our notation, which is Nygaard complete, with the notation for the non-complete setting. 

\begin{note} 

Define $\mathcal{O}_{\widehat{\DD}}\{1\}\in \operatorname{Shv}_{D(\ZZ_p)}(\mathrm{QSyn}_A^{op})$ to be the sheaf which under Theorem~\ref{thm: qrsp are basis} corresponds to the sheaf $\tp_2(-;\ZZ_p)$. For all $i\in \ZZ$ and any $\mathcal{O}_{\widehat{\DD}}$-vector bundle $\mathcal{V}$ define $\mathcal{V}\{i\}:=\mathcal{V}\otimes_{\mathcal{O}_{\widehat{\DD}}}\mathcal{O}_{\widehat{\DD}}\{1\}^{\otimes_{\mathcal{O}_{\widehat{\DD}}} i}$. 

\end{note}

The sheaves $\mathcal{O}_{\widehat{\DD}}\{i\}$ are line bundles as can be seen by reducing to $\mathrm{QRSPerfd}_A$ where $\mathrm{TP}(-;\ZZ_p)$ is $2$-periodic by the paragraph before \cite[Theorem 7.2]{BMS2}. Since there are set theoretic issues with defining the global tensor product anyway, we will use $-\otimes -$ to mean $-\otimes_{\mathcal{O}_{\widehat{\DD}}}-$ when we are working with $\mathcal{O}_{\widehat{\DD}}$ vector bundles.

Finally, we will often make use of the following result from equivariant homotopy theory.

\begin{prop}[\cite{Hesselholt_Nikolaus}, Proposition 3]\label{prop: Induced/coinduced action and fixed points} 

Let $G$ be a compact Lie group. Let $H\subseteq G$ be a closed subgroup, let $\lambda=T_H(G/H)$ be the tangent space at $H=eH$ with the adjoint left $H$-action, and let $S^\lambda$ be the one-point compactification of $\lambda$. For every spectrum with $G$-action $X$, there are canonical natural equivalences \[\left(X\wedge \left(G/H\right)_+\right)^{hG}\simeq \left(X\wedge S^\lambda\right)^{hH},\] \[\left(X\wedge \left(G/H\right)_+\right)^{tG}\simeq \left(X\wedge S^\lambda\right)^{tH}.\] 

\end{prop}

In particular for $G=\TT$ and $H=C_n$ we have for any $\TT$ spectrum $X$ natural euqivalences \[(X\wedge(\TT/C_n)_+)^{h\TT}\simeq \Sigma X^{hC_n}\] and \[(X\wedge (\TT/C_n)_+)^{t\TT}\simeq \Sigma X^{tC_n}\] which is the form of this statement we will use. 

\subsection{\texorpdfstring{F-dimension and $\mathcal{N}$-dimension}{F-dimension and N-dimension}}\label{ssec:f and n dimension}

 We first recall and introduce some definitions. We will make use of prismatic cohomology, the Hodge-Tate divisor, and the Nygaard filtration as constructed by Bhatt, Morrow, and Scholze in \cite{BMS1,BMS2, Bhatt_Scholze}. We include a brief review of the relevant constructions in Section~\ref{ssec: background info}. As in \cite{Bhatt_Scholze} we will abbreviate $\DD_S/I$ as $\overline{\DD}_S$

\begin{defn}[Definition 1.7, \cite{fsmooth_paper}]
Let $S$ be a $p$-complete ring. Then we say that $S$ is F-smooth if 
\begin{enumerate}
    \item the complexes $\DD_S\{i\}$ are Nygaard complete for all $i$; and
    \item the fiber of the map $\phi_i:\mathcal{N}^i\DD_S\to \overline{\DD}_S\{i\}$ has $p$-complete Tor amplitude in degrees $\geq i+2$.
\end{enumerate}
\end{defn}

\begin{ex}
Bouis studied the relative versions of this definition, which they call $p$-Cartier smooth maps in \cite{bouis2022cartier}. When the base ring is perfectoid these notions agree, and Bouis showed that $p$-completely smooth rings over a perfectoid ring and Cartier smooth $\mathbb{F}_p$-algebras are $p$-Cartier smooth, see \cite[Theorem C]{bouis2022cartier}. Bhatt and Mathew have also shown independently that any $p$-completely smooth ring over a perfectoid ring (\cite[Proposition 4.12]{fsmooth_paper}), any regular noetherian $p$-complete ring (\cite{fsmooth_paper}), and any Cartier smooth $\mathbb{F}_p$-algebra (\cite[Proposition 4.14]{fsmooth_paper}) is F-smooth. 
\end{ex}

Instead of taking $S$ to be a $p$-complete ring in the above definition we could instead only define F-smoothness for quasisyntomic rings. This is what is done in \cite[Definition 1.7]{fsmooth_paper}.

While strictly speaking we do not need the above definition in anything that follows, each of the following definitions are only well behaved on this class or rings. For example in what follows we would no longer have any relation between the F and $\mathcal{N}$ dimensions. We would also need to account for this in our Theorem statements. In addition all of the examples of interest where our results apply are F-smooth anyway.

\begin{defn}\label{defn: f-dim}
Let $S$ be an F-smooth ring in the sense of \cite[Definition 1.7]{fsmooth_paper}. We say that $S$ has F-dimension at most $j$, written as $\fdim(S)\leq j$, if the maps\[\phi_i:\mathcal{N}^i\DD_S\to \overline{\DD}_S\{i\}\] are equivalences for all $i\geq j$.
\end{defn}

\begin{ex}
    As a first example note that any perfectoid ring $R$ has $\fdim(R)=0$. In this case $\mathcal{N}^i\DD_R=d^iA_{inf}(R)/d^{i+1}A_{inf}(R)$, $\overline{\DD}_R=A_{inf}(R)/\Tilde{d}$, and the map $\phi_i$ is the map $\phi(-)/\Tilde{d}^i$. Here $d$ is an orientation of $R$.
\end{ex}

One common source of examples is the following which is a consequence of Lemma~\ref{lem: cohomology bounds stable with I}.

\begin{lem}
Any F-smooth ring $S$ with $\mathcal{N}^{\geq i}\DD_S\in D^{[0,j]}(\ZZ_p)$ for all $i\in \ZZ$ has $\fdim(S)\leq j$.
\end{lem}

\begin{ex}
Let $A$ be a complete DVR of mixed characteristic $(0,p)$ with perfect residue field. Then a special case of the first example is that $\fdim(A)\leq 1$. The ring $A$ satisfies the above conditions by \cite[Proposition 5.27]{antieau2021beilinson}, or for a purely algebraic proof see the proof of Lemma~\ref{lem: thh(A) purely algebraically}.
\end{ex}

\begin{ex}
Let $S$ be a $p$-completely smooth $R$-algebra, $R$ a perfectoid ring. Then \[\fdim(S)\leq \dim(S \textrm{ rel. }R)\] which can be seen by the first example together with the Hodge-Tate comparison Theorem \cite[Theorem 4.11]{Bhatt_Scholze}.
\end{ex}

Since the condition defining the above is important in its own right we also give it a name.

\begin{defn}\label{defn: n-dim}
Let $S$ be an F-smooth ring. Then we say that $S$ has $\ndim(S)\leq j$ if for all $i\in \ZZ$, $\mathcal{N}^{\geq i}\DD_S\in D^{[0,j]}(\ZZ_p)$. By the above we always have that $\fdim(S)\leq \ndim(S)$.
\end{defn}

\begin{ex}
    For $R$ a perfectoid ring we have that $\ndim(R)=0$.
\end{ex}

\begin{rem}
In Definition~\ref{defn: n-dim}, if we did not include the requirement that $S$ is F-smooth then we would have that $\ndim(S)=0$ for any quasiregular semiperfectoid ring. On the other hand, for any quaisregular semiperfectoid ring $S$ which is not a perfectoid ring, such as $S=\FF_p[x^{1/p^{\infty}}]/(x)$, we have that $\fdim(S)=\infty$.
\end{rem}

\section{A BMS style spectral sequence at finite levels}\label{sec: BMS spectral sequence}

From \parencite[Remark 3.5]{BMS2} we have that $\thh(-;\ZZ_p)^{hC_{p^n}}$ and $\thh(-;\ZZ_p)^{tC_{p^n}}$ are fpqc sheaves. It follows that they are also quasisyntomic sheaves. We will see in Section~\ref{sec: bms spectral sequence for truncated polynomial algebras} that the topological cyclic homology of truncated polynomials decomposes naturally into terms involving these sheaves and so it will be helpful to understand these objects as much as possible. In this section we redo the arguments in \cite{BMS2} used to produce the spectral sequences in \parencite[Theorem 1.12(4)]{BMS2} and identify the $E_2$ terms of these spectral sequences. We begin by identifying what these sheaves look like for $S$ a quasiregular semiperfectoid ring.

\begin{lem}\label{lem: finite levels for qrsp}
Let $S$ be a quasiregular semiperfectoid ring, and let $R$ be a perfectoid ring with a map $R\to S$. Let $d$ be an orientation of $R$. Then there are canonical (up to the choice of $R$ and $d$) equivalences \[\thh(S;\ZZ_p)^{hC_{p^n}}\cong \tc^{-}(S;\ZZ_p)/(\phi(d)\ldots \phi^n(d)v)\] and \[\thh(S;\ZZ_p)^{tC_{p^n}}\cong \tp(S;\ZZ_p)/(\phi(d)\ldots \phi^n(d))\] where $v\in \tc_{-2}^-(S;\ZZ_p)$ is the image of the generator of $\tc^{-}_{-2}(R;\ZZ_p)$ under the induced map. 
\end{lem}

The proof of this follows the proof of the end of \parencite[Proposition 6.4]{BMS2}, which itself follows the proof of \parencite[Lemma IV.4.12]{NS}. In order to apply results of this style, we first need to know what happens for Eilenberg-MacLane $\thh(R;\ZZ_p)$-modules.

\begin{lem}\label{lem: EM module case}
For some unit $r\in R^\times$, the square 
\[
\begin{tikzcd}
\Sigma^{-2}\tc^{-}(R;\ZZ_p) \arrow[rr, "\phi(d)\ldots\phi^n(d)v"] \arrow[d] &  & \tc^{-}(R;\ZZ_p) \arrow[d] \\
\Sigma^{-2}R^{h\TT} \arrow[rr, "rp^n v"]                                                &  & R^{h\TT}                  
\end{tikzcd}
\]
commutes, where the vertical arrows are induced by the truncation map $\thh(R;\ZZ_p)\to R$, and $v\in \pi_{-2}(R^{h\TT})\cong H^2(\CC P^\infty; R)$ is the first chern class of the tautological bundle.
\end{lem}
\begin{proof}
By definition the truncation map will send $v$ to $v$, so to check this it is enough to show that the map $A_{inf}(R)=\tc_0(R;\ZZ_p)\xrightarrow{\theta}R=\pi_0(R^{h\TT})$ send $\phi(d)\ldots \phi^n(d)\mapsto rp^n$. We will show something slightly stronger than the result, that in fact $\phi^n(d)=r_n p\mod d$ for all $n$, where $r_n\in A_{inf}(R)^\times$. We proceed by induction, noting that when $n=1$, $\phi(d)=d^p+p\delta(d)=p\delta(d)\mod d$ and $\delta(d)$ is already a unit in $A_{inf}(R)$.

Since $A_{inf}(R)$ is $(p,d)$-complete, $p,d\in \mathrm{rad}(A_{inf}(R))$ and since $\phi$ is an isomorphism it follows that $\phi^n(d)\in \mathrm{rad}(A_{inf}(R))$ for all $n\in \ZZ$. Consequently $\phi^n(d)$ is distinguished if and only if $p\in (\phi^n(d), \phi^{n+1}(d))$ by \parencite[Lemma 2.25]{Bhatt_Scholze}. From $d$ being distinguished we see that $p\in (d,\phi(d))$, but then $p=\phi^n(p)\in (\phi^n(d),\phi^{n+1}(d))$ and $\delta(\phi^n(d))$ is a unit.

 Now suppose that $\phi^n(d)=r_np \mod d$ for some $r_n\in A_{inf}(R)^\times$. It is then enough to show that $\phi^{n+1}(d)=pr_{n+1}\mod d$ for some $r_{n+1}\in A_{inf}(R)^\times$. To that end, note that \[\phi^{n+1}(d)=\phi(\phi^{n}(d))=(\phi^n(d))^p+p\delta(\phi^n(d))=p(r_n^pp^{p-1}+\delta(\phi^n(d))) \mod d\] and since $\delta(\phi^n(d))$ is a unit and $p\in \mathrm{rad}(A_{inf}(R))$ we have that $r_{n+1}:=\delta(\phi^n(d))+r_n^pp^{p-1}$ is a unit in $A_{inf}(R)$.
\end{proof}

From the fiber sequence $\TT\to BC_{p^n}\to B\TT$ and the resulting AH spectral sequence we can see that for a discrete $\ZZ$-module $M$, the map $M^{h\TT}\to M^{hC_{p^n}}$ identifies the target with the cofiber of $M^{h\TT}\xrightarrow{p^n v}M^{h\TT}$. 

\begin{proof}[Proof of Lemma~\ref{lem: finite levels for qrsp}]
Let $M$ be a $\TT$-equivariant $\thh(R;\ZZ_p)$ module. Then we will first show that the map $\thh(R;\ZZ_p)^{hC_{p^n}}\otimes_{\tc^{-}(R;\ZZ_p)}M^{h\TT}\to M^{hC_{p^n}}$ induced by the lax symmetric monoidal structure of homotopy $\TT$ fixed points is an equivalence. Since $\thh(R;\ZZ_p)^{hC_{p^n}}$ is a compact $\tc^{-}(R;\ZZ_p)$ module by \parencite[Corollary 5.8]{Riggenbach}, tensoring with it will commute with both filtered limits and colimits. Thus taking the Postnikov and Whitehead filtrations and using the fact the homotopy fixed points commute with both by \parencite[Lemma I.2.6]{NS}, we reduce to the case of $M$ Eilenberg-MacLane concentrated in degree zero.

In this the $\thh(R;\ZZ_p)$-module structure on $M$ comes from an $R$-module structure on $M$ via the map $\thh(R;\ZZ_p)\to R$. It follows that the $\tc^{-}(R;\ZZ_p)$-module structure on $M^{h\TT}$ comes from an $R^{h\TT}$-module structure via the map $\tc^{-}(R;\ZZ_p)\to R^{h\TT}$. Then by Lemma~\ref{lem: EM module case} it follows that the diagram 
\[
\begin{tikzcd}
\Sigma^{-2}\tc^{-}(R;\ZZ_p)\otimes_{\tc^{-}(R;\ZZ_p)}M^{h\TT} \arrow[d] \arrow[r, "\Tilde{d}_n v\otimes id_M"] & \tc^{-}(R;\ZZ_p)\otimes_{\tc^{-}(R;\ZZ_p)}M^{h\TT} \arrow[d] \\
\Sigma^{-2}R^{h\TT}\otimes_{R^{h\TT}}M^{h\TT} \arrow[r, "rp^nv\otimes id_M"]                                   & R^{h\TT}\otimes_{R^{h\TT}}M^{h\TT}               
\end{tikzcd}
\]
commutes. The vertical arrows are equivalences, so the cofibers are equivalent. The cofiber of the top is $\thh(R;\ZZ_p)^{hC_{p^n}}\otimes_{\tc^{-}(R;\ZZ_p)}M^{h\TT}$ and the cofiber of the bottom map is $M^{hC_{p^n}}$ and the maps in this diagram are the induced maps from the lax symmetric monoidal structure of homotopy $\TT$ fixed points, so the result follows.

 For the Tate construction, take $M=\thh(S;\ZZ_p)^{tC_p}$ with module structure given by the $\mathbb{E}_\infty$ ring map $\thh(R;\ZZ_p)\xrightarrow{\phi_p}\thh(R)^{tC_p}\to \thh(S)^{tC_p}$.
\end{proof}

In order to get the desired filtration, we will also want to know what the homotopy groups of these spectra are in terms of the prismatic cohomology of $S$. This almost follows immediately from Lemma~\ref{lem: finite levels for qrsp}, but we need to know that we do not introduce odd degree homotopy groups. For $\thh(S;\ZZ_p)^{tC_{p^n}}$ this follows from \parencite[Corollary 7.10(2)]{BMS2}, and for $\thh(S;\ZZ_p)^{hC_{p^n}}$ this follows from the statement for $\thh(S;\ZZ_p)^{tC_{p^n}}$ and the fact that $can: \tc^{-}(S;\ZZ_p)\to \tp(S;\ZZ_p)$ is injective on homotopy groups. This leads to the following corollary.

\begin{cor}
Let $S$, $R$, and $d$ be as in Lemma~\ref{lem: finite levels for qrsp}. Then $\tilde{d}_n:=\phi(d)\phi^2(d)\ldots \phi^{n}(d)$ is a non zero-divisor in the prismatic cohomology $\DD_S$ of $S$. Furthermore, there are isomorphisms \[\pi_{2i}\left(\thh(S;\ZZ_p)^{hC_{p^n}}\right)\cong \mathcal{N}^{\geq i}\widehat{\DD}_S\{i\}/\mathcal{N}^{\geq i+1}\widehat{\DD}_S\{i\}\tilde{d}_n \] and \[\pi_{2i}\left(\thh(S;\ZZ_p)^{tC_{p^n}}\right)\cong \widehat{\DD}_S\{i\}/\tilde{d}_n\] of $\widehat{\DD}_S$ modules, and the odd homotopy groups vanish.
\end{cor}
\begin{proof}
The only thing left to show is the first statement, i.e. that even before Nygaard completion $\tilde{d}_n$ is not a zero-divisor. Note that since $(\DD_S,d)$ is a prism, $\DD_S$ in particular will have no $d$-torsion. In addition, after inverting $d$ the Frobinius $\phi: \DD_S[\frac{1}{d}]\to \DD_S[\frac{1}{d}]$ is an isomorphism by \parencite[Theorem 1.8(6)]{Bhatt_Scholze} taking $A=A_{inf}(R)$. Thus $\tilde{d}_n$ is a non zero-divisor in $\DD_S[\frac{1}{d}]$, but then since $\DD_S\hookrightarrow \DD_S[\frac{1}{d}]$ is an injection $\tilde{d}_n$ is a non zero-divisor in $\DD_S$ as well.
\end{proof}

\begin{con}[The BMS filtration at finite levels]
We see from the homotopy group calculation that $\tau_{\geq 2i}\thh(-;\ZZ_p)^{hC_{p^n}}$ and $\tau_{\geq 2i}\thh(-;\ZZ_p)^{tC_{p^n}}$ are already quasisyntomic sheaves on $\mathrm{QRSPerfd}_R$. We already have that $\tau_{2j}\thh(-;\ZZ_p)$ is a sheaf on quasisyntomic rings and so combining this we have sheaves \[\mathrm{Fil}^{\geq i, \geq j}\thh(-;\ZZ_p)= \tau_{\geq 2i}(\tau_{\geq 2j}\thh(-;\ZZ_p))^{hC_{p^n}}: \mathrm{QRSPerfd}\to \Sp\] where we can remove the perfectoid ring $R$ since every map of quasisyntomic rings is an $R$-module map for some perfectoid ring $R$. Assembling this data into a single object gives a sheaf \[\mathrm{Fil}^{\geq *, \geq \ast}\thh(-;\ZZ_p)^{hC_{p^n}}: \mathrm{QRSPerfd}\to \widehat{\operatorname{DF}}(\widehat{\operatorname{DF}})\] where $\widehat{\mathrm{DF}}(\mathcal{C})$ is the derived completely filtered category of objects in $\mathcal{C}$ and $\widehat{\mathrm{DF}}:=\widehat{\mathrm{DF}}(\mathbb{S})$, see \cite[Section 5.1]{BMS2} for a review of the relevant facts about this construction. We will call the second filtration the Nygaard filtration by virtue of \cite[Theorem 1.12]{BMS2}, and since we will need to use it later we will use the notation \[\mathcal{N}^{\geq *}\thh(-;\ZZ_p)^{hC_{p^n}}:=\mathrm{Fil}^{\geq -\infty, \geq *}\thh(-;\ZZ_p)^{hC_{p^n}}\] and similarly for the Tate construction. Unfolding this then gives a sheaf \[\mathrm{Fil}^{\geq *, \geq *}\thh(-;\ZZ_p)^{hC_{p^n}}:\mathrm{QSyn}\to \widehat{\mathrm{DF}}(\widehat{\mathrm{DF}})\] and the above works with the Tate construction as well. Note that the sheaves $\mathcal{N}^{\geq *}\thh(-;\ZZ_p)^{hC_{p^n}}$ and $\mathcal{N}^{\geq *}\thh(-;\ZZ_p)^{tC_{p^n}}$ are given by unfolding the double speed Postnikov filtration on $\thh(-;\ZZ_p)$ and then applying the homotopy fixed points or Tate construction, respectively.
\end{con}

\begin{lem}\label{lem: Associated graded terms at finite levels}
In the filtration constructed above, we have that \[\mathrm{gr}^i\thh(-;\ZZ_p)^{hC_{p^n}}\simeq \mathcal{N}^{\geq i}\mathcal{O}_{\widehat{\DD}}\{i\}/\mathcal{N}^{\geq i+1}\mathcal{O}_{\widehat{\DD}}\{i\}\otimes \mathcal{I}_n[2i]\] and \[\mathrm{gr}^i\thh(-;\ZZ_p)^{tC_{p^n}}\simeq \mathcal{O}_{\widehat{\DD}}\{i\}/\mathcal{I}_n[2i]\] where we are taking the graded terms of the first filtration and the second filtration is induced from the Nygaard filtration. 
\end{lem}
\begin{proof}
We will show this for the homotopy fixed points case, the Tate construction case is similar. Note that by unfolding the map $\tc^{-}(-;\ZZ_p)\to \thh(-;\ZZ_p)^{hC_{p^n}}$ is a filtered map since on quasiregular semiperfectoid rings both filtrations are given by the Postnikov filtration which is functorial. In particular we get a map $\mathrm{gr}^i_{BMS}\tc^{-}(-;\ZZ_p)\to \mathrm{gr}^i\thh(-;\ZZ_p)^{hC_{p^n}}$ which gives a map of sheaves \[\mathcal{N}^{\geq i}\mathcal{O}_{\widehat{\DD}}\{i\}[2i]\to \mathrm{gr}^i\thh(-;\ZZ_p)^{hC_{p^n}}\] which is again a filtered map. Consider the composite map \[\begin{tikzcd}
{\mathcal{N}^{\geq i+1}\mathcal{O}_{\widehat{\DD}}\{i\}\otimes \mathcal{I}_n[2i]} \arrow[r] \arrow[rrd] & {\mathcal{N}^{\geq i}\mathcal{O}_{\widehat{\DD}}\{i\}\otimes \mathcal{I}_n[2i]} \arrow[r] & {\mathcal{N}^{\geq i}\mathcal{O}_{\widehat{\DD}}\{i\}[2i]} \arrow[d] \\
                                                                                    &                                                                       & \mathrm{gr}^i\thh(-;\ZZ_p)^{hC_{p^n}}                     
\end{tikzcd}\] which is contractible on quasiregular semiperfectoid rings by Lemma~\ref{lem: finite levels for qrsp}. By unfolding it must therefore be the zero map, and hence there exists a map \[\mathcal{N}^{\geq i}\mathcal{O}_{\widehat{\DD}}\{i\}/\mathcal{N}^{\geq i+1}\mathcal{O}_{\widehat{\DD}}\{i\}\otimes \mathcal{I}_n[2i]\to \mathrm{gr}^i\thh(-;\ZZ_p)^{hC_{p^n}}\] which is an equivalence on quasiregular semiperfectoid rings. Unfolding then once again shows that this is an equivalence globally as desired.
\end{proof}

\begin{ex}
Let $R$ be a perfectoid ring. Then the outside filtration on $\thh(R;\ZZ_p)^{hC_{p^n}}$ is the double speed Postnikov filtration and \[\mathrm{gr}^i(\thh(R;\ZZ_p)^{hC_{p^n}})\simeq \pi_{2i}(\thh(R;\ZZ_p)^{hC_{p^n}})[2i]\simeq d^{\max\{i,0\}}A_{inf}(R)/d^{\max\{i+1,0\}}\tilde{d}_{n}A[2i].\] This identification uses the fact that the Nygaard filtration can be identified with the $d$-adic filtration on $A_{inf}(R)$. The filtration on the graded terms is then given as follows: \[\mathrm{Fil}^{\geq j}\mathrm{gr}^i(\thh(R;\ZZ_p)^{hC_{p^n}})\simeq d^{\max\{i, j ,0\}}A_{inf}(R)/d^{\max\{i+1, j+1 ,0\}}\tilde{d}_{n}A_{inf}(R)[2i] \] in other words it is the $d$-adic filtration starting in filtration degree $\max\{i, 0\}$. Similarly we have that \[\mathrm{Fil}^{\geq j}\mathrm{gr}^i(\thh(R;\ZZ_p)^{tC_{p^n}})\simeq d^{\max\{j,0\}}A_{inf}(R)/d^{\max\{j,0\}}\tilde{d}_nA_{inf}(R)[2i]\] in other words the $d$-adic filtration starting in filtration degree $0$.
\end{ex}

To get this to work for all $p$-complete commutative rings, we need to show that the filtration we have constructed is in fact already Kan extended from $p$-completed finitely generated polynomial algebras. This is the strategy taken in \parencite[Section 5]{antieau2021beilinson} to extend the BMS filtration to all $p$-complete rings, and we will follow their strategy here. 

\begin{lem}\label{lem: Left Kan Extended}
For all and all $n\in \NN\cup\{\infty\}$ the functors \[\mathrm{Fil}^{\geq *, \geq *}\thh(-;\ZZ_p)^{hC_{p^n}}:\mathrm{QSyn}\to \widehat{\mathrm{DF}}(\widehat{\mathrm{DF}})\] and \[\mathrm{Fil}^{\geq *, \geq *}\thh(-;\ZZ_p)^{tC_{p^n}}:\mathrm{QSyn}\to \widehat{\mathrm{DF}}(\widehat{\mathrm{DF}})\] are left Kan extended from $p$-completed finitely generated polynomial $\ZZ_p$-algebras.
\end{lem}
\begin{proof}
In order to show that these functors are left Kan extended, it is enough to show that the underlying functors \[\mathcal{N}^{\geq *}\thh(-;\ZZ_p)^{hC_{p^n}}:\mathrm{QSyn}\to \widehat{\mathrm{DF}}\] and \[\mathcal{N}^{\geq *}\thh(-;\ZZ_p))^{tC_{p^n}}:\mathrm{QSyn}\to \widehat{\mathrm{DF}}\] are left Kan extended and that the associated graded terms on the outside filtration are left Kan extended. 

For the first assertion, by Proposition~\ref{prop: Induced/coinduced action and fixed points} we have an equivalence of sheaves \[\mathcal{N}^{\geq *}\thh(-;\ZZ_p)^{hC_{p^n}}\simeq (\mathrm{Fil}^{\geq *}\thh(-;\ZZ_p) \wedge (\TT/C_{p^n})_+)^{h\TT}\] where the filtration on $\thh(-;\ZZ_p)$ on the right is the unfolding of the double speed Postnikov filtration. Notice that we have that $\mathrm{Fil}^{\geq *}\thh(-;\ZZ_p)\wedge (\TT/C_{p^n})_+\to \thh(-;\ZZ_p)\wedge (\TT/C_{p^n})_+$ is an equivalence in degrees $\geq 2*+1$. To see this note that $\tau_{\geq 2*}:\Sp\to \Sp_{\geq 2*}$ is right adjoint to the inclusion functor $\Sp_{\geq 2*}\to \Sp$ where $\Sp_{\geq 2*}$ is the full subcategory of $\Sp$ spanned by $(2*)$-connective objects. It the follows that for any quasisyntomic ring $A$ and any quasisyntomic cover of $A$ by a quasiregular semiperfectoid ring $A\to S$, \begin{align*}
\tau_{\geq 2*}\mathrm{Fil}^{\geq *}\thh(A;\ZZ_p) &\simeq \tau_{\geq 2*}(\lim_{[n]\in\Delta}(\tau_{\geq 2*}\thh(S^{\widehat{\otimes}_{A}\bullet+1};\ZZ_p)))\\
&\simeq \lim_{[n]\in \Delta}(\tau_{\geq 2*}\tau_{\geq 2*}\thh(S^{\widehat{\otimes}_A n+1};\ZZ_p))\\
&\simeq \lim_{[n]\in \Delta}(\tau_{\geq 2*}\thh(S^{\widehat{\otimes}_A n+1};\ZZ_p))\\
&\simeq \tau_{\geq 2*}(\lim_{[n]\in \Delta}\thh(S^{\widehat{\otimes}_A n+1};\ZZ_p))\\
&\simeq \tau_{\geq 2*}\thh(A;\ZZ_p)
\end{align*} where the limits in the second and third line are taken in the category $\Sp_{\geq 2*}$. After smashing with $(\TT/C_{p^n})_+$, which nonequivariantly is just $S^0\vee S^1$, the connectivity bounds follow.

From \parencite[Corollary 5.21]{antieau2021beilinson} we have equivalences \[\thh(S;\ZZ_p)\simeq \operatorname{colim}_{P\in \widehat{\mathrm{PFG}}_{/S}}\thh(P;\ZZ_p)\] and \[\mathrm{Fil}^{\geq *}\thh(S;\ZZ_p)\simeq \operatorname{colim}_{P\in \widehat{\mathrm{PFG}}_{/S}}\mathrm{Fil}^{\geq *}\thh(P;\ZZ_p)\] where $S\in \mathrm{QSyn}$ and $\widehat{\mathrm{PFG}}$ is the category of $p$-completed finitely generated polynomial $\ZZ_p$ algebras. For any $S\in\mathrm{QSyn}$ we then have that 
\begin{align*}
    \thh(S;\ZZ_p)^{hC_{p^n}} &\simeq \lim \thh(S;\ZZ_p)^{hC_{p^n}}/\mathcal{N}^{\geq *}\thh(S;\ZZ_p)^{hC_{p^n}}\\
                            &\simeq \lim_{*\to \infty} \Sigma^{-1}\left((\thh(S;\ZZ_p)/\mathrm{Fil}^{\geq *}\thh(S;\ZZ_p)) \wedge (\TT/C_{p^n})_+\right)^{h\TT}\\
                            &\simeq \lim_{*\to \infty} \Sigma^{-1}\left((\colim_{P\in \widehat{\mathrm{PFG}}_{/S}}\thh(P;\ZZ_p)/\mathrm{Fil}^{\geq *}\thh(P;\ZZ_p)) \wedge (\TT/C_{p^n})_+\right)^{h\TT}\\
                            &\simeq \lim_{*\to \infty} \colim\limits_{P\in \widehat{\mathrm{PFG}}_{/S}} \Sigma^{-1}\left((\thh(P;\ZZ_p)/\mathrm{Fil}^{\geq *}\thh(P;\ZZ_p)) \wedge (\TT/C_{p^n})_+\right)^{h\TT}\\
                            &\simeq \lim_{*\to \infty} \colim\limits_{P\in \widehat{\mathrm{PFG}}_{/S}} \thh(P;\ZZ_p)^{hC_{p^n}}/\mathcal{N}^{\geq *}\thh(P;\ZZ_p)^{hC_{p^n}}
\end{align*}
where the colimit and homotopy fixed points commute since the spectra on the inside are bounded and $B\TT$ is simply connected and of finite type, see \cite[Lemma 4.2]{Cusps-Paper}. The last term in this line is the colimit in $\widehat{\mathrm{DF}}$ by the identification of completion as the left adjoint of the inclusion functor $\widehat{\operatorname{DF}}\subseteq \operatorname{DF}$ in \cite[Lemma 5.2(2)]{BMS2}, and so we have that $ \mathcal{N}^{\geq *}\thh(-;\ZZ_p)^{hC_{p^n}}$ is Kan extended from $p$-completed finitely generated polynomial $\ZZ_p$ algebras as desired. The argument for the Tate construction is verbatim. 

It remains to show that the associated graded terms of the outside filtration are also Kan extended. By the discussion above these are given by terms of the form \[(\mathcal{N}^{\geq i}\mathcal{O}_{\widehat{\DD}}\{i\}/(\mathcal{N}^{\geq i+1}\mathcal{O}_{\widehat{\DD}}\otimes \mathcal{I}_n\{i\})[2i], \mathcal{N}^{\geq \max\{*, i\}}\mathcal{O}_{\widehat{\DD}}\{i\}/(\mathcal{N}^{\geq \max\{*+1, i+1\}}\mathcal{O}_{\widehat{\DD}}\otimes \mathcal{I}_n\{i\})[2i])\] as an element of $\operatorname{Shv}_{\widehat{\operatorname{DF}(\ZZ_p)}}(\mathrm{QSyn})$. By devissage we reduce to showing that $\mathcal{N}^j\mathcal{O}_{\widehat{\DD}}\{i\}$ and $\mathcal{N}^{j+1}\mathcal{O}_{\widehat{\DD}}\{i\}\otimes \mathcal{I}_n$ are Kan extended. In this case the Breuil-Kisin twists are canonically trivializable and so the first case follows from \parencite[Corollary 5.21]{antieau2021beilinson}. In particular by devissage $\mathcal{N}^{\geq *}\mathcal{O}_{\widehat{\DD}}$ is Kan extended from $p$-complete finitely generated polynomial $\ZZ_p$-algebras and so tensoring with an $\mathcal{O}_{\widehat{\DD}}$ vector bundle will not change if something is Kan extended from the same subcategory. The result follows, and the Tate case follows from the above since the associated graded terms of the outside filtration on the Tate construction are the sheaves $\mathcal{O}_{\widehat{\DD}}\{i\}/\mathcal{I}_n\{i\}[2i]$ again with its Nygaard filtration.
\end{proof}

We may now further Kan extend this result to all simplicial commutative $p$-complete rings. 

\begin{cor}
Let $\mathrm{SCR}_p$ be the category of simplicial commutative $p$-complete rings. Then the functors $\thh(-;\ZZ_p)^{hC_{p^n}},\thh(-;\ZZ_p)^{tC_{p^n}}:\mathrm{SCR}_p\to \Sp$ lift to functors \[\mathrm{Fil}^{\geq *, \geq *}\thh(-;\ZZ_p)^{hC_{p^n}}:\mathrm{SCR}_p\to \widehat{\mathrm{DF}}(\widehat{\mathrm{DF}})\] and \[\mathrm{Fil}^{\geq *, \geq *}\thh(-;\ZZ_p)^{tC_{p^n}}:\mathrm{SCR}_p\to \widehat{\mathrm{DF}}(\widehat{\mathrm{DF}})\] where the individual filtration pieces commute with sifted colimits as functors to $\widehat{\mathrm{DF}}$, and the internal filtration is the Nygaard filtration.
\end{cor}
\begin{proof}
This extension is given by Kan extending the filtrations produced in this section to all of $\mathrm{SCR}_p$, and the desired properties follow trivially. 
\end{proof}

\begin{cor}
Let $S\in \mathrm{SCR}_p$. Then there are functorial conditionally convergent multiplicative spectral sequences 
\[E^{s,t}_2 = H^{s-t}(\mathcal{N}^{\geq -t}\widehat{\DD}_S\{-t\}/\mathcal{N}^{\geq 1-t}\widehat{\DD}_S\{-t\}\otimes I_n)\implies \pi_{s+t}(\thh(S;\ZZ_p)^{hC_{p^n}}) \] and \[E^{s,t}_2 = H^{s-t}(\widehat{\DD}_S\{-t\}/I_n)\implies \pi_{s+t}(\thh(S;\ZZ_p)^{tC_{p^n}})\] which comes from a complete filtration which for $S\in \mathrm{QSyn}$ is exhaustive.
\end{cor}

\section{Two auxiliary computations}\label{sec: two auxiliary computations}
Let $\Pi_e$ denote the pointed monoid $\Pi_e = \{0, 1, x ,\ldots, x^{e-1}\}$ under multiplication with the relation that $x^e=0$. We then have that for any $\mathbb{E}_\infty$-ring spectrum $A$ that $A[x]/x^e\cong A\wedge \mathbb{S}[\Pi_e]$ and so to understand $\thh(-[x]/x^e)$ one approach is to first understand $\thh(\mathbb{S}[\Pi_e])\simeq \Sigma^{\infty}B^{cy}(\Pi_e)$.

From the description of $B^{cy}(\Pi_e)$ by Hesselholt, which we record in this paper as Theorem~\ref{thm: calculation of the Bn}, in order to compute the topological negative cyclic and periodic homologies of the truncated polynomial ring we will need to understand two auxiliary constructions. One is the map \[M^{hC_{p^k}}\to M^{hC_{p^{k+e}}}\] induced by the map $\TT/C_{p^k}\to \TT/C_{p^{k+e}}$. We will be able to do this fairly explicitly when $M$ is a $\TT$-equivariant $\thh(R;\ZZ_p)$-module, $R$ a perfectoid ring. This then extends to a description over any quasisyntomic ring by unfolding. We will also need to understand how representation spheres interact with these constructions, which will follow a similar story.

\subsection{\texorpdfstring{The cofiber of the $V_e$ map}{The cofiber of the V map}}

\begin{lem}\label{lem: coker of Verschiebung qrsp case.}
Let $M$ be a $\TT$-equivariant $\thh(R;\ZZ_p)$-module. Then the map \[M^{hC_{p^k}}\xrightarrow{V_{e,k}} M^{hC_{p^{e+k}}}\] induced by the map $\TT/C_{p^k}\to \TT/C_{p^{e+k}}$ fits into the commutative diagram \[
\begin{tikzcd}
M^{h\TT} \arrow[r, "-\times r\tfrac{\Tilde{d}_{e+k}}{\Tilde{d}_{k}}"] \arrow[d] & M^{h\TT} \arrow[d]\\
M^{hC_{p^k}} \arrow[r,"V_{e,k}"] & M^{hC_{p^{e+k}}}
\end{tikzcd}
\]
where the top map is induced by the $\tc^{-}(R;\ZZ_p)$-module structure. The element $r\in \tc^{-}_0(R;\ZZ_p)$ is a unit.
\end{lem}
\begin{proof}
We will first show this in the case of $M=\thh(R;\ZZ_p)$. In this case we have that $V_{e,k}$ induces the $e$ Verschiebung map on homotopy groups by \cite[Lemma 5.13]{Riggenbach}. From \cite[Lemma 3.4]{BMS1} and \cite[Corollary 5.8]{Riggenbach} we may then write this map as \[A_{inf}(R)[u,v]/(uv-d,\Tilde{d}_{k}v)\xrightarrow{\times \Tilde{\lambda}_{k+1}\Tilde{\lambda}_{k+2}\ldots \Tilde{\lambda}_{k+e}}A_{inf}[u,v]/(uv-d,\Tilde{d}_{k+e}v)\] where $\Tilde{\lambda}_i$ is an element such that $\Tilde{\theta}_i(\Tilde{\lambda_i})=V([1])\in W_{i}(R)$. In particular by \cite[Remark 3.11]{BMS1} each $\Tilde{\lambda}_i$ is distinguished for all $i$, and so up to a unit are the appropriate twist of our orientation $\phi^i(d)$. In this case the map above becomes 
\[A_{inf}(R)[u,v]/(uv-d, \Tilde{d}_kv)\xrightarrow{\times r\tfrac{\Tilde{d}_{e+k}}{\Tilde{d}_k}}A_{inf}(R)[u,v]/(uv-d,\Tilde{d}_{e+k}v))\] for some unit $r$ as desired.

Now, this map is a map of $\tc^{-}(R;\ZZ_p)$-modules, and both sides are finite as $\tc^{-}(R;\ZZ_p)$-modules. Consequently this map being as described on homotopy groups implies that it is as desired on the level of spectra. All that remains is to show that this remains true for any $\TT$-equivariant $\thh(R;\ZZ_p)$-module $M$. 

To this end, recall that in the proof of Lemma~\ref{lem: finite levels for qrsp} we actually showed that the map \[M^{h\TT}\otimes_{\tc^{-}(R;\ZZ_p)}\thh(R;\ZZ_p)^{hC_{p^k}}\to M^{hC_{p^k}}\] is an equivalence. Note that this map can be factored through \[M^{h\TT}\otimes_{\thh(R;\ZZ_p)^{h\TT}}\left(\Sigma^{-1}\thh(R;\ZZ_p)\wedge \left(\TT/C_{k}\right)_+\right)^{h\TT}\to \left(M\otimes_{\thh(R;\ZZ_p)}\Sigma^{-1}\thh(R;\ZZ_p)\wedge\left(\TT/C_k\right)_+\right)^{h\TT}\] by using Proposition~\ref{prop: Induced/coinduced action and fixed points}. From this we see that the square \[
\begin{tikzcd}
M^{h\TT}\otimes_{\tc^{-}(R;\ZZ_p)}\thh(R;\ZZ_p)^{hC_{p^k}}\arrow[d]\arrow[r, "id\otimes V_{e,k}"] & M^{h\TT}\otimes_{\tc^{-}(R;\ZZ_p)}\thh(R;\ZZ_p)^{hC_{p^{k+e}}}\arrow[d]\\
M^{hC_{p^k}}\arrow[r, "V_{e,k}"] & M^{hC_{p^{k+e}}}
\end{tikzcd}
\]
commutes. By tensoring the commutative square for $\thh(R;\ZZ_p)$ with $M^{h\TT}$ our result follows. 
\end{proof}

We are now able to compute the cofiber of the map $V_{e,k}$ using the above result.

\begin{lem}\label{lem: cofiber of virshebung general case}
Let $\mathrm{cofib}(V_{e,k,hfp})(-)$ denote the quasisyntomic sheaf given by \[\mathrm{cofib}(V_{e, k, hfp})(S)=\mathrm{cofib}(V_{e,k}:\thh(S;\ZZ_p)^{hC_{p^k}}\to \thh(S;\ZZ_p)^{hC_{p^{k+e}}})\] and $\mathrm{cofib}(V_{e,k,Tate})(-)$ the cofiber of the map on the Tate construction. Then there are isomorphisms \[\mathrm{gr}^i\left(\mathrm{cofib}(V_{e,k,hfp})\right)\simeq \mathcal{N}^{\geq i}\mathcal{O}_{\widehat{\DD}}/(\phi^{k})^*(\mathcal{I}_e)\{i\}[2i]\] and \[\mathrm{gr}^i\left(\mathrm{cofib}(V_{e,k,Tate})\right)=\mathcal{O}_{\widehat{\DD}}\{i\}/(\phi^k)^*(\mathcal{I}_e)[2i]\] where the filtration is the Postnikov filtration of quasisyntomic sheaves.  
\end{lem}
\begin{proof}
By Theorem~\ref{thm: qrsp are basis} it is enough to identify these sheaves on quasiregular semiperfectoid rings. On quasiregular semiperfectoid rings the diagram 
\[
\begin{tikzcd}
\mathcal{N}^{\geq i}\mathcal{O}_{\widehat{\DD}}\otimes (\phi^k)^*(\mathcal{I}_e)\{i\} \arrow[r] & \mathcal{N}^{\geq i}\mathcal{O}_{\widehat{\DD}}\{i\} \arrow[d]& \\
\mathcal{N}^{\geq i}\mathcal{O}_{\widehat{\DD}}\{i\}/(\mathcal{N}^{\geq i+1}\mathcal{O}_{\widehat{\DD}}\otimes \mathcal{I}_k) \arrow[r,"gr^i(V_e)"] & \mathcal{N}^{\geq i}\mathcal{O}_{\widehat{\DD}}\{i\}/(\mathcal{N}^{\geq i+1}\mathcal{O}_{\widehat{\DD}}\otimes \mathcal{I}_{k+e}) \arrow[r] & gr^i(\mathrm{cofib}(V_{e,k,hfp}))[-2i]
\end{tikzcd}
\] gives a map of sheaves $\mathcal{N}^{\geq i}\mathcal{O}_{\widehat{\DD}}\{i\}\to gr^i(\mathrm{cofib}(V_{e,k,hfp}))[-2i]$. Precomposing with the map \[\mathcal{N}^{\geq i}\mathcal{O}_{\widehat{\DD}}\otimes (\phi^k)^*(\mathcal{I}_e)\{i\}\to \mathcal{N}^{\geq i}\mathcal{O}_{\widehat{\DD}}\{i\}\] will then give a nullhomotopic map by Lemma~\ref{lem: coker of Verschiebung qrsp case.}. In particular there is a map of sheaves \[\mathcal{N}^{\geq i}\mathcal{O}_{\widehat{\DD}}\{i\}/(\phi^k)^*(\mathcal{I}_e)\to gr^i(\mathrm{cofib}(V_{e,k,hfp}))\] which we will show is an equivalence.

To see this, let $S$ be a quasiregular semiperfectoid ring with $R\to S$ a map, $R$ a perfectoid ring with orientation $d$. Note that on the level of spectra the commutative diagram giving the map in question is
\[\begin{tikzcd}
\tc^{-}(S;\ZZ_p) \arrow[d] \arrow[r, "-\times r\phi^k(\tilde{d}_e)"] & \tc^{-}(S;\ZZ_p) \arrow[d]   \\
\thh(S;\ZZ_p)^{hC_{p^k}} \arrow[r, "V_e"]                            & \thh(S;\ZZ_p)^{hC_{p^{k+e}}}
\end{tikzcd}\] and our goal is to identify the horizontal cofibers, in other words to show that this square is pullback. Taking horizontal fibers then gives the commutative diagram
\[
\begin{tikzcd}
\Sigma^{-2}\tc^{-}(S;\ZZ_p) \arrow[d, "v\tilde{d}_k"] \arrow[r] & \Sigma^{-2}\tc^{-}(S;\ZZ_p) \arrow[d, "v\tilde{d}_{k+e}"]\\
\tc^{-}(S;\ZZ_p) \arrow[r, "r\phi^k(\tilde{d}_e)"] & \tc^{-}(S;\ZZ_p)
\end{tikzcd}
\]
from which we see that the top map must be multiplication by a unit. In particular the top map is an equivalence as desired and so the original square is pullback as desired.
\end{proof}

\subsection{Representation spheres}

In this Subsection we will study the sheaves $\left(\Sigma^{\lambda_k-2k}\thh(-;\ZZ_p)\right)^{hC_{p^n}}$ in the same way as in the previous section. Here $\lambda_k:=\CC(1)\oplus \CC(2)\oplus \ldots \oplus \CC(k)$ where $\CC(l)$ is the $\TT$-representation $\CC$ where $\theta$ acts by multiplication by $e^{il\theta}$. By \cite[Corollary 4.4]{Riggenbach} on a quasiregular semiperfectoid ring $S$ we have that \[\left(\Sigma^{\lambda_k-2k}\thh(S;\ZZ_p)\right)^{hC_{p^n}}\simeq \thh(S;\ZZ_p)^{hC_{p^n}}\] by a non-canonical isomorphism. Consequently $\tau_{\geq *}\left(\Sigma^{\lambda_k-2k}\thh(-;\ZZ_p)\right)^{hC_{p^n}}$ will be a sheaf on $\mathrm{QRSPerfd}$ and we still have that \[\left((\Sigma^{\lambda_k-2k}\thh(-;\ZZ_p))^{hC_{p^n}}, \tau_{\geq 2*}(\Sigma^{\lambda_k-2k}\thh(-;\ZZ_p))^{hC_{p^n}}\right):\mathrm{QRSPerfd}\to \widehat{\mathrm{DF}}(\widehat{\mathrm{DF}})\] is a sheaf. Unfolding this construction then gives a sheaf \[\left((\Sigma^{\lambda_k-2k}\thh(-;\ZZ_p))^{hC_{p^n}}, \mathrm{Fil}^{\geq *}(\Sigma^{\lambda_k-2k}\thh(-;\ZZ_p))^{hC_{p^n}}\right):\mathrm{QSyn}\to \widehat{\mathrm{DF}}(\widehat{\mathrm{DF}})\]

For the Tate construction to be a sheaf it remains to check that the homotopy orbits are still a sheaf after smashing with a representation sphere. To see this, first note that $S^{\lambda_k-2k}\simeq \bigwedge_{l=1}^k S^{\CC(l)-2}$. Then there is a cofiber sequence of pointed $\TT$-spaces \[S^0\hookrightarrow S^{\CC(k+1)}\to \Sigma (\TT/C_{k+1})_+\] where the first map is the inclusion of the points at $0$ and $\infty$. Applying $\Sigma^\infty$ to this sequence, which is left adjoint and so will preserve cofiber sequences, there is then a cofiber sequence of Borel $\TT$-spectra \[\mathbb{S} \to \mathbb{S}^{\CC(k+1)}\to \Sigma^{\infty +1}_+ \TT/C_{k+1}\] and then a fiber sequence 
\begin{equation}\label{eqn: fiber sequence for C(l)-2}
    \mathbb{S}^{-2}\to \mathbb{S}^{\CC(k+1)-2}\to \Sigma^{-1}_+ \TT/C_{k+1}
\end{equation}
after tensoring the above cofiber sequence with $\mathbb{S}^{-2}$. We may then tensor this fiber sequence with the functor  $\Sigma^{\lambda_{k}-2k}\thh(-;\ZZ_p)$ to get an equivariant cofiber sequence 
\[\Sigma^{-2}(\Sigma^{\lambda_{k}-2k}\thh(-;\ZZ_p))\to \Sigma^{\lambda_{k+1}-2(k+1)}\thh(-;\ZZ_p)\to \Sigma^{-1}(\Sigma^{\lambda_{k}-2k}\thh(-;\ZZ_p)\wedge (\TT/C_{k+1})_+)\] and so upon taking homotopy orbits a cofiber sequence \[\Sigma^{-2}\left(\Sigma^{\lambda_k-2k}\thh(-;\ZZ_p) \right)_{h\TT}\to \left( \Sigma^{\lambda_k-2k}\thh(-;\ZZ_p) \right)_{h\TT}\to \Sigma^{-1}\left(\Sigma^{\lambda_k-2k}\thh(-;\ZZ_p) \right)_{hC_{k+1}}\] and a similar cofiber sequence for the $C_{k+1}$ and $C_{p^n}$ homotopy orbits instead of the $\TT$ homotopy orbits. The identification of the last term comes from Proposition~\ref{prop: Induced/coinduced action and fixed points}. For the finite cyclic groups, note that there are cofiber sequences of $C_m$ space \[(C_m/C_{(l,m)})_+\to (\TT/C_l)_+ \to \Sigma_+ C_m/C_{(l,m)}\] for all $l,m\in \ZZ_+$. Thus we may express the homotopy $C_m$ orbits of a spectrum $X\wedge \TT/C_l$ in terms of the homotopy $C_{(l,m)}$ orbits of $X$. Hence by an inductive argument we have that the homotopy orbits are still a sheaf after smashing with $S^{\lambda_k-2k}$. By a similar argument on the Tate construction we then have a sheaf 
\[\left((\Sigma^{\lambda_k-2k}\thh(-;\ZZ_p))^{tC_{p^n}}, \mathrm{Fil}^{\geq *}(\Sigma^{\lambda_k-2k}\thh(-;\ZZ_p))^{tC_{p^n}}\right):\mathrm{QSyn}\to \widehat{\mathrm{DF}}(\widehat{\mathrm{DF}})\] where the outer filtration is the unfolding of the double speed Postnikov filtration after taking the Tate construction and the internal filtration is the Nygaard filtration. 

We will spend the rest of this section identifying the associated graded pieces of these terms. To this end we will first use the following lemma to reduce $\lambda_k$ to its individual summands.

\begin{lem}
There is an isomorphism of quasisyntonic sheaves 
\[\mathrm{gr}^0\left(\Sigma^{\CC(k+1)-2}\thh(-;\ZZ_p)\right)^{h\TT}\otimes_{\mathcal{O}_{\widehat{\DD}}}\mathrm{gr}^i\left(\Sigma^{\lambda_k-2k}\thh(-;\ZZ_p)\right)^{hC_{p^n}}\to \mathrm{gr}^i\left(\Sigma^{\lambda_{k+1}-2k-2}\thh(-;\ZZ_p)\right)^{hC_{p^n}}\] and similarly for the Tate construction.
\end{lem}
\begin{proof}
This map exists for quasiregular semiperfectoid rings by the lax symmetric monoidal structure of $(-)^{h\TT}$ applied to the decomposition \[\Sigma^{\lambda_{k+1}-2k-3}\thh(-;\ZZ_p)\wedge(\TT/C_{p^n})_+\simeq \left(\Sigma^{\CC(k+1)-2}\thh(-;\ZZ_p)\right)\otimes_{\thh(-;\ZZ_p)}\left(\Sigma^{\lambda_{k}-2k-1}\thh(-;\ZZ_p)\wedge (\TT/C_{p^n})_+\right)\] as $\TT$-equivariant $\thh(-;\ZZ_p)$-modules. For quasiregular semiperfectoid rings the associated graded terms are all the homotopy groups since the filtration is the double speed Postnikov filtration and the odd homotopy groups vanish. 

Hence we get a map \[\mathrm{gr}^0\left(\Sigma^{\CC(k+1)-2}\thh(-;\ZZ_p)\right)^{h\TT}\otimes_{\mathcal{O}_{\widehat{\DD}}}\mathrm{gr}^i\left(\Sigma^{\lambda_k-2k}\thh(-;\ZZ_p)\right)^{hC_{p^n}}\to \mathrm{gr}^i\left(\Sigma^{\lambda_{k+1}-2(k+1)}\thh(-;\ZZ_p)\right)^{hC_{p^n}}\] as sheaves on $\mathrm{QRSPerfd}$ which is given by the composition of the canonical map \[\begin{tikzcd}\pi_0\left(\Sigma^{\CC(k+1)-2}\thh(-;\ZZ_p)\right)^{h\TT}\otimes_{\tc^{-}_0(-;\ZZ_p)}\pi_i\left(\Sigma^{\lambda_{k}-2k-1}\thh(-;\ZZ_p)\wedge (\TT/C_{p^n})_+\right)^{h\TT}\arrow[d]\\ \pi_i\left(\left(\Sigma^{\CC(k+1)-2}\thh(-;\ZZ_p)\right)^{h\TT}\otimes_{\tc^-(-;\ZZ_p)}\left(\Sigma^{\lambda_{k}-2k-1}\thh(-;\ZZ_p)\wedge (\TT/C_{p^n})_+\right)^{h\TT}\right)\end{tikzcd}\] with the lax monoidal structure map for homotopy fixed points. This map is an equivalence of $\mathrm{QRSPerfd}$ sheaves since we can check this levelwise and on a given quasiregular semiperfectoid ring we may trivialize the representation spheres even on the level of Borel $\TT$-spectra by \cite[Corollary 4.4]{Riggenbach}. 

Since this is an equivalence on $\mathrm{QRSPerfd}$ by unfolding we get an equivalence globally.
\end{proof}

By the above Lemma we need only compute the associated graded terms for the representation spheres $S^{\CC(l)-2}$. These will be given by quasisyntomic line bundles, and it will turn out to be easier to identify their inverses. In order to do this we will first find a topological description of their inverses.

\begin{lem}
There is an isomorphism of quasisyntomic sheaves \[\mathrm{gr}^0\left(\Sigma^{\CC(k)-2}\thh(-;\ZZ_p)\right)^{h\TT}\otimes\mathrm{gr}^0\left(\Sigma^{2-\CC(k)}\thh(-;\ZZ_p)\right)^{h\TT}\to \mathcal{O}_{\widehat{\DD}}\] and the same is true for the Tate construction.
\end{lem}
\begin{proof}
This is true for quasiregular semiperfectoid rings since $S^{\CC(k)-2}\wedge S^{2-\CC(k)}\simeq S^0$ canonically. Hence by unfolding it is globally true.
\end{proof}

We consequently have an equivalence \[\mathrm{gr}^0\left(\Sigma^{\CC(k)-2}\thh(-;\ZZ_p)\right)^{h\TT}\simeq \left(\mathrm{gr}^0\left(\Sigma^{2-\CC(k)}\thh(-;\ZZ_p)\right)^{h\TT}\right)^{-1}\] and both are line bundles as $\mathcal{O}_{\widehat{\DD}}$-modules. 

\begin{lem}
There is an equivalence of sheaves \[\mathrm{gr}^0\left(\Sigma^{2-\CC(k)}\thh(-;\ZZ_p)\right)^{h\TT}\simeq \mathcal{I}_{v_p(k)}\{-1\}\] and the same statement is true for the Tate construction. Here $v_p(k)$ is the $p$-adic valuation of $k$.
\end{lem}

\begin{proof}
Twisting the fiber sequence~\ref{eqn: fiber sequence for C(l)-2} by $S^{4-\CC(k)}$ gives the fiber sequence \[S^{2-\CC(k)}\to S^2\to \Sigma^1_+\TT/C_k\] where we are using the fact that $S^{\CC(k)}\wedge \left(\TT/C_{k}\right)_+\simeq \Sigma^2_+ \TT/C_k$ as Borel equivariant $\TT$-spectra. To see this note that we have the chain of equivalences \[S^{\CC(k)}\wedge (\TT/C_k)_+\simeq (S^{\CC(k)}\wedge S^0)\wedge_{C_k}\TT_+\simeq S^2\wedge_{C_k}\TT_+\] since the $C_k$-action on $\CC(k)$ is trivial.

We therefore have a fiber sequence \[\left(\Sigma^{2-\CC(k)}\thh(-;\ZZ_p)\right)^{h\TT}\to \Sigma^2\tc^{-}(-;\ZZ_p)\to \Sigma^2\thh(-;\ZZ_p)^{hC_{p^{v_p(k)}}}\] of quasisyntomic sheaves. In the last term we can take either the homotopy $C_k$ fixed points by Proposition~\ref{prop: Induced/coinduced action and fixed points} or, as we have done, take the homotopy $C_{p^{v_p(k)}}$ fixed points. These are equivalent on $p$-complete spectra since the map $BC_{p^{v_p(k)}}\to BC_k$ is a $p$-adic equivalence.

On $\mathrm{QRSPerfd}$ by taking $\pi_0$ we then get a fiber sequence of sheaves \[\mathrm{gr}^0\left(\Sigma^{2-\CC(k)}\thh(-;\ZZ_p)\right)^{h\TT}\to \mathcal{O}_{\widehat{\DD}}\{-1\}\to \mathcal{O}_{\widehat{\DD}}\{-1\}/\mathcal{I}_{v_p(k)}\] and hence by unfolding a fiber sequence globally. It remains to show that the map \[\mathcal{O}_{\widehat{\DD}}\{-1\}\to \mathcal{O}_{\widehat{\DD}}\{-1\}/\mathcal{I}_{v_p(k)}\] is the quotient map. To this end note that fiber sequence~\ref{eqn: fiber sequence for C(l)-2} is exactly the fiber sequence used to construct the transfer of \cite{BG_transfer}. Thus the map in question is induced from the usual map from $\tc^{-}(-;\ZZ_p)$ to $\thh(-;\ZZ_p)^{hC_{p^{v_p(k)}}}$ which was exactly the map we used to identified the associated graded terms in Lemma~\ref{lem: Associated graded terms at finite levels}.
\end{proof}

Combining this all we get the following.

\begin{cor}\label{cor: finite levels and rep spheres associated graded.}
For any $k\geq 0$ there are equivalences  \[\mathrm{gr}^i\left(\Sigma^{\lambda_k-2k}\thh(-;\ZZ_p)\right)^{hC_{p^n}}\simeq \left(\bigotimes_{l=1}^k \mathcal{I}_{v_{p}(l)}\right)^{-1}\otimes(\mathcal{N}^{\geq i}\mathcal{O}_{\widehat{\DD}}/(\mathcal{N}^{\geq i+1}\mathcal{O}_{\widehat{\DD}}\otimes \mathcal{I}_n))\{i+k\}[2i]\]
and 
\[\mathrm{gr}^i\left(\Sigma^{\lambda_k-2k}\thh(-;\ZZ_p)\right)^{tC_{p^n}}\simeq \left(\bigotimes_{l=1}^k \mathcal{I}_{v_{p}(l)}\right)^{-1}\otimes \mathcal{O}_{\widehat{\DD}}/\mathcal{I}_n\{i+k\}[2i]\]
of quasisyntomic $\mathcal{O}_{\widehat{\DD}}$-modules.
\end{cor}

We can also adapt the above arguments to work for $n=\infty$ as well, which gives the following. 
\begin{cor}
For any $k\geq 0$ there are equivalences \[\mathrm{gr}^i\left(\Sigma^{\lambda_k-2k}\thh(-;\ZZ_p)\right)^{h\TT}\simeq \left(\bigotimes_{l=1}^k \mathcal{I}_{v_{p}(l)}\right)^{-1}\otimes\mathcal{N}^{\geq i}\mathcal{O}_{\widehat{\DD}}\{i+k\}[2i]\]
and 
\[\mathrm{gr}^i\left(\Sigma^{\lambda_k-2k}\thh(-;\ZZ_p)\right)^{t\TT}\simeq \left(\bigotimes_{l=1}^k \mathcal{I}_{v_{p}(l)}\right)^{-1}\otimes \mathcal{O}_{\widehat{\DD}}\{i+k\}[2i]\]
of quasisyntomic sheaves.
\end{cor}

We also need to know how this interacts with the construction we produced in the previous Subsection.

\begin{cor}\label{cor: ve and rep spheres associated graded}
Let $\mathrm{cof}(\Sigma^{\lambda_k-2k}V_{e, r, hfp})$ denote the quasisyntomic sheaf of spectra given by the cofiber of the map of sheaves \[\left(\Sigma^{\lambda_k-2k}\thh(-;\ZZ_p)\right)^{hC_{p^r}}\xrightarrow{\Sigma^{\lambda_k-2k}V_{e, r}}\left(\Sigma^{\lambda_k-2k}\thh(-;\ZZ_p)\right)^{hC_{p^{r+e}}}\] and similar notation for the Tate construction. Then the associated graded terms of the induced filtrations on these sheaves is given by 
\[\mathrm{gr}^i\left(\mathrm{cof}(\Sigma^{\lambda_k-2k}V_{e, r, hfp})\right)\simeq \left(\bigotimes_{l=1}^k \mathcal{I}_{v_p(l)}\right)^{-1}\otimes\mathcal{N}^{\geq i}\mathcal{O}_{\widehat{\DD}}/(\phi^k)^*(\mathcal{I}_e)\{i+k\}[2i]\] and 
\[\mathrm{gr}^i\left(\mathrm{cof}(\Sigma^{\lambda_k-2k}V_{e, r, Tate})\right)\simeq \left(\bigotimes_{l=1}^k \mathcal{I}_{v_p(l)}\right)\otimes \mathcal{O}_{\widehat{\DD}}/(\phi^k)^*(\mathcal{I}_e)\{i+k\}[2i].\]
\end{cor}

From \cite[Lemma 3.6]{Bhatt_Scholze} all the Frobenius twists of the Hodge-Tate divisor $\mathcal{I}$ should appearing in the formulas above should be trivializable. We finish this Subsection with a proof of this fact using the language of this paper. In order to do this we will first want an equivariant map which will do the trivialization. To this end consider the one point compactification of the $p^{th}$ power map $S^{\CC(k)}\to S^{\CC(pk)}$. This is a $\TT$-equivariant map of spaces, and so we may take the $C_p$ Tate construction to get a map $(S^{\CC(k)})^{tC_p}\to (S^{\CC(pk)})^{tC_p}$. After smashing with $\Sigma^{-2}\thh(-;\ZZ_p)^{tC_p}$ we get a map \[(S^{\CC(k)-2})^{tC_p}\wedge \thh(-;\ZZ_p)^{tC_p}\to (S^{\CC(pk)-2})^{tC_p}\wedge \thh(-;\ZZ_p)^{tC_p}\] sheaves of Borel $\TT$ spectra. 

We will now restrict our attention to the case when $p\mid k$. Then the $C_p$ action on $S^{\CC(k)-2}$ is trivial and so the map $(S^{\CC(k)-2})^{tC_p}\wedge X^{tC_p}\to (\Sigma^{\CC(k)-2}\wedge X)^{tC_p}$ witnessing the lax monoidal structure of the Tate construction is an equivalence for any Borel $\TT$ spectrum $X$. This map is natural and so there is an equivalence of sheaves of Borel $\TT$ spectra $(S^{\CC(k)-2})^{tC_p}\wedge \thh(-;\ZZ_p)^{tC_p}\simeq (\Sigma^{\CC(k)-2}\thh(-;\ZZ_p))^{tC_p}$.

There is then a map of sheaves of \[(\Sigma^{\CC(k)-2}\thh(-;\ZZ_p))^{tC_p}\to (\Sigma^{\CC(pk)-2}\thh(-;\ZZ_p))^{tC_p}\] and so upon taking homotopy fixed points there is a map of sheaves \[(\Sigma^{\CC(k)-2}\thh(-;\ZZ_p))^{t\TT}\to (\Sigma^{\CC(pk)-2}\thh(-;\ZZ_p))^{t\TT}\] which we will denote by $\phi^{B^{cy}}$.

\begin{lem}\label{lem: Frobenius simplifies rep spheres}
Let $k\in p\ZZ_+$ and let  $\phi^{B^{cy}}$ be the map constructed above. Then the map \[\mathrm{gr}^0\phi^{B^{cy}}:(\mathcal{I}_{v_p(k)}\{-1\})^{-1}\to (\mathcal{I}_{v_{p}(k)+1}\{-1\})^{-1}\] is an equivalence. 
\end{lem}
\begin{proof}
This is essentially an algebraic version of \cite[Lemma 2]{Hesselholt_Nikolaus}, and since we can define this map using topological methods we will prove this by reducing to their result. First note that it is enough to show this statement on $\mathrm{QRSPerfd}$ by Theorem~\ref{thm: qrsp are basis}. On such rings this becomes a statement about the groups $\pi_0((\Sigma^{\CC(k)-2}\thh(-;\ZZ_p))^{t\TT})$ and $\pi_0((\Sigma^{\CC(pk)-2}\thh(-;\ZZ_p))^{t\TT})$.

Consequently it is enough to show that the map $(\mathbb{S}^{\CC(k)})^{tC_p}\to (\mathbb{S}^{\CC(pk)})^{tC_p}$ is a $p$-adic equivalence since after smashing with $\thh(-;\ZZ_p)^{tC_p}$ the result will be $p$-complete since $\mathbb{S}^{\CC(k)}$ is a finite $C_p$-spectrum. To this end note that the map $S^{\CC(k)}\to (S^{\CC(pk)})^{C_p}$ inducing this map is an equivalence. Thus the cofiber of this map when $S^{\CC(k)}$ has trivial $C_p$-action is given by a finite $C_p$-space with free cells. After applying the Tate construction the cofiber will vanish and so $(S^{\CC(k)})^{tC_p}\to (S^{\CC(pk)})^{tC_p}$ is an equivalence. The result follows.
\end{proof}

\section{The filtration on the topological cyclic homology of truncated polynomial algebras}\label{sec: bms spectral sequence for truncated polynomial algebras}

In this section we will first combine the previous two sections to study a filtration on the quasisyntomic sheaves $\tc^{-}(-[x]/x^e;\ZZ_p)$ and $\tp(-[x]/x^e;\ZZ_p)$. To do this we will first use that we are working with a pointed monoid ring to express these functors in terms of the sheaves appearing in Section~\ref{sec: BMS spectral sequence} and Section~\ref{sec: two auxiliary computations}. Then we will use this filtration to construct a filtration on $\tc(-[x]/x^e;\ZZ_p)$ as a quasisyntomic sheaf.

\subsection{The filtration on topological negative cyclic and periodic homologies}

In order to compute these functors, we will first express the functor $\thh(-[x]/x^e;\ZZ_p)$ in terms of the topological Hochschild homology of the base. To this end recall that by \parencite[Section IV.2]{NS} the functor $\thh:\mathrm{Alg}_{\mathbb{E}_\infty}(\Sp)\to \mathrm{CycSp}$ is symmetric monoidal. Thus \[\thh(-[x]/x^e)\simeq \thh(-)\wedge B^{cy}(\Pi_e)\] where $B^{cy}(\Pi_e)$ is the cyclic bar construction which takes in a pointed monoid and is given in simplicial degree $\bullet$ by $\Pi_e^{\wedge \bullet +1}$ with the usual Hochschild style face and degeneracy maps. In our case $\Pi_e=\{0, 1, x, x^2, \ldots, x^{e-1}\}$ with $x^e=0$. This was studied extensively in \cite{Cyclic_Polytopes}, and we record the relevant results here.

First, consider the filtration on $B^{cy}(\Pi_e)$ induced by the degree filtration on $\Pi_e$. Specifically, this is the filtration which in weight $i$ and simplicial degree $j$ has terms $f_0(x)\wedge\ldots\wedge f_j(x)$ with $\deg(f_0)+\ldots + \deg(f_j)\geq i$. Let $B_n$ be the $n^{th}$ associated graded term of this filtration. This is split by the inclusion of degree $n$ terms and so we get a decomposition \[B^{cy}(\Pi)\simeq \bigvee_{n\geq 0}B_n\] and the relative functor $\thh(-[x]/x^e, (x))$ is given by excluding the zero summand.

Thus, we have an equivalence \[\thh(-[x]/x^e)\simeq \bigvee_{n\geq 0}\thh(-)\wedge B_n\] which is functorial and equivariant. The Frobenius will increase the wedge sum weight by multiplying it by $p$. Hesselholt and Madsen have already computed what the $\TT$-spaces are. We will use the computation in the form it is written in \cite[Theorem 5]{Hesselholt_handbook}:

\begin{thm}[\cite{Cyclic_Polytopes}, Theorem B;\cite{Hesselholt_handbook}, Theorem 5]\label{thm: calculation of the Bn}
There are equivalences of $\TT$-spaces \[B_n\simeq \Sigma^{\lambda_{\lfloor \frac{n-1}{e}\rfloor}}_+ \TT/C_n\] if $e\nmid n$ and \[B_n\simeq \Sigma^{\lambda_{\lfloor \frac{n-1}{e}\rfloor}}\mathrm{cof(V_e:(\TT/C_{\frac{n}{e}})_+\to (\TT/C_n)_+)}\] if $e\mid n$. 

\end{thm}

A few remarks are in order. First, $\dim(\lambda_{\lfloor\tfrac{n}{e}\rfloor})=2\lfloor \frac{n}{e}\rfloor$ and so there are only finitely many $n$ with $\thh(-)\wedge B_n$ having connectivity below a given value. Consequently \[\bigvee_{n\geq 1}\thh(-)\wedge B_n\simeq \prod_{n\geq 1}\thh(-)\wedge B_n\] as cyclotomic spectra. Thus \[\left(\bigvee_{n\geq 1}\thh(-)\wedge B_n\right)^{h\TT}\simeq \prod_{n\geq 1}(\thh(-)\wedge B_n)^{h\TT}\] since limits commute with limits.

On the other hand, \[\left(\bigvee_{n\geq 1}\thh(-)\wedge B_n\right)_{h\TT}\simeq \bigvee_{n\geq 1}(\thh(-)\wedge B_n)_{h\TT}\] since colimits commute with colimits. Since homotopy orbits preserve connectivity we still have the connectivity bounds of the previous paragraph and so \[\bigvee_{n\geq 1}(\thh(-)\wedge B_n)_{h\TT}  \simeq \prod_{n\geq 1}(\thh(-)\wedge B_n)_{h\TT}\] and so this product description is true for the homotopy orbits and fixed points. Consequently, it is also the case for the Tate construction.

Combining all of this gives the following.

\begin{lem}\label{lem: weight decomposition} 

There are equivalences  

\[\tc^-(-[x]/x^e, (x);\ZZ_p)\simeq \left(\prod_{n\geq 1, e\nmid n} (\Sigma^{\lambda_{\lfloor\tfrac{n-1}{e}\rfloor}+1}\thh(-;\ZZ_p))^{hC_{p^{v_p(n)}}}\right)\times \left(\prod_{k\geq 1}\mathrm{cof}\left(\Sigma^{\lambda_{k-1}+1}V_{v_p(e), v_p(k),hfp}\right)\right)\] and \[\tp(-[x]/x^e, (x);\ZZ_p)\simeq \left(\prod_{n\geq 1, e\nmid n} (\Sigma^{\lambda_{\lfloor\tfrac{n-1}{e}\rfloor}+1}\thh(-;\ZZ_p))^{tC_{p^{v_p(n)}}}\right)\times \left(\prod_{k\geq 1}\mathrm{cof}\left(\Sigma^{\lambda_{k-1}+1}V_{v_p(e),v_p(k), Tate}\right)\right)\] with notation as in Corollary~\ref{cor: ve and rep spheres associated graded}. 

\end{lem} 

\begin{proof} 

We have the product decomposition discussed above, and so it remains only to identify the \[(\thh(-;\ZZ_p)\wedge B_n)^{h\TT}\] and \[(\thh(-;\ZZ_p)\wedge B_n)^{t\TT}\]terms. In the case that $e\nmid n$ this follows from Theorem~\ref{thm: calculation of the Bn} together with Proposition~\ref{prop: Induced/coinduced action and fixed points}. For the case when $e\mid n$, take $k=\frac{n}{e}$ and apply the same results. 

\end{proof}

As products of sheaves, both $\tc^{-}(-[x]/x^e, (x); \ZZ_p)$ and $\tp(-[x]/x^e,(x);\ZZ_p)$ are quasisyntomic sheaves. We can then define a filtration on these objects using the filtrations we produced in the previous Sections.

\begin{con} 

Using the decompositions of Lemma~\ref{lem: weight decomposition} we define decreasing filtrations on the sheaves $\tc^{-}(-[x]/x^e,(x);\ZZ_p)$ and $\tp(-[x]/x^e, (x);\ZZ_p)$ with $i^{th}$ filtered term given by 

\begin{align*} 
   \mathrm{Fil}^{\geq i}\tc^{-}(-[x]/x^e, (x);\ZZ_p) = &\left(\prod_{n\geq 1, e\nmid n} \mathrm{Fil}^{\geq i-\lfloor \tfrac{n-1}{e}\rfloor}(\Sigma^{\lambda_{\lfloor\tfrac{n-1}{e}\rfloor}-2\lfloor\tfrac{n-1}{e}\rfloor}\thh(-;\ZZ_p))^{hC_{p^{v_p(n)}}}[2\lfloor\frac{n-1}{e}\rfloor+1]\right)\\ 
   &\times \left(\prod_{k\geq 1}\mathrm{Fil}^{\geq i-k+1}\mathrm{cof}\left(\Sigma^{\lambda_{k-1}-2k+2}V_{v_p(e),v_p(k),hfp}\right)[2(k-1)+1]\right) 
\end{align*} 

and  

\begin{align*} 
    \mathrm{Fil}^{\geq i}\tp(-[x]/x^e, (x);\ZZ_p) = &\left(\prod_{n\geq 1, e\nmid n} \mathrm{Fil}^{\geq i-\lfloor \tfrac{n-1}{e}\rfloor}(\Sigma^{\lambda_{\lfloor\tfrac{n-1}{e}\rfloor}-2\lfloor\tfrac{n-1}{e}\rfloor}\thh(-;\ZZ_p))^{tC_{p^{v_p(n)}}}[2\lfloor\frac{n-1}{e}\rfloor+1]\right)\\ 
   &\times \left(\prod_{k\geq 1}\mathrm{Fil}^{\geq i-k+1}\mathrm{cof}\left(\Sigma^{\lambda_{k-1}-2k+2}V_{v_p(e),v_p(k),Tate}\right)[2(k-1)+1]\right) 
\end{align*} 
respectively. 

\end{con}

\begin{rem} 

The definition of the filtrations above involve several points which deserve some justification. To begin with, in the decompositions of Lemma~\ref{lem: weight decomposition} we have terms of the form $(\Sigma^{\lambda_k}\thh(-;\ZZ_p))^{hC_{p^n}}$ for various values of $k$ and $n$. In Section~\ref{sec: two auxiliary computations} we defined filtrations not on terms of these sheaves but on sheaves of the form $(\Sigma^{\lambda_k-2k}\thh(-;\ZZ_p))^{hC_p^n}$, where we have changed the sphere we are suspending by.

To fix this, note that \[(\Sigma^{\lambda_k}\thh(-;\ZZ_p))^{hC_{p^n}}\simeq \Sigma^{2k}(\Sigma^{\lambda_{k}-2k}\thh(-;\ZZ_p))^{hC_{p^n}}\] since the $C_{p^n}$ action on $\mathbb{S}^{2k}$ is trivial. This later term is the suspension of a sheaf we constructed a filtration on in Section~\ref{sec: two auxiliary computations} and so we use that filtration appropriately suspended.

Another point of the above construction which deserves an explanation is the shifts in filtration on the individual terms in the product. To justify those, let $\mathrm{Fil}^{\geq *}\mathcal{F}:\mathrm{Qsyn}\to \widehat{\operatorname{DF}}$ be any of the filtered sheaves constructed in Section~\ref{sec: two auxiliary computations}, and so also appearing up to a shift in the above construction. Then for any $i$, $\mathrm{gr}^i\mathcal{F}(S)$ will be concentrated in degree $2i$ for any given quasiregular semiperfectoid ring $S$. In particular $\mathrm{gr}^i\mathcal{F}(S)[2k+1]$ will be concentrated in degree $2i+2k+1$. Hence if we do not shift the filtrations in the construction above, then even for quasiregular semiperfectoid rings the associated graded terms of the filtration will not be concentrated in a single degree.

On the other hand, in the above convention we have shifted the filtration down by $k$. Thus, the terms appearing in the $i^{th}$ associated graded will be of the form $\mathrm{gr}^{i-k}\mathcal{F}(S)[2k+1]$, which will be concentrated in degree $2(i-k)+2k+1 = 2i+1$. In other words, the shifts in the filtrations for the above formula appear so that on quasiregular semiperfectoid rings $S$ this filtration becomes \[\mathrm{Fil}^{\geq *}\tc^{-}(S[x]/x^e, (x);\ZZ_p)\simeq \tau_{\geq 2*+1}\tc^{-}(S[x]/x^e,(x);\ZZ_p)\] and \[\mathrm{Fil}^{\geq *}\tp(S[x]/x^e,(x);\ZZ_p)\simeq \tau_{\geq 2*+1}\tp(S[x]/x^e,(x);\ZZ_p)\] respectively. It follows that both of these filtrations are exhaustive on quasiregular semiperfectoid rings, and therefore they will be exhaustive for all quasisyntomic rings by descent. 

\end{rem}

\begin{thm} 

There are identifications of sheaves on $\mathrm{QSyn}$ 
\begin{align*} 
    \mathrm{gr}^i\tc^{-}&(-[x]/x^e, -;\ZZ_p) \simeq\\ &\left(\prod_{n\geq 1, e\nmid n} \left(\bigotimes_{l=1}^{\lfloor\tfrac{n-1}{e}\rfloor}\mathcal{I}_{v_p(l)}\right)^{-1}\otimes(\mathcal{N}^{\geq i-\lfloor \tfrac{n-1}{e}\rfloor}\mathcal{O}_{\widehat{\DD}}/(\mathcal{N}^{\geq i+1-\lfloor \tfrac{n-1}{e}\rfloor}\mathcal{O}_{\widehat{\DD}}\otimes \mathcal{I}_{v_p(n)}))\{i\}[2i+1]\right)\\ 
    &\times \left(\prod_{k\geq 1}\left(\bigotimes_{l=1}^{k-1}\mathcal{I}_{v_p(l)}\right)^{-1} \otimes (\mathcal{N}^{\geq i-k+1}\mathcal{O}_{\widehat{\DD}}/(\phi^{v_p(k)})^*(\mathcal{I}_{v_p(e)}))\{i\}[2i+1]\right) 
\end{align*} 

and 

\begin{align*} 
    \mathrm{gr}^i\tp(-[x]/x^e, -;\ZZ_p) \simeq &\left(\prod_{n\geq 1, e\nmid n} \left(\bigotimes_{l=1}^{\lfloor\tfrac{n-1}{e}\rfloor}\mathcal{I}_{v_p(l)}\right)^{-1}\otimes(\mathcal{O}_{\widehat{\DD}}/\mathcal{I}_{v_p(n)})\{i\}[2i+1]\right)\\ 
    &\times \left(\prod_{k\geq 1}\left(\bigotimes_{l=1}^{k-1}\mathcal{I}_{v_p(l)}\right)^{-1}\otimes(\mathcal{O}_{\widehat{\DD}}/(\phi^{v_p(k)})^*(\mathcal{I}_{v_p(e)}))\{i\}[2i+1]\right) 
\end{align*} 
where the filtration is the one constructed in the previous paragraph. All the tensor products are taken in the category of quasisyntomic $\mathcal{O}_{\widehat{\DD}}$-module sheaves. 

\end{thm} 

\begin{proof} 

Since we may commute products and fiber sequences the associated graded terms of this filtration are the products of the associated graded pieces of the individual factors. There are then four cases to consider, two for the homotopy fixed points and two for the Tate construction. Since the arguments are verbatim, we will only go through the homotopy fixed point argument in detail. Recall the function $t=t(u,p,s,e)=t(n,e):=\lfloor\frac{n-1}{e}\rfloor$.

Case One: $n\geq 1$, $e\nmid n$ terms in $\mathrm{gr}^i\tc^{-}(-[x]/x^e, (x);\ZZ_p)$. In this case we get a contribution of the form \begin{align*} 
    \mathrm{gr}^{i-t}&\left(\Sigma^{\lambda_{t}-2t}\thh(-;\ZZ_p)\right)^{hC_{p^{v_p(n)}}}[2t +1] \\ 
    &\simeq \left(\left(\bigotimes_{l=1}^{t}\mathcal{I}_{v_p(l)}\right)^{-1}\otimes \mathcal{N}^{\geq i-t}\mathcal{O}_{\widehat{\DD}}/\mathcal{N}^{\geq  i-t + 1}\mathcal{O}_{\widehat{\DD}}\otimes \mathcal{I}_{v_p(n)}\{i-t + t\}[2i-2t]\right)[2t + 1]\\ 
    &\simeq \left(\bigotimes_{l=1}^{t}\mathcal{I}_{v_p(l)}\right)^{-1}\otimes \left(\mathcal{N}^{\geq i-t}\mathcal{O}_{\widehat{\DD}}/\mathcal{N}^{\geq i-t+1}\mathcal{O}_{\widehat{\DD}}\otimes\mathcal{I}_{v_{p}(n)}\right)\{i\}[2i+1] 
\end{align*} 
as claimed. The first equivalence is from Corollary`\ref{cor: finite levels and rep spheres associated graded.} with $i$ replaced with $i-t$ and $k$ replaced with $t$.

Case Two: $k\geq 1$ terms in $\mathrm{gr}^i\tc^{-}(-[x]/x^e,(x);\ZZ_p)$. In this case we get a contribution of the form \begin{align*} 
    &\mathrm{gr}^{i-k+1}\mathrm{cof}\left(\Sigma^{\lambda_{k-1}-2(k-1)}V_{v_p(e),v_p(k),hfp}\right)[2k-1] \simeq\\ 
    &\left(\left(\bigotimes_{l=1}^{k-1}\mathcal{I}_{v_p(l)}\right)^{-1}\otimes \mathcal{N}^{\geq i-(k-1)}\mathcal{O}_{\widehat{\DD}}/(\phi^{v_p(k)})^*(\mathcal{I}_{v_p(e)})\{i-(k-1) + (k-1)\}[2i-2(k-1)]\right)[2k - 1]\\ 
    &\simeq \left(\bigotimes_{l=1}^{k-1}\mathcal{I}_{v_p(l)}\right)^{-1}\otimes \left(\mathcal{N}^{\geq i-(k-1)}\mathcal{O}_{\widehat{\DD}}/(\phi^{v_p(k)})^*(\mathcal{I}_{v_{p}(n)})\right)\{i\}[2i+1] 
\end{align*} 
as claimed. The first equivalence is Corollary~\ref{cor: ve and rep spheres associated graded} with $i$ replaced with $i-(k-1)$ and $k$ replaced with $k-1$. 

\end{proof}

\begin{rem} 

To go from the above Theorem to a statement about $\tc^{-}(-[x]/x^e;\ZZ_p)$ and $\tp(-[x]/x^e;\ZZ_p)$ all that is unaccounted for is the weight zero component. This is given by $\tc^{-}(-;\ZZ_p)$ and $\tp(-;\ZZ_p)$ since $B_0=S^0$. The filtration will then be the BMS filtration.  

\end{rem}

\subsection{The filtration on topological cyclic homology}

Since $\mathrm{Fil}^{\geq i}\tc^{-}(-[x]/x^e,(x);\ZZ_p)$ and $\mathrm{Fil}^{\geq i}\tp(-[x]/x^e, (x) ;\ZZ_p)$ are given by the Postnikov truncation for quasi\-regular semiperfectoid rings it follows that both the canonical and Frobenius maps are maps of filtered sheaves. In particular, we get the following.

\begin{cor}\label{cor: First attempt at describing the associated graded for tc} 

The sheaf $\tc(-[x]/x^e,(x);\ZZ_p):\mathrm{QSyn}\to \Sp$ lifts to a functor \[(\tc(-[x]/x^e,(x);\ZZ_p),\mathrm{Fil}^{\geq i}\tc(-[x]/x^e,(x);\ZZ_p)):\mathrm{QSyn}\to \widehat{\mathrm{DF}}\] such that  
\begin{align*} 
&\mathrm{gr}^i(\tc(-[x]/x^e,(x);\ZZ_p))\simeq\\ 
&\mathrm{fib}\left( 
\begin{tikzcd} 
{\left(\prod_{n\geq 1, e\nmid n} \left(\bigotimes_{l=1}^{\lfloor\tfrac{n-1}{e}\rfloor}\mathcal{I}_{v_p(l)}\right)^{-1}\otimes (\mathcal{N}^{\geq i-\lfloor \tfrac{n-1}{e}\rfloor}\mathcal{O}_{\widehat{\DD}}/(\mathcal{N}^{\geq i+1-\lfloor \tfrac{n-1}{e}\rfloor}\mathcal{O}_{\widehat{\DD}}\otimes \mathcal{I}_{v_p(n)}))\right)\{i\}[2i+1]} \\ 
{\times \left(\prod_{k\geq 1}\left(\bigotimes_{l=1}^{k-1}\mathcal{I}_{v_p(l)}\right)^{-1}\otimes (\mathcal{N}^{\geq i-k+1}\mathcal{O}_{\widehat{\DD}}/(\phi^{v_p(k)})^*(\mathcal{I}_{v_p(e)}))\{i\}[2i+1]\right)} \arrow[dd, "can-\phi_p"]                                                                                \\ 
                                                                                                                                                                                                                                                                                                                     \\ 
{\left(\prod_{n\geq 1, e\nmid n} \left(\bigotimes_{l=1}^{\lfloor\tfrac{n-1}{e}\rfloor}\mathcal{I}_{v_p(l)}\right)^{-1}\otimes (\mathcal{O}_{\widehat{\DD}}/\mathcal{I}_{v_p(n)})\{i\}[2i+1]\right)}                                                                                                                       \\ 
{\times \left(\prod_{k\geq 1}\left(\bigotimes_{l=1}^{k-1}\mathcal{I}_{v_p(l)}\right)^{-1}\otimes(\mathcal{O}_{\widehat{\DD}}/(\phi^{v_p(k)})^*(\mathcal{I}_{v_p(e)}))\{i\}[2i+1]\right)}     
\end{tikzcd} 
\right)\end{align*} 
\end{cor}

We will spend the rest of this section simplifying this expression for the associated graded as much as possible. First note that the equivalence of Proposition~\ref{prop: Induced/coinduced action and fixed points} will preserve the canonical map since it will not affect the universal property given by \cite[Theorem I.4.1]{NS}. Consequently, the canonical map in the above fiber sequence is induced by the inclusions $\mathcal{N}^{\geq i}\mathcal{O}_{\widehat{\DD}}\hookrightarrow \mathcal{O}_{\widehat{\DD}}$. In particular when $\lfloor \frac{n-1}{e}\rfloor\geq i+1$ we then have that $can$ is an isomorphism. Since $\phi_p$ sends $B_n$ to $B_{pn}$ we have that the induced Frobenius in the fiber sequence above multiplies $n$ and $k$ by $p$ unless $e\mid pn$ in which case it sends the degree $n$ factor in the source to the degree $\frac{np}{e}-1$ factor in the second product in the target. Hence for degree reasons this must be an injective map of sheaves. In fact more is true, which we record in the following Lemma

\begin{lem}\label{lem: slightly smaller product} 

There is an equivalence of sheaves  
\begin{align*} 
&\mathrm{gr}^i(\tc(-[x]/x^e,(x);\ZZ_p))\simeq\\ 
&\mathrm{fib}\left( 
\begin{tikzcd} 
{\left(\prod_{e(i+1)>n-1\geq 0, e\nmid n} \left(\bigotimes_{l=1}^{\lfloor\tfrac{n-1}{e}\rfloor}\mathcal{I}_{v_p(l)}\right)^{-1}\otimes (\mathcal{N}^{\geq i-\lfloor \tfrac{n-1}{e}\rfloor}\mathcal{O}_{\widehat{\DD}}/(\mathcal{N}^{\geq i+1-\lfloor \tfrac{n-1}{e}\rfloor}\mathcal{O}_{\widehat{\DD}}\otimes \mathcal{I}_{v_p(n)}))\{i\}[2i+1]\right)} \\ 
{\times \left(\prod_{i+1>k\geq 1}\left(\bigotimes_{l=1}^{k-1}\mathcal{I}_{v_p(l)}\right)^{-1}\otimes (\mathcal{N}^{\geq i-k+1}\mathcal{O}_{\widehat{\DD}}/(\phi^{v_p(k)})^*(\mathcal{I}_{v_p(e)}))\{i\}[2i+1]\right)} \arrow[dd, "can-\phi_p"]                                                                                   \\ 
                                                                                                                                                                                                                                                                                                                     \\ 
{\left(\prod_{e(i+1)>n-1\geq 1, e\nmid n} \left(\bigotimes_{l=1}^{\lfloor\tfrac{n-1}{e}\rfloor}\mathcal{I}_{v_p(l)}\right)^{-1}\otimes (\mathcal{O}_{\widehat{\DD}}/\mathcal{I}_{v_p(n)})\{i\}[2i+1]\right)}                                                                                                                        \\ 
{\times \left(\prod_{i+1>k\geq 1}\left(\bigotimes_{l=1}^{k-1}\mathcal{I}_{v_p(l)}\right)^{-1}\otimes (\mathcal{O}_{\widehat{\DD}}/(\phi^{v_p(k)})^*(\mathcal{I}_{v_p(e)}))\{i\}[2i+1]\right)}                
\end{tikzcd} 
\right)\end{align*} 
\end{lem} 

\begin{proof} 

It is enough to show that there is a pullback square between the fiber sequence above and the one in Corollary~\ref{cor: First attempt at describing the associated graded for tc}. There is a natural map from the fiber sequence in Corollary~\ref{cor: First attempt at describing the associated graded for tc} to the one above, and the total fiber is given by the fiber of the map \[\begin{tikzcd} 
{\left(\prod_{n-1\geq e(i+1), e\nmid n} \left(\bigotimes_{l=1}^{\lfloor\tfrac{n-1}{e}\rfloor}\mathcal{I}_{v_p(l)}\right)^{-1}\otimes \mathcal{N}^{\geq i-\lfloor \tfrac{n-1}{e}\rfloor}\mathcal{O}_{\widehat{\DD}}/\mathcal{N}^{\geq i+1-\lfloor \tfrac{n-1}{e}\rfloor}\mathcal{O}_{\widehat{\DD}}\otimes \mathcal{I}_{v_p(n)}\{i\}[2i+1]\right)} \\ 
{\times \left(\prod_{k\geq i+1}\left(\bigotimes_{l=1}^{k-1}\mathcal{I}_{v_p(l)}\right)^{-1} \otimes \mathcal{N}^{\geq i-k+1}\mathcal{O}_{\widehat{\DD}}/(\phi^{v_p(k)})^*(\mathcal{I}_{v_p(e)})\{i\}[2i+1]\right)} \arrow[dd, "can-\phi_p"]                                                                                   \\ 
                                                                                                                                                                                                                                                                                                                     \\ 
{\left(\prod_{n-1\geq e(i+1), e\nmid n} \left(\bigotimes_{l=1}^{\lfloor\tfrac{n-1}{e}\rfloor}\mathcal{I}_{v_p(l)}\right)^{-1}\otimes \mathcal{O}_{\widehat{\DD}}/\mathcal{I}_{v_p(n)}\{i\}[2i+1]\right)}                                                                                                                        \\ 
{\times \left(\prod_{k\geq i+1}\left(\bigotimes_{l=1}^{k-1}\mathcal{I}_{v_p(l)}\right)^{-1}\otimes \mathcal{O}_{\widehat{\DD}}/(\phi^{v_p(k)})^*(\mathcal{I}_{v_p(e)})\{i\}[2i+1]\right)}                                                                                                                                    
\end{tikzcd}\] which we will show is an equivalence. To show this it is enough to show that it is an equivalence of sheaves after restricting to $\mathrm{QRSPerfd}$ where all the sheaves in question are discrete. We also already know that the map in question is injective by the discussion before this Lemma, so it is enough to show that this map is surjective. Let $S$ be a quasiregular semiperfectoid ring with orientation $d$ and let $y\in \left(\prod_{n-1\geq e(i+1), e\nmid n}\widehat{\DD}_S/\tilde{d}_{v_p(n)}[2i+1]\right)\times \left(\prod_{k\geq i+1} \widehat{\DD}_S/(\phi^{v_p(k)})^*(\Tilde{d}_{v_p(e)})[2i+1]\right)$, in other words an element of the target of the above map when applied to $S$. Let $A_n=\widehat{\DD}_S/\tilde{d}_{v_p(n)}[2i+1]$ if $e\nmid n$ and $A_n=\widehat{\DD}_S/(\phi^{v_p(n)-v_p(e)})^*(\Tilde{d}_{v_p(e)})[2i+1]$ if $e\mid n$, so that $y\in \prod_{n> e(i+1)} A_n$. The map we are now trying to show is surjective is then a map \[\prod_{n>e(i+1)}A_n\xrightarrow{id-\phi} \prod_{n>e(i+1)}A_n\] where $\phi$ multiplies the product index by $p$.

To see that $y$ is in the image of the above map we will define a pre-image inductively. For each $n>e(i+1)$ which is coprime to $p$ define $x_{n}=y_n$, where $y_n$ is the projection of $y$ onto the $n^{th}$ factor $A_n$. Inductively suppose that elements $x_{p^{r}n}$ have been defined for $r\in \mathbb{N}$ and all $n>e(i+1)$ coprime to $p$. Then define $x_{p^{r+1}n}=y_{p^{r+1}n}+\phi(x_{p^rn})$. Then the element $x\in \prod_{n>e(i+1)}A_n$ whose entries are given by the above inductive procedure is a pre-image of the element $y$.

Hence we have that the horizontal fibers are equivalent and therefore the square is a pullback square. 

\end{proof}

It will be helpful to consider how the Frobenius acts on the product grading. The first step will be to re-index our products so that $can-\phi$ will respect at least one of the indices. Let $e'=\frac{e}{v_p(e)}$ and let $J_p=\{a\in \ZZ_+|(a,p)=1\}$. Recall the notation $t=t(u,p,m,e):= \lfloor \frac{up^m-1}{e}\rfloor$. Then we get that  

\begin{equation}\label{eqn: quasisyntomic filtration for tc} 
\begin{aligned} 
&\mathrm{gr}^i(\tc(-[x]/x^e,(x);\ZZ_p))\simeq \\ 
&\prod_{u\in J_p\setminus e'J_p}\mathrm{fib}\left( 
\begin{tikzcd} 
{\left(\prod_{m\geq 0, up^m-1<e(i+1)} \left(\bigotimes_{l=1}^{t}\mathcal{I}_{v_p(l)}\right)^{-1}\otimes (\mathcal{N}^{\geq i-t}\mathcal{O}_{\widehat{\DD}}/(\mathcal{N}^{\geq i+1-t}\mathcal{O}_{\widehat{\DD}}\otimes \mathcal{I}_{m}))\{i\}[2i+1]\right)}\arrow[d, "can-\phi_p"]                                                                                   \\ 
{\left(\prod_{m\geq 0, up^m-1<e(i+1)} \left(\bigotimes_{l=1}^{t}\mathcal{I}_{v_p(l)}\right)^{-1}\otimes (\mathcal{O}_{\widehat{\DD}}/\mathcal{I}_{m})\{i\}[2i+1]\right)}    
\end{tikzcd} 
\right)\\ 
&\times \prod_{u\in e'J_p}\mathrm{fib}\left(\begin{tikzcd} 
{\left(\prod_{v_p(e)>u\geq 0, up^m-1<e(i+1)} \left(\bigotimes_{l=1}^{t}\mathcal{I}_{v_p(l)}\right)^{-1}\otimes (\mathcal{N}^{\geq i-t}\mathcal{O}_{\widehat{\DD}}/(\mathcal{N}^{\geq i+1-t}\mathcal{O}_{\widehat{\DD}}\otimes \mathcal{I}_{m}))\{i\}[2i+1]\right)} \\ 
{\times \left(\prod_{m\geq v_p(e), {up^m-1}<e(i+1)}\left(\bigotimes_{l=1}^{t}\mathcal{I}_{v_p(l)}\right)^{-1}\otimes (\mathcal{N}^{\geq i-t}\mathcal{O}_{\widehat{\DD}}/(\phi^{m-v_p(e)})^*(\mathcal{I}_{v_p(e)}))\{i\}[2i+1]\right)} \arrow[dd, "can-\phi_p"]                         \\            \\ 
{\left(\prod_{v_p(e)>m\geq 0, up^m-1<e(i+1)} \left(\bigotimes_{l=1}^{t}\mathcal{I}_{v_p(l)}\right)^{-1}\otimes \mathcal{O}_{\widehat{\DD}}/\mathcal{I}_{m}\{i\}[2i+1]\right)}                                                                                                                      \\ 
{\times \left(\prod_{m\geq v_p(e), up^m-1<e(i+1)}\left(\bigotimes_{l=1}^{t}\mathcal{I}_{v_p(l)}\right)^{-1}\otimes (\mathcal{O}_{\widehat{\DD}}/(\phi^{m-v_p(e)})^*(\mathcal{I}_{v_p(e)}))\{i\}[2i+1]\right)}              
\end{tikzcd}\right)\end{aligned}
\end{equation} 

\section{Computations}

This section is devoted to using the previous section to compute the algebraic $K$-theory of truncated polynomials for various different kinds of rings. In order to do this we first will simplify the filtration produced in the previous section in the case that we have a strong form of the Segal conjecture.

In order to do this we first prove some helpful Lemmas which are algebraic analogues of Tsalidis's Theorem from \cite{Tsalidis}.

\begin{lem}\label{lem: Tsaladis's algebraic theorem}
Let $S$ be an F-smooth ring with $\fdim(S)\leq j$. Then the map \[\phi_j:\mathcal{N}^{\geq k}\widehat{\DD}_S\{j\}\to I^{\otimes (k-j)}\{j\}\] is an equivalence in $\widehat{\operatorname{DF}}(\ZZ_p)$ for all $k\geq j$. Here the filtration on the left hand term is the Nygaard filtration $F^{\geq *}\mathcal{N}^{\geq k}\widehat{\DD}_S\{j\}:= \operatorname{R\Gamma}(S; \mathcal{N}^{\geq k+*}\mathcal{O}_{\widehat{\DD}}\{j\})$ and the filtration on the right is the $I$-adic filtration $F^*I^{\otimes k-j}\{j\}=\operatorname{R\Gamma}(S;\mathcal{I}^{\otimes k-j+*}\{j\})$. 
\end{lem}
\begin{proof}
Note that $\mathcal{N}^{\geq k}\widehat{\DD}_S\{j\}$ and $I^{\otimes (k-j)}\{j\}$ are exactly the $k^{th}$ filtered terms of $\mathcal{N}^{\geq j}\widehat{\DD}_S\{j\}$ and $\widehat{\DD}_S\{j\}$(ie the case of $k=j$ of this Lemma). Thus it is enough to show the statement for $k=j$. In this case the assumption that $\phi_i:\mathcal{N}^i\widehat{\DD}_S\to \overline{\DD}_S\{i\}$ are equivalences is exactly the condition that $\mathrm{gr}^{(i-j)}\phi_j$ is an equivalence for all $i$. Consequently since both filtrations are complete the result follows.
\end{proof}

\begin{ex}
    When $S$ is a perfectoid ring the statement above becomes that the map \[\frac{\phi}{\tilde{d}^j}:d^kA_{inf}(S)\to \tilde{d}^{k-j}A_{inf}(S)\] is an equivalence of filtered $A_{inf}(S)$-modules where the filtration on the left is the $d$-adic filtration and the filtration on the right is the $\tilde{d}$-adic filtration. This then follows from the fact that $\phi:A_{inf}(S)\to A_{inf}(S)$ is an isomorphism for $S$ perfectoid.
\end{ex}

In order to get that the Frobenius is an isomorphism on the associated graded terms from the previous section we also need that the quotient terms appearing there are also still isomorphic. 
\begin{lem}\label{lem: Segal conjecture with line bundles}
Let $S$ and $j$ be as above. Then the induced map $\operatorname{R\Gamma}(S;\mathcal{N}^{\geq j+1}\mathcal{O}_{\widehat{\DD}}\{j\}\otimes \mathcal{I}_m)\to\operatorname{R\Gamma}(S;\mathcal{I}_{m+1}\{j\})$ is an equivalence in $\widehat{\operatorname{DF}}(\ZZ_p)$ for all $m\geq 0$. 
\end{lem}
\begin{proof}
Upon taking associated graded terms this becomes a map \[\operatorname{R\Gamma}(S;\mathcal{N}^k\mathcal{O}_{\widehat{\DD}}\otimes \mathcal{I}_m)\to \operatorname{R\Gamma}(S;\overline{\mathcal{O}}_{\widehat{\DD}}\otimes\phi^*(\mathcal{I}_m)\{k\})\] which upon linearization is the identity on $\phi^*(\mathcal{I}_m)$ tensored with $\phi_k:\mathcal{N}^k\mathcal{O}_{\widehat{\DD}}\to \overline{\mathcal{O}}_{\widehat{\DD}}\{k\}$. Thus in order to show the main result it is enough to show that $\phi_k$ is a filtered isomorphism where the filtration on the source is given by $\mathrm{Fil}^{\geq *}\mathcal{N}^k\widehat{\DD}_S=I_*\mathcal{N}^k\widehat{\DD}_S$ and the filtration on the target is $\mathrm{Fil}^{\geq *}\overline{\DD}_S\{k\}:=\phi^*(I_*)\overline{\DD}_S\{k\}$. 

In fact both filtrations are the $p$-adic filtration. To see this first note that we can reduce to the case of $S$ quasiregular semiperfectoid by Theorem~\ref{thm: qrsp are basis}. Let $R$ be a perfectoid ring admitting a map $R\to S$ and let $d$ be an orientation of $R$. Then $\mathcal{N}^{k}\widehat{\DD}_S$ is an $\widehat{\DD}_S/d$-module and $I_m\mod d\cong (\phi(d)\phi^2(d)\ldots \phi^m(d))\mod d=p^m$. Similarly $\overline{\DD}_S\{k\}=\widehat{\DD}_S/\tilde{d}$ and so $\phi^*(I_m)\mod \tilde{d}=\phi(p^m)=p^m$. Thus there is a map from the $p$-adic filtration into the above filtrations which is an equivalence. The choice of $R$ and $d$ are artificial since the map above does not depend on them, it only makes the isomorphism clear.
\end{proof}

Thus so long as the Frobenius maps in Equation~\ref{eqn: quasisyntomic filtration for tc} are induced from the divided Frobenius maps in the above way we can then apply the above Lemmas to extend a Segal conjecture to our setting. To see that these do indeed give the formula for the Frobenius map in Equation~\ref{eqn: quasisyntomic filtration for tc}, note that it is enough to check that it is on quasiregular semiperfectoid rings. By Lemma~\ref{lem: finite levels for qrsp} $\thh(-;\ZZ_p)^{hC_{p^n}}$ and $\thh(-;\ZZ_p)^{tC_{p^{n+1}}}$ are compact $\tc^{-}(-;\ZZ_p)$ and $\tp(-;\ZZ_p)$ modules respectively for quasiregular semiperfectoid rings. Since the square
\[\begin{tikzcd}
\tc^{-}(-;\ZZ_p)\arrow[r,"\phi_p^{h\TT}"] \arrow[d] & \tp(-;\ZZ_p)\arrow[d]\\
\thh(-;\ZZ_p)^{hC_{p^n}}\arrow[r,"\phi_p^{hC_{p^n}}"] & \thh(-;\ZZ_p)^{tC_{p^{n+1}}}
\end{tikzcd}\]
commutes and the vertical maps are quotient maps the result follows.

Note that since $\thh(-;\ZZ_p):\mathrm{CAlg}(\Sp)\to \mathrm{CycSp}_p$ is symmetric monoidal we can use Lemma~\ref{lem: Frobenius simplifies rep spheres} to, up to equivalence, ignore the line bundles coming for $S^{\lambda}$ in our formulas for the Frobenius maps.

With this in hand we may dramatically simplify the associated graded terms of Equation~\ref{eqn: quasisyntomic filtration for tc} when we have a strong version of the Segal conjecture.
\begin{thm}\label{thm: tc filtration simplification segal}
Let $S$ be an F-smooth ring. Then 
\begin{align*}
&\mathrm{gr}^i(\tc(S[x]/x^e,(x);\ZZ_p))[-2i-1]\simeq \\
&\prod_{u\in J_p\setminus e'J_p}
{\operatorname{R\Gamma}_{\mathrm{QSyn}}\left(S, \mathcal{I}^{-\lfloor t/p\rfloor }\otimes (\mathcal{N}^{\geq i-t}\mathcal{O}_{\widehat{\DD}}/\mathcal{N}^{\geq i+1-t}\mathcal{O}_{\widehat{\DD}}\otimes \mathcal{I}_{s-1})\{i\}\right)} \\ 
&\times \prod_{u\in e'J_p}\begin{cases}{\operatorname{R\Gamma}_{\mathrm{QSyn}}\left(S, \mathcal{I}^{-\lfloor t/p\rfloor}\otimes (\mathcal{N}^{\geq i-t}\mathcal{O}_{\widehat{\DD}}/\mathcal{N}^{\geq i+1-t}\mathcal{O}_{\widehat{\DD}}\otimes I_{s-1})\{i\}\right)} & \textrm{if }up^{v_p(e)}\geq e(i+1)\\
\operatorname{R\Gamma}_{\mathrm{QSyn}}\left(S, \mathcal{I}^{-\lfloor t/p\rfloor}\otimes (\mathcal{N}^{\geq i-t}\mathcal{O}_{\widehat{\DD}}/(\phi^{s-1-v_p(e)})^*(I_{v_p(e)}))\{i\}\right) & \textrm{otherwise}\end{cases} \end{align*}
with $t=t(u,p,s-1,e)$ and $i> \frac{p}{p-1} (\fdim(S)-1)$. If $s=0$ then the above sheaves are understood to be zero.
\end{thm}

\begin{proof}
First note that by our assumption and the above Lemmas we have that the map induced by the Frobenius is an equivalence on the terms in Equation~\ref{eqn: quasisyntomic filtration for tc} with $i-t\geq \fdim(S)$. On the terms with $i-t<\fdim(S)$ we may re-express $t=i-a$ for some $0\leq a \leq \fdim(S)-1$. Then the Frobenius multiplies the direct product index by $p$ so it sends a term with shift $t$ to a term with shift at least $pt=pi-pa$. Consequently we have that 
\begin{align*}
    i-t(u,p,m+1,e) &\leq i-pt(u,p,m,e)\\
                    &= i-pi+pa\\
                    &\leq (1-p)i+p\fdim(S)-p\\
                    &< (1-p)\left(\frac{p}{p-1}(\fdim(S)-1)\right)+p(\fdim(S)-1)\\
                    &=0
\end{align*}
and so the target of the Frobenius map on these terms lands in the factors killed when we simplified the associated graded terms of $\tc(S[x]/x^e, (x);\ZZ_p)$ in Equation~\ref{eqn: quasisyntomic filtration for tc}. These are also by definition terms with $m=s-1$.

Consider the map of sheaves 
\begin{align*}
&\mathrm{gr}^i(\tc(S[x]/x^e,(x);\ZZ_p))[-2i-1]\to \\
&\prod_{u\in J_p\setminus e'J_p}
{\operatorname{R\Gamma}_{\mathrm{QSyn}}\left(S, \left(\bigotimes_{l=1}^{t}\mathcal{I}_{v_p(l)}\right)^{-1}\otimes \mathcal{N}^{\geq i-t}\mathcal{O}_{\widehat{\DD}}/\mathcal{N}^{\geq i+1-t}\mathcal{O}_{\widehat{\DD}}\otimes \mathcal{I}_{s-1}\{i\}\right)} \\ 
&\times \prod_{u\in e'J_p}\begin{cases}{\operatorname{R\Gamma}_{\mathrm{QSyn}}\left(S, \left(\bigotimes\limits_{l=1}^{t}\mathcal{I}_{v_p(l)}\right)^{-1}\otimes \mathcal{N}^{\geq i-t}\mathcal{O}_{\widehat{\DD}}/\mathcal{N}^{\geq i+1-t}\mathcal{O}_{\widehat{\DD}}\otimes \mathcal{I}_{s-1}\{i\}\right)} & \textrm{if }up^{v_p(e)}\geq e(i+1)\\
\operatorname{R\Gamma}_{\mathrm{QSyn}}\left(S, \left(\bigotimes\limits_{l=1}^{t}\mathcal{I}_{v_p(l)}\right)^{-1}\otimes \mathcal{N}^{\geq i-t}\mathcal{O}_{\widehat{\DD}}/(\phi^{s-1-v_p(e)})^*(\mathcal{I}_{v_p(e)})\{i\}\right) & \textrm{otherwise}\end{cases} \end{align*}
which is the composition of the map from $\mathrm{gr}^i\tc $ to the source of the fiber sequence followed by the projection map to the indicated terms. This map is injective in each cohomological degree. To see this note that an element of the kernel of the map $\phi_p-can$ is determined by the degree $m=s-1$ terms  since at every other degree the Frobenius is an isomorphism and so the other indices are determined by applying the canonical map and pulling back by the Frobenius. This map is also surjective since in every cohomological degree an element in the $m=s-1$ degree will determine an element of the kernel by pushing forward iteratively applying $can$ and pulling back by the Frobenius. In particular the induced map is a quasi-isomorphism. 

Finally we can use Lemma~\ref{lem: Frobenius simplifies rep spheres} to reduce line bundles appearing in the above formula. Specifically, by Lemma~\ref{lem: Frobenius simplifies rep spheres} inductively if $p\mid l$ then we have equivalences $(\mathcal{I}_1)^{-1}\simeq (\mathcal{I_2})^{-1}\simeq \ldots \simeq (\mathcal{I}_{v_p(l)})^{-1}$ and so may replace the tensor factor $\mathcal{I}_{v_p(l)}$ appearing in the above sheaves with $\mathcal{I}$. In the tensor product appearing in the above sheaves there are $\lfloor \frac{t}{p}\rfloor$ such $l$ with $p\mid l$. If $\nmid l$ then $\mathcal{I}_{v_p(l)}=\mathcal{O}_{\widehat{\DD}}$ and so we may ignore these factors. The result follows.
\end{proof}

From this result we have several Corollaries.

\begin{cor}[Theorem~\ref{thm: main computation for perfectoids}]
Let $R$ be a perfectoid ring and $r\geq 1$. Then \[K_{2r-1}(R[x]/x^e, (x);\ZZ_p)\cong \mathbb{W}_{re}(R)/V_e\mathbb{W}_r(R)\] and the even groups are trivial.
\end{cor}
\begin{proof}
Since $(x)\subseteq R[x]/x^e$ is nilpotent by the Dundas-Goodwillie-McCarthy Theorem \cite[Theorem 7.0.0.1]{Dundas_Goodwillie_McCarthy} to compute the relative $K$-theory it is enough to instead compute the relative topological cyclic homology. For this we then have by Theorem~\ref{thm: tc filtration simplification segal} a complete and exhaustive filtration on $\tc(R[x]/x^e, (x);\ZZ_p)$ with 
\begin{align*}
&\mathrm{gr}^i(\tc(R[x]/x^e,(x);\ZZ_p))[-2i-1]\simeq \\
&\prod_{u\in J_p\setminus e'J_p}
{\operatorname{R\Gamma}_{\mathrm{QSyn}}\left(R, \mathcal{I}^{-\lfloor t/p\rfloor }\otimes \mathcal{N}^{\geq i-t}\mathcal{O}_{\widehat{\DD}}/\mathcal{N}^{\geq i+1-t}\mathcal{O}_{\widehat{\DD}}\otimes \mathcal{I}_{s-1}\{i\}\right)} \\ 
&\times \prod_{u\in e'J_p}\begin{cases}{\operatorname{R\Gamma}_{\mathrm{QSyn}}\left(R, \mathcal{I}^{-\lfloor t/p\rfloor}\otimes \mathcal{N}^{\geq i-t}\mathcal{O}_{\widehat{\DD}}/\mathcal{N}^{\geq i+1-t}\mathcal{O}_{\widehat{\DD}}\otimes I_{s-1}\{i\}\right)} & \textrm{if }up^{v_p(e)}\geq e(i+1)\\
\operatorname{R\Gamma}_{\mathrm{QSyn}}\left(R, \mathcal{I}^{-\lfloor t/p\rfloor}\otimes \mathcal{N}^{\geq i-t}\mathcal{O}_{\widehat{\DD}}/(\phi^{s-1-v_p(e)})^*(I_{v_p(e)})\{i\}\right) & \textrm{otherwise}\end{cases} \end{align*}
which in our case we can compute explicitly. Note first that since $R$ is perfectoid the line bundle $\mathcal{I}$ and the Breuil-Kisin twists are trivial. In particular we can reduce the above to 
\begin{align*}
&\mathrm{gr}^i(\tc(R[x]/x^e,(x);\ZZ_p))[-2i-1]\simeq \\
&\prod_{u\in J_p\setminus e'J_p}
{\mathcal{N}^{\geq i-t}\widehat{\DD}_{R}/(\mathcal{N}^{\geq i+1-t}{\widehat{\DD}_R}\otimes \mathcal{I}_{s-1})} \\ 
&\times \prod_{u\in e'J_p}\begin{cases}{ \mathcal{N}^{\geq i-t}\widehat{\DD}_R/(\mathcal{N}^{\geq i+1-t}{\widehat{\DD}_R}\otimes I_{s-1})} & \textrm{if }up^{v_p(e)}\geq e(i+1)\\
\mathcal{N}^{\geq i-t}{\widehat{\DD}_R}/(\phi^{s-1-v_p(e)})^*(I_{v_p(e)}) & \textrm{otherwise}\end{cases} \end{align*}
and we have that $\mathcal{N}^{\geq i}\widehat{\DD}_R\cong d^iA_{inf}(R)$ and that $I_l\cong \tilde{d}_lA_{inf}(R)=\phi(d)\phi^2(d)\ldots \phi^{l}(d)A_{inf}(R)$. Since $d\in A_{inf}(R)$ is a nonzero divisor we have that 
\begin{align*}
    \mathrm{gr}^i(\tc(R[x]/x^e,(x);\ZZ_p))[-2i-1]\simeq &\prod_{u\in J_p\setminus e'J_p}A_{inf}(R)/d\tilde{d}_{s(u,e(i+1),p)-1}\\ &\times \prod_{u\in e'J_p}\begin{cases} A_{inf}(R)/d\tilde{d}_{s(u,e(i+1),p)-1} & \textrm{if }up^{v_p(e)}\geq e(i+1)\\
    A_{inf}(R)/\phi^{s-1-v_p(e)}(\tilde{d}_{v_p(e)}) &\textrm{otherwise}\end{cases}\\
    &\cong \prod_{u\in J_p\setminus e'J_p}W_{s(u,e(i+1),p)}(R)\\
    &\times \prod_{u\in e'J_p}\begin{cases}W_{s(u,e(i+1),p)}(R) & \textrm{if }up^{v_p(e)}\geq e(i+1)\\
    W_{v_p(e)}(R)&\textrm{otherwise}\end{cases}
\end{align*}
where the second equivalence comes from the fact that the Frobenius is an isomorphism on $A_{inf}(R)$ and $A_{inf}(R)/\tilde{d}_l\cong W_{l}(R)$ by \cite[Lemma 3.12]{BMS1}. Note that this is a $\phi$-linear isomorphism.

Thus the associated graded pieces of $\tc(R[x]/x^e, (x);\ZZ_p)$ are all concentrated in a single degree, so for degree reasons the associated spectral sequence collapses and we have that the even homotopy groups vanish and the odd homotopy groups are given by the product above. Note that the argument given in \cite[Lemma 2]{Speirs_truncated} works verbatim for any $p$-complete ring to identify the above product of $p$-typical witt vectors with the big Witt vector expression given above, and so the result follows.
\end{proof}

The above was already known to experts. Sulyma in private correspondence has informed the author that their methods in \cite{Sulyma_truncated} can be extended to cover the above case as well. The utility of the approach taken in this paper is that in addition to the above we also get two interesting extensions.

\begin{cor}
Let $S$ be a $p$-completely formally smooth curve over a perfectoid ring $R$. Then for $r\geq 2$ there are isomorphisms
\begin{align*}K_{2r-1}(S[x]/x^e,(x);\ZZ_p)\cong &\prod_{u\in J_p\setminus e'J_p}H^0\left(\mathcal{N}^{\geq r-1-t}\DD_S/\tilde{d}_{s}\mathcal{N}^{\geq r-t}\DD_S\right)\\
&\times \prod_{u\in e'J_p}\begin{cases}H^0\left(\mathcal{N}^{\geq r-1-t}\DD_S/\tilde{d}_{s}\mathcal{N}^{\geq r-t}\DD_S\right)&\textrm{if }up^{v_p(e)}\geq er\\
H^0\left(\mathcal{N}^{\geq r-1-t}\DD_S/\phi^{s-v_p(e)}(\tilde{d}_{v_p(e)})\right)&\textrm{otherwise}\end{cases}\end{align*} and \begin{align*}K_{2r-2}(S[x]/x^e,(x);\ZZ_p)\cong &\prod_{u\in J_p\setminus e'J_p}H^1\left(\mathcal{N}^{\geq r-1-t}\DD_S/\tilde{d}_{s}\mathcal{N}^{\geq r-t}\DD_S\right)\\
&\times \prod_{u\in e'J_p}\begin{cases}H^1\left(\mathcal{N}^{\geq r-1-t}\DD_S/\tilde{d}_{s}\mathcal{N}^{\geq r-t}\DD_S\right)&\textrm{if }up^{v_p(e)}\geq er\\
H^1\left(\mathcal{N}^{\geq r-1-t}\DD_S/\phi^{s-v_p(e)}(\tilde{d}_{v_p(e)})\right)&\textrm{otherwise}\end{cases}\end{align*}
\end{cor}

\begin{ex}
Let $R$ be a $p$-torsion free perfectoid ring with orientation $d$. Let $S:=R[t]^\wedge_p$. Then from \cite[Lemma 4.6]{Cusps-Paper} we find that $H^0(\mathcal{N}^{\geq i}\widehat{S})\cong d^i A_{inf}(R)$ and \[H^1(\mathcal{N}^{\geq i}\widehat{\DD}_S)\simeq d^i\left(\bigoplus_{j\geq 1}A_{inf}(R)/d\tilde{d}_{v_p(j)}\right)^\wedge_{(p,d)}.\] Consequently we have that \[H^1(\mathcal{N}^{\geq r-1-t}\DD_S/\tilde{d}_s\mathcal{N}^{\geq r-t}\DD_S)\simeq \left(\bigoplus_{j\geq 1}A_{inf}(R)/d\tilde{d}_{v_p(j)}\right)^\wedge_{(p,d)}/d\tilde{d}_s\cong \left(\bigoplus_{j\geq 1}A_{inf}(R)/d\tilde{d}_{\min\{v_p(j),s\}}\right)^\wedge_p\] and  
\begin{align*}H^1(\mathcal{N}^{\geq r-1-t}\DD_S/\phi^{s-v_p(e)}(d_{v_p(e)}))&\cong \left(\bigoplus_{j\geq 1}A_{inf}(R)/d\tilde{d}_{v_p(j)}\right)^\wedge_{(p,d)}/\phi^{s-v_p(e)}(\tilde{d}_{v_p(e)})\\
&\cong \left(\bigoplus_{j\geq 1}A_{inf}(R)/(d\tilde{d}_{v_p(j)}, \phi^{s-v_p(e)}(\tilde{d}_{v_p(e)})\right)^\wedge_p
\end{align*} where we can remove the $d$-adic completion from both since modulo $p$ we have that all the modules appearing in both direct sums are uniformly bounded $d$-torsion. This completely computes the odd $K$-groups and computes the even $K$-groups up to an extension. 

For the even groups note that $S$ is split as a ring by the maps $R\to R[t]^\wedge_p\to R$ so all the above complexes split as groups by functoriality. The inclusions $\mathcal{N}^{\geq i}\DD_R\to \mathcal{N}^{\geq i}\DD_S$ is an equivalence on $H^0$ and so all of these complexes are formal. Thus the short exact sequences we get describing the $H^0$ terms from the previous paragraph are all canonically split. Therefore \[H^0(\mathcal{N}^{\geq r-1-t}\DD_S/\tilde{d}_s\mathcal{N}^{\geq r-t}\DD_S)\cong A_{inf}(R)/d\tilde{d}_s \oplus \ker\left(d\tilde{d}_s:H^1(\DD_S)\to H^1(\DD_S)\right)\] and there is a similar formula for $H^0(\mathcal{N}^{\geq r-1-t}\DD_S/\phi^{s-v_p(e)}(\tilde{d}_{v_p(e)}))$. Note that we have an explicit formula for $H^1(\DD_S)$ but for clarity and brevity we will leave the computation as it is stated above. 
\end{ex}
\begin{rem}
In the case of $R$ perfectoid with $p$-torsion the computation above still works for the odd $K$-groups. The computation of $H^0(\mathcal{N}^{\geq i}\DD_S)$ used in the above no longer works however. The issue is that the $p^\infty$ torsion in $\bigoplus_{\geq j}A_{inf}(R)/d\tilde{d}_{v_p(j)}$ is unbounded and so the derived $(p,d)$-adic completion will contribute to $H^0(\mathcal{N}^{\geq i}\DD_S)$. We still get summands of the form $A_{inf}(R)/d\tilde{d}_s$ showing up in the even $K$-groups but there is an a priori non-split short exact sequence that must be studied to find the other summand. 
\end{rem}
As a final Corollary we also get Theorem~\ref{thm: computation for F-curves}.

\begin{cor}[Theorem~\ref{thm: computation for F-curves}]
Let $S$ be an F-smooth ring with $\ndim(S)\leq 1$ and $r\geq 2$. Then there are isomorphisms
\begin{align*}
&K_{2r-1}(S[x]/x^e, (x);\ZZ_p)\simeq \\
&\prod_{u\in J_p\setminus e'J_p}
{H^0\left(S, \mathcal{I}^{-\lfloor t/p\rfloor}\otimes \mathcal{N}^{\geq r-t-1}\mathcal{O}_{\widehat{\DD}}/(\mathcal{N}^{\geq r-t}\mathcal{O}_{\widehat{\DD}}\otimes \mathcal{I}_{s-1})\{r-1\}\right)} \\ 
&\times \prod_{u\in e'J_p}\begin{cases}{H^0\left(S, \mathcal{I}^{-\lfloor t/p\rfloor}\otimes \mathcal{N}^{\geq r-t-1}\mathcal{O}_{\widehat{\DD}}/(\mathcal{N}^{\geq r-t}\mathcal{O}_{\widehat{\DD}}\otimes \mathcal{I}_{s-1})\{r-1\}\right)} & \textrm{if }up^{v_p(e)}\geq er\\
H^0\left(S, \mathcal{I}^{-\lfloor t/p\rfloor}\otimes \mathcal{N}^{\geq r-t-1}\mathcal{O}_{\widehat{\DD}}/(\phi^{s-1-v_p(e)})^*(\mathcal{I}_{v_p(e)})\{r-1\}\right) & \textrm{otherwise}\end{cases} \end{align*}
and
\begin{align*}
&K_{2r-2}(S[x]/x^e, (x);\ZZ_p)\simeq \\
&\prod_{u\in J_p\setminus e'J_p}
{H^1\left(S,\mathcal{I}^{-\lfloor t/p\rfloor}\otimes \mathcal{N}^{\geq r-t-1}\mathcal{O}_{\widehat{\DD}}/(\mathcal{N}^{\geq r-t}\mathcal{O}_{\widehat{\DD}}\otimes \mathcal{I}_{s-1})\{r-1\}\right)} \\ 
&\times \prod_{u\in e'J_p}\begin{cases}{H^1\left(S, \mathcal{I}^{-\lfloor t/p\rfloor}\otimes\mathcal{N}^{\geq r-t-1}\mathcal{O}_{\widehat{\DD}}/(\mathcal{N}^{\geq r-t}\mathcal{O}_{\widehat{\DD}}\otimes \mathcal{I}_{s-1})\{r-1\}\right)} & \textrm{if }up^{v_p(e)}\geq er\\
H^1\left(S, \mathcal{I}^{-\lfloor t/p\rfloor}\otimes\mathcal{N}^{\geq r-t-1}\mathcal{O}_{\widehat{\DD}}/(\phi^{s-1-v_p(e)})^*(\mathcal{I}_{v_p(e)})\{r-1\}\right) & \textrm{otherwise}\end{cases} \end{align*}
with $s=s(p,re,u)$. When $s=0$ we define the above groups to be zero.
\end{cor}

These Corollaries both follow from the spectral sequence constructed above collapsing for degree reasons. This is because cohomological bounds produced by the definition of $\ndim(S)\leq 1$ are preserved when taking the tensor product by the line bundles appearing in the spectral sequence, which we prove below.

\begin{lem}\label{lem: cohomology bounds stable with I}
Let $S$ be a quasisyntomic ring. Then $\operatorname{R\Gamma}_{\mathrm{QSyn}}(S;\mathcal{I}^{k}\otimes \mathcal{N}^{\geq i}\mathcal{O}_{\widehat{\DD}}\{j\})\in D^{[0,\ndim(S)]}(\ZZ_p)$ for all $i,j,k\in \ZZ$.
\end{lem}
\begin{proof}
On the associated graded terms of the Nygaard filtration the line bundle $\mathcal{I}$ becomes trivial, in fact in Lemma~\ref{lem: Segal conjecture with line bundles} we showed that it is equivalent to the line bundle $(p)$. By devissage the result follows.
\end{proof}

\section{\texorpdfstring{Comments on the quasisyntomic filtration on $\mathrm{TR}$}{Comments on the quasisyntomic filtration on topological restriction homology}}\label{sec: tr stuff}
This Section is in some sense independent from the previous two Sections. While there are certainly connections via \cite[Theorem A]{McCandless_curves} which we believe should be explored, we have found that the methods developed in this paper more readily apply to a different formulation of topological restriction homology. One interesting question is if one can prove \cite[Theorem A]{McCandless_curves} purely algebraically by combining the previous two sections with what we do in this section.

In order to study the quasisyntomic filtration on topological restriction homology it will first be helpful to have a purely Borel equivariant description of it. We recall the result of McCandless which lets us do this.

\begin{thm}[Remark 2.4.5, \cite{McCandless_curves}]
Let $X$ be a bounded below $p$-cyclotomic spectrum. Then there is a functorial in $X$ equivalence 
\[\tr(X)\simeq \operatorname{Eq}\left(\begin{tikzcd}
\prod\limits_{n\geq 0}X^{hC_{p^n}} \arrow[r, shift left] \arrow[r, shift right] & \prod\limits_{n\geq 0}X^{tC_{p^n}}
\end{tikzcd}\right)\] where the top map is the product over the canonical maps and the bottom map is the Frobenius map $X^{hC_{p^n}}\xrightarrow{\phi_X^{hC_{p^n}}}(X^{tC_p})^{hC_{p^n}}\simeq X^{tC_{p^{n+1}}}$.
\end{thm}

Recall that for us $\tr(-)$ always means the $p$-typical topological restriction homology. If one wants to work with integral $\tr(-;\ZZ_p)$, since everything we consider is $p$-complete, we note that there is a product decomposition $\tr(-;\ZZ_p)\simeq \prod_{J_p}\tr(-)$ where $J_p:=\NN\setminus p\NN$.

In particular on $\mathrm{QRSPerfd}$ we may define the filtration \[\mathrm{Fil}^{\geq *}\tr(-):=\operatorname{Eq}\left(\begin{tikzcd}
\prod\limits_{n\geq 0}\tau_{\geq 2*}\thh(-;\ZZ_p)^{hC_{p^n}} \arrow[r, shift left] \arrow[r, shift right] & \prod\limits_{n\geq 0}\tau_{\geq 2*}\thh(-;\ZZ_p)^{tC_{p^n}}
\end{tikzcd}\right)\] which since each term in the filtration is the limit of quasisyntomic sheaves the whole filtration will be by quasisyntomic sheaves. By Theorem~\ref{thm: qrsp are basis} we then get a filtration of $\tr(-):\mathrm{QSyn}\to \Sp$ by quasisyntomic sheaves. We can identify the associated graded sheaves.

\begin{cor}\label{cor: quasisyntomic associated graded for tr}
For all $i\in \ZZ$ there are equivalences \[\mathrm{gr}^i\tr(-)\simeq \operatorname{Eq}\left(\begin{tikzcd}
\prod\limits_{n\geq 0}(\mathcal{N}^{\geq i}\mathcal{O}_{\widehat{\DD}}/\mathcal{N}^{\geq i+1}\mathcal{O}_{\widehat{\DD}}\otimes \mathcal{I}_{n})\{i\} \arrow[r, shift left] \arrow[r, shift right] & \prod\limits_{n\geq 0}\mathcal{O}_{\widehat{\DD}}\{i\}/\mathcal{I}_{n}
\end{tikzcd}\right)[2i]\]
where the top map is the map induced by the inclusions $\mathcal{N}^{\geq i}\mathcal{O}_{\widehat{\DD}}, \mathcal{N}^{\geq i+1}\mathcal{O}_{\widehat{\DD}}\hookrightarrow \mathcal{O}_{\widehat{\DD}}$ and the bottom map is the map induced by $\phi_i$ on each of the factors. 
\end{cor}
\begin{proof}
    There is, by definition, an identification of $\mathrm{Fil}^{\geq *}\tr(-)$ with the equalizer \[\operatorname{Eq}\left(\begin{tikzcd}
\prod\limits_{n\geq 0}\tau_{\geq 2*}\thh(-;\ZZ_p)^{hC_{p^n}} \arrow[r, shift left] \arrow[r, shift right] & \prod\limits_{n\geq 0}\tau_{\geq 2*}\thh(-;\ZZ_p)^{tC_{p^n}}
\end{tikzcd}\right)\] for quasiregular semiperfecoid inputs, so it follows that \[\mathrm{Fil}^{\geq *}\tr(-)\simeq \operatorname{Eq}\left(\begin{tikzcd}\prod\limits_{n\geq 0}\mathrm{Fil}^{\geq *}\thh(-;\ZZ_p)^{hC_{p^n}} \arrow[r, shift left] \arrow[r, shift right] & \prod\limits_{n\geq 0}\mathrm{Fil}^{\geq}\thh(-;\ZZ_p)^{tC_{p^n}}\end{tikzcd}\right)\] by Theorem~\ref{thm: qrsp are basis}. Therefore since taking associated graded terms commutes with limits there are equivalences
\begin{align*}
    \mathrm{gr}^i\tr(-) &\simeq \operatorname{Eq}\left(\begin{tikzcd}[ampersand replacement=\&]\prod\limits_{n\geq 0}\mathrm{gr}^i\thh(-;\ZZ_p)^{hC_{p^n}}\arrow[r, shift left]\arrow[r, shift right] \& \prod\limits_{n\geq 0}\mathrm{gr}^i\thh(-;\ZZ_p)^{tC_{p^n}}\end{tikzcd}\right)\\
    &\simeq \operatorname{Eq}\left(\begin{tikzcd}[ampersand replacement=\&]
\prod\limits_{n\geq 0}(\mathcal{N}^{\geq i}\mathcal{O}_{\widehat{\DD}}/\mathcal{N}^{\geq i+1}\mathcal{O}_{\widehat{\DD}}\otimes \mathcal{I}_{n})\{i\} \arrow[r, shift left] \arrow[r, shift right] \& \prod\limits_{n\geq 0}\mathcal{O}_{\widehat{\DD}}\{i\}/\mathcal{I}_{n}
\end{tikzcd}\right)[2i]
\end{align*}
where the second equivalence comes from Lemma~\ref{lem: Associated graded terms at finite levels}.
\end{proof}

Notice that the top map in the above Corollary preserves the product index and the bottom map  multiplies the product index by $p$. Consequently whenever the top map is an equivalence the sheaves $\mathrm{gr}^i\tr(-)$ will vanish since we can chase any error terms introduced by the bottom map to infinity. In particular $\mathrm{gr}^i\tr(-)\simeq 0$ whenever $i<0$. In addition we note that Lemma~\ref{lem: Tsaladis's algebraic theorem} and Lemma~\ref{lem: Segal conjecture with line bundles} apply to the bottom map in the equalizer diagram above as well. 

\begin{con}
Let $S$ be an F-smooth ring with $\fdim(S)\leq k$. Then we may define maps \[\theta^{\widehat{\DD}_S}_{n+1,n}\{i\}:\widehat{\DD}_S/I_{n+1}\{i\}\to \widehat{\DD}_S/I_{n}\{i\}\] for all $i\geq k$ as the composition \[\widehat{\DD}_S/I_{n+1}\{i\}\xrightarrow{\operatorname{R\Gamma}_{\mathrm{QSyn}}(S;\phi_i)^{-1}}\operatorname{R\Gamma}_{\mathrm{QSyn}}(S;\mathcal{N}^{\geq i}\mathcal{O}_{\widehat{\DD}}/\mathcal{N}^{\geq i+1}\mathcal{O}_{\widehat{\DD}}\otimes \mathcal{I}_{n}\{i\})\xrightarrow{can}\widehat{\DD}_S/I_{n}\{i\}\] which is well defined by our assumption on the F-dimension of $S$ and Lemma~\ref{lem: Tsaladis's algebraic theorem} and Lemma~\ref{lem: Segal conjecture with line bundles}. 
\end{con}

\begin{rem}
To justify our notation, take $R$ to be a perfectoid ring and consider the maps $\theta^{\widehat{\DD}_R}_{n+1, n}:W_{n+1}(R)\to W_n(R)$. These are exactly given by the restriction maps by the first diagram in \cite[Lemma 3.4]{BMS1}. In particular by induction the map $A_{inf}(R)\to W_n(R)$ obtained by taking the limit of these maps is exactly the map $\theta_n$ from $p$-adic Hodge theory. The higher Breuil-Kisin twisted maps are given by this map multiplied by a corresponding Frobenius twisted product of the orientation.
\end{rem}

\begin{defn}
Let $S$ be an F-smooth ring with $\fdim(S)\leq j$. For all $i\geq j$ define \[\theta^{\widehat{\DD}_S}_{\infty}(i):=\operatorname{\lim}_{\theta^{\widehat{\DD}_S}_{n+1,n}\{i\}}\widehat{\DD}_S/I_n\{i\}\] as an object of $D(\ZZ_p)$. We will use the shorthand notation $\theta^{\widehat{\DD}_S}_\infty := \theta^{\widehat{\DD}_S}_\infty(0)$.
\end{defn}

\begin{ex}
Let $R$ be a perfectoid ring with orientation $d$. We will use the shorthand notation  $d_n$ for the product $d\phi^{-1}(d)\ldots \phi^{-n+1}(d)$. Then there are isormorphisms \[\theta^{\widehat{\DD}_R}_{\infty}(i)\cong \operatorname{\lim}\left(\ldots\to A_{inf}(R)/\tilde{d}_{n+1}\xrightarrow{\phi^{-1}(-)\times d^i}A_{inf}(R)/\tilde{d}_{n}\to\ldots\right)\] so for example when $R$ is a perfect $\FF_p$-algebra we have that $\theta^{\widehat{\DD}_R}_\infty = A_{inf}(R)$ and $\theta^{\widehat{\DD}_R}_\infty(i)=0$ for all $i>0$. In general note that there is an isomorphism of diagrams \[\begin{tikzcd}
\ldots \arrow[r] & A_{inf}(R)/\tilde{d}_{n+1} \arrow[r, "\phi^{-1}(-)\times d^i"] \arrow[d, "\phi^{-n-1}"] & A_{inf}(R)/\tilde{d}_n \arrow[r] \arrow[d, "\phi^{-n}"] & \ldots \\
\ldots \arrow[r] & A_{inf}(R)/d_{n+1} \arrow[r, "-\times \phi^{-n}(d)^i"]                & A_{inf}(R)/\tilde{d}_n \arrow[r]             & \ldots
\end{tikzcd}\] so we will compute the limit of the bottom inverse system. This inverse system is the diagonal of mod $(d_\bullet)$ system as columns and multiplication by $\phi^{-\bullet}(d)^i$ rows. Since the limit of an $\NN\times \NN$-index system agrees with the limit along the diagonal we may compute the $\NN\times \NN$ limit instead. This then agrees with the limit along the first factor followed by the limit along the second factor.

Note first that $\operatorname{R\lim} A_{inf}(R)/d_{\bullet}\cong A_{inf}(R)$. This is because the limit of $A_{inf}(R)/d_{\bullet}$ is the same as the limit of $W_{\bullet}(R)$ along the restriction maps and the transition maps are surjective so there is no $\lim^1$ term. Consequently we have that \[\theta^{\widehat{\DD}_R}_\infty (i)\simeq \operatorname{\lim}\left(\ldots \to A_{inf}(R)\xrightarrow{-\times \phi^{-\bullet}(d)^i}A_{inf}(R)\to\ldots\right)\] as complexes.

To see what the limit is consider the exact sequence of inverse systems
\[
\begin{tikzcd}
\vdots \arrow[d] & \vdots \arrow[d] &\vdots \arrow[d]\\
A_{inf}(R) \arrow[r, "\times d_3^i"] \arrow[d, "\times \phi^{-2}(d)^i"] & A_{inf}(R) \arrow[d, "id"] \arrow[r] & A_{inf}(R)/(d_3)^i\arrow[d]\\
A_{inf}(R) \arrow[r, "\times d_2^i"] \arrow[d, "\times \phi^{-1}(d)^i"] & A_{inf}(R) \arrow[d, "id"] \arrow[r] & A_{inf}(R)/(d_2)^i\arrow[d]\\
A_{inf}(R) \arrow[r, "\times d_1^i"] & A_{inf}(R)  \arrow[r] & A_{inf}(R)/(d)^i
\end{tikzcd}
\]
The middle system has no $\lim^1$ term and so we have an identification \[\theta^{\widehat{\DD}_R}_{\infty}(i)\simeq \operatorname{fib}(A_{inf}(R)\to \lim A_{inf}(R)/(d_{\bullet})^i)\] as complexes. When $i=1$ this map is $\theta_{\infty}:A_{inf}(R)\to W(R)$.
\end{ex}

This complex, in addition to having the above connection to $p$-adic Hodge theory, can be used to compute topological restriction homology.

\begin{thm}
Let $S$ be a F-smooth ring with $\fdim(S)\leq j$. Then for all $i\geq j$ there is an equivalence \[\mathrm{gr}^i\tr(S)\simeq \theta^{\widehat{\DD}_S}_{\infty}(i)[2i]\] of complexes.
\end{thm}
\begin{proof}
Upon evaluating at $S$ we may rewrite the equalizer diagram in Corollary~\ref{cor: quasisyntomic associated graded for tr} as
\[\mathrm{gr}^i\tr(S)\simeq \operatorname{Eq}\left(\begin{tikzcd}
\prod\limits_{n\geq 0}(\mathcal{N}^{\geq i}{\widehat{\DD}_S}/\mathcal{N}^{\geq i+1}{\widehat{\DD}_S}\otimes I_{n})\{i\} \arrow[r, shift left] \arrow[r, shift right] & \prod\limits_{n\geq 1}{\widehat{\DD}_S}\{i\}/I_{n}
\end{tikzcd}\right)[2i]\] and we know that on each factor the bottom map is an equivalence. Thus up to equivalence of diagrams we may replace this equalizer diagram with \[\operatorname{Eq}\left(\begin{tikzcd}
\prod\limits_{n\geq 0}({\widehat{\DD}_S}/I_{n})\{i\} \arrow[r, shift left] \arrow[r, shift right] & \prod\limits_{n\geq 1}{\widehat{\DD}_S}\{i\}/I_{n}
\end{tikzcd}\right)[2i]\] where the bottom map is the identity and the top map is the product of the maps $\theta_{n+1,n}^{\widehat{\DD}_S}\{i\}$ and is the zero map on the last factor. In other words we have an inverse system $\ldots X_{n+1}\xrightarrow{f_{n+1}} X_{n}\to\ldots$ and we are taking the fiber of the map $\prod_{n\in \NN}X_n\to \prod_{n\in \NN}X_n$ given by $1-\prod_{n\in \NN}f_n$ which is exactly the derived limit of the system. 
\end{proof}

\begin{proof}[Proof of Theorem~\ref{thm: main calculation for tr}]
We have already shown the first part of this Theorem, so we will now prove parts 2 and 3. It will be helpful to have the cohomological bounds on the complexes $\theta^{\widehat{\DD}_S}_\infty(i)$ in the proof of part 2, so we will prove part 3 first. To see the bounds on the cohomology, first notice that by Lemma~\ref{lem: cohomology bounds stable with I} we have that $I_n\{i\},\widehat{\DD}_S\{i\}\in D^{[0,\ndim(S)]}(\ZZ_p)$ for all $n\in \NN$ and $i\in \ZZ$. Consequently by the long exact sequence the cofbers $\widehat{\DD}_S/I_n\{i\}\in D^{[0,\ndim(S)]}(\ZZ_p)$ for all $n\in \NN$ and $i\in \ZZ$. From the Milnor short exact sequence it then follows that $\theta^{\widehat{\DD}_S}_\infty(i)\in D^{[0,\ndim(S)+1]}(\ZZ_p)$ whenever it is defined. 

To see part 2,  note that the spectral sequence associated to $\mathrm{Fil}^{\geq *}\tr(S)$ collapses when $\ndim(S)\leq 1$. This is because each $\theta^{\widehat{\DD}_S}_\infty (i)$ is in $D^{[0,2]}(\ZZ_p)$ for $i\geq 1$ and $\mathrm{gr}^0\tr(S)\in D^{\geq 0}(\ZZ_p)$ since quasisyntomic locally is it discrete. Thus for degree reasons there can be no differentials. The statement for the odd homotopy groups follows, and there are short exact sequences \[0\to H^2(\theta^{\widehat{\DD}_S}_\infty(i+1))\to \tr_{2i}(S;\ZZ_p)\to H^0(\theta^{\widehat{\DD}_S}_\infty(i))\to 0\] for all $i\geq 1$.
\end{proof}

We close this Section with a note for when $S$ is quasiregular semiperfectoid. Note that in this case $\mathrm{gr}^i\tr(S)$ is given by the equalizer of discrete groups. In particular by Theorem~\ref{thm: qrsp are basis} we have that \[\mathrm{gr}^i\tr(-)\simeq \tau^{\mathrm{QSyn}}_{[2i-1,2i]}\tr(-)\] since these agree on quasiregular semiperfectoid rings. Thus $\tr(-)$ is quasisyntomic locally even if and only if $\mathrm{gr}^i\tr(-)$ is quasisyntomic locally discrete. 

\section{Proof of Theorem~\ref{thm: AGH theorem}}\label{sec: agh proof}
This Section is dedicated to a proof of Theorem~\ref{thm: AGH theorem} which we record in this Section as Theorem~\ref{thm: AGH theorem redux}. We will first make a change of notation which we wish to emphasize: in the rest of the paper, we used $e$ for the degree of nilpotence we were considering. This was done to match notation with previous work on the subject. In this situation we will be working with complete discrete valuation rings where it is customary to denote the ramification degree by $e$. Thus for this Section:\[n=\textrm{ the exponent of } x \textrm{ we are modding out by}\] and \[e=\textrm{ the ramification of }A.\] 

To prove Theorem~\ref{thm: AGH theorem} for a given ring $A$ satisfying the conditions of the Theorem we will need some control on $\mathcal{N}^i{\widehat{\DD}}_{A}$. In the cases we consider we can in fact get a precise computation of these groups by importing the computations of $\thh(A;\ZZ_p)$ appearing either as \cite[Theorem 5.1]{Lindenstrauss_Madsen} in the case when $A$ has finite residue field or \cite[Theorem 4.4]{Krause_Nikolaus} when the residue field of $A$ is perfect. 

In the proof, we do not need to know any of the groups $\mathcal{N}^{i}\widehat{\DD}_A$ apart from $\mathcal{N}^0\widehat{\DD}_A=A$. What we need are the following: 
\begin{enumerate} 
    \item each $\mathcal{N}^i\DD_A$ is concentrated in homological degree $-1$ for $i\geq 1$, and 
    \item $|H^1(\mathcal{N}^i\DD_A)|=|A/E'(\pi)|\cdot|A/i|$ where $\pi$ is the uniformizer of $A$ and $E(x)$ is the Eisenstein polynomial of $\pi$. 
\end{enumerate} 

We can in fact prove these facts without knowing what the topological Hochschild homology of $A$ is, and we can get these results purely algebraically. The following Lemma was worked out with Ayelet Lindenstrauss.

\begin{lem} \label{lem: thh(A) purely algebraically}

Let $A$ be a CDVR of mixed characteristic $(0,p)$ and with perfect residue field $k$. Let $\pi$ be the uniformizer of $A$ and let $E(x)$ be its Eisenstein polynomial. Then  
\begin{enumerate} 
    \item each $\mathcal{N}^i\DD_A$ is concentrated in homological degree $-1$ for $i\geq 1$, and 
    \item $|H^1(\mathcal{N}^i\DD_A)|=|A/E'(\pi)|\cdot|A/i|$ where $\pi$ is the uniformizer of $A$ and $E(x)$ is the Eisenstein polynomial of $\pi$. 
\end{enumerate} 

\end{lem} 

\begin{proof} 

We will first show that $\ndim(A)\leq 1$. Taking $A$ as a $\ZZ[x]$-algebra with $x\mapsto \pi$ there is a fiber sequence \[\left(\bigoplus_{i\geq 0} \mathcal{N}^i\DD_A\right)^\wedge_p/\pi\to \left(\bigoplus_{i\geq 0}\mathcal{N}^i\DD_{k}\right)\to \left(\bigoplus_{i\geq 0}\mathcal{N}^i\DD_{k}\right)\] by \cite[Proposition 3.13]{fsmooth_paper}. Since $k$ is perfect both the middle and right terms in this fiber sequence are discrete, so the left-hand side has homotopy concentrated in degrees $[-1,0]$. Since the left-hand side is $p$-complete and $\pi^e$ is a unit times $p$ it follows that $\left(\bigoplus_{i\geq 0} \mathcal{N}^i\DD_A\right)^\wedge_p$ has homotopy concentrated in degrees $[-1,1]$. Consequently, each of the summands $\mathcal{N}^i\DD_A$ have homotopy concentrated in degrees $[-1,1]$, and the Nygaard associated graded terms are always coconnective so $\ndim(A)\leq 1$. 

To show the desired properties for $\mathcal{N}^i\DD_A$ we must first understand the complexes $\mathbb{L}\Omega^i_A$. Define $C(-):=\bigoplus_{i\in \ZZ}\mathbb{L}\Omega^i_{-/\ZZ_p}[-i]$. Recall from \cite[Appendix B]{Bhatt_Lurie} that this construction is in fact an $\mathbb{E}_\infty$-algebra in $\operatorname{D}(\ZZ_p)$, and as a functor $C:\mathrm{CAlg}^{an}\to \mathrm{CAlg}(\operatorname{D}(\ZZ_p))$ commutes with all small colimits. In particular the identification $A\cong W(k)[x]\otimes_{W(k)[y]}W(k)$, where $y\mapsto E(x)$ in $W(k)[x]$ and $y\mapsto 0$ in $W(k)$, gives an equivalence \[C(A)\simeq C(W(k)[x])\otimes_{C(W(k)[y])}W(k)\] as $\mathbb{E}_\infty$-algebras. Since $W(k)[x]$ is ind-smooth over $\ZZ_p$ we have that \[C(W(k)[x])\simeq \bigoplus_{i\geq 0}\Omega^i_{W(k)[x]/\ZZ_p}[-i]\cong W(k)[x,dx]\] where $|dx|=-1$ and $(dx)^2=0$. The map $W(k)[y,dy]\to W(k)[x,dx]$ is given by $y\mapsto E(x)$ and $dy\mapsto E'(x)dx$. Consequently \[C(A)\simeq W(k)[x,dx]\otimes_{W(k)[y,dy]}W(k)\simeq A[dx]\otimes_{W(k)[dy]}W(k)\simeq A\oplus (\bigoplus_{i\geq 1}A/E'(\pi))[-1]\] and since all the maps in question are graded this identifies \[\mathbb{L}\Omega^i_{A/\ZZ_p}\simeq A/E'(\pi)[i-1]\] for all $i\geq 1$.  

Let $\widehat{\Omega}^{\slashed{D}}_A$ be the diffracted Hodge cohomology of \cite[Notation 4.7.12]{Bhatt_Lurie}, and let $\mathrm{Fil}^{cong}_j\widehat{\Omega}^{\slashed{D}}_A$ be the conjugate filtration as constructed in \cite[Construction 4.7.1]{Bhatt_Lurie}. Then $\mathrm{gr}_j^{conj}\widehat{\Omega}^{\slashed{D}}$ is given by $\mathbb{L}\Omega^j_{A}[-j]$ and so $\widehat{\Omega}^{\slashed{D}}_A$ is concentrated in homological degrees $[-1,0]$ with $H^0(\widehat{\Omega}^{\slashed{D}}_A)\cong A$. There are also fiber sequences \[\mathcal{N}^i\DD_A\to \mathrm{Fil}_{i}^{conj}\widehat{\Omega}^{\slashed{D}}_A\xrightarrow{\Theta-i}\mathrm{Fil}^{conj}_{i-1}\widehat{\Omega}^{\slashed{D}}_A\] by \cite[Remark 5.5.8]{Bhatt_Lurie}. Here $\Theta$ is the Sen operator which on $\mathrm{gr}^{conj}_j\widehat{\Omega}^{\slashed{D}}_A$ acts by multiplication by $j$. Thus there are exact sequences \[\begin{tikzcd} 
0 \arrow[r] & H^0(\mathcal{N}^i\DD_A) \arrow[r] & A \arrow[r, "\times i"]                                              & A \arrow[lld, out=-30]  &   \\ 
            & H^1(\mathcal{N}^i\DD_A) \arrow[r] & H^1(\mathrm{Fil}_i^{conj}\widehat{\Omega}^{\slashed{D}}_A) \arrow[r] & H^1(\mathrm{Fil}_{i-1}^{conj}\widehat{\Omega}^{\slashed{D}}_A) \arrow[r] & 0 
\end{tikzcd}\] for all $i\geq 0$ where the last map is surjective since $H^2(\mathcal{N}^i\DD_A)=0$ since $\ndim(A)\leq 1$. Thus for $i\geq 1$,  $H^0(\mathcal{N}^i\DD_A)=0$ and there is an exact sequence \[0\to A/i\to H^1(\mathcal{N}^i\DD_A)\to H^1(\mathrm{Fil}_{i}^{conj}\widehat{\Omega}^{\slashed{D}}_A) \to H^1(\mathrm{Fil}_{i-1}^{conj}\widehat{\Omega}^{\slashed{D}}_A)\to 0.\] Since $\mathrm{gr}^{conj}_j(\widehat{\DD}^{\slashed{D}}_A)\simeq A/E'(\pi)[-1]$ for all $j\geq 1$ it follows that $|H^1(\mathrm{Fil}_{i}^{conj}\widehat{\Omega}^{\slashed{D}}_A)|=|A/E'(\pi)|^{i}$ and so the exact sequences above gives that $|H^1(\mathcal{N}^i\DD_A)|=|A/i|\cdot|A/E'(\pi)|$ as desired.  

\end{proof}

\begin{lem}\label{lem: only junk in odd degrees} 

Let $A$ be a CDVR of mixed characteristic $(0,p)$ and perfect residue field. Let $\pi$ be a uniformizer of $A$ with Eisenstein polynomial $E(x)$. Finally, let $i, s\in \NN$ and $u\in J_p$ be such that $i-t(u,p,s-1,n)>0$. Then \[\operatorname{R\Gamma}_{\mathrm{qSyn}}(A; \mathcal{I}^{-\lfloor t/p\rfloor}\otimes(\mathcal{N}^{\geq i-t}\mathcal{O}_{\widehat{\DD}}/\mathcal{N}^{\geq i+1-t}\mathcal{O}\otimes \mathcal{I}_{s-1})\{i\})\] is $p$-power torsion concentrated in homological degree $-1$. If in addition $|A/\pi|<\infty$ then the order of this group is given inductively by the formula \begin{align*} 
v_p(|H^1(A; \mathcal{I}^{-\lfloor t/p\rfloor}\otimes(\mathcal{N}^{\geq i-t}\mathcal{O}_{\widehat{\DD}}&/\mathcal{N}^{\geq i+1-t}\mathcal{O}\otimes \mathcal{I}_{s-1})\{i\})|)\\ &=v_p(|A/\pi|)(v_\pi(E'(\pi))+e(s-1)+ev_p(i-t))+\\ 
&v_p(|H^1(A; \mathcal{I}^{-\lfloor t'/p\rfloor}\otimes(\mathcal{N}^{\geq i-t'}\mathcal{O}_{\widehat{\DD}}/\mathcal{N}^{\geq i+1-t'}\mathcal{O}\otimes \mathcal{I}_{s-2})\{i\})|) 
\end{align*} 
with $t'=t(u,p,s-2,n)<t$ as in \cite[Theorem B(ii)]{AGH}. 

\end{lem} 

\begin{proof} 

We will prove this by induction on $s$. When $s=0$ this is by definition trivial. It will also be helpful to handle the case $s=1$ separately. When $s=1$ this is equivalent to the complex $\mathcal{N}^{i-t}\DD_{A}$ and by our assumptions on the integers $i$ and $t$ together with the previous Lemma the result follows.

Now suppose that for some $s\in \NN$ we have the above result. We then have a fiber sequence \[ 
\mathcal{I}^{-\lfloor t/p\rfloor}\otimes \mathcal{N}^{i-t}\mathcal{O}_{\DD}\otimes \mathcal{I}_s\{i\}\to \mathcal{I}^{-\lfloor t/p\rfloor}\otimes (\mathcal{N}^{\geq i-t}\mathcal{O}_{\DD}/(\mathcal{N}^{\geq i-t-1}\mathcal{O}_\DD\otimes \mathcal{I}_s)\{i\})\to \mathcal{I}^{-\lfloor t/p\rfloor}\otimes (\mathcal{N}^{\geq i-t}\mathcal{O}_\DD/(\mathcal{N^{\geq i-t}\mathcal{O}_\DD\otimes \mathcal{I}_s})\{i\}) 
\] coming from factoring $\mathcal{N}^{\geq i-t-1}\mathcal{O}_\DD\otimes \mathcal{I}_s\to \mathcal{N}^{\geq i-t}\mathcal{O}_\DD$ through $\mathcal{N}^{\geq i-t}\mathcal{O}_\DD\otimes \mathcal{I}_s$. By the above and the fact that the line bundles $\mathcal{I}$ and $\mathcal{O}_\DD\{1\}$ are trivial after tensoring with $\mathcal{N}^i\mathcal{O}_{\DD}$ we are reduced to showing the result for $\mathcal{I}^{-\lfloor t/p\rfloor}\otimes \mathcal{N}^{\geq i-t}\mathcal{O}_{\DD}\{t\}/(\mathcal{N}^{\geq i-t}\mathcal{O}_{\DD}\otimes \mathcal{I}_{s})\{i\}$. Assuming that this is true we then also have from the induced short exact sequence that 
\begin{align*} 
    v_p(|H^1(A; \mathcal{I}^{-\lfloor t/p\rfloor}\otimes(\mathcal{N}^{\geq i-t}\mathcal{O}_{\widehat{\DD}}&/\mathcal{N}^{\geq i+1-t}\mathcal{O}\otimes \mathcal{I}_{s})\{i\})|)\\ 
    &=v_p(|H^1(\mathcal{N}^{i-t}\DD_A)|)\\ 
    &+v_p(|H^1(\mathcal{O}^\wedge_p; \mathcal{I}^{-\lfloor t/p\rfloor}\otimes(\mathcal{N}^{\geq i-t}\mathcal{O}_{\widehat{\DD}}/\mathcal{N}^{\geq i-t}\mathcal{O}\otimes \mathcal{I}_{s})\{i\})|)\\ 
    &= v_p(|A/i-t|)+v_p(|A/E'(\pi)|)\\&+v_p(|H^1(\mathcal{O}^\wedge_p; \mathcal{I}^{-\lfloor t/p\rfloor}\otimes(\mathcal{N}^{\geq i-t}\mathcal{O}_{\widehat{\DD}}/\mathcal{N}^{\geq i-t}\mathcal{O}\otimes \mathcal{I}_{s})\{i\})|)\\ 
    &= v_p(|A/\pi|)(ev_p(i-t)+v_\pi(E'(\pi)))\\ 
    &+v_p(|H^1(\mathcal{O}^\wedge_p; \mathcal{I}^{-\lfloor t/p\rfloor}\otimes(\mathcal{N}^{\geq i-t}\mathcal{O}_{\widehat{\DD}}/\mathcal{N}^{\geq i-t}\mathcal{O}\otimes \mathcal{I}_{s})\{i\})|).
\end{align*}

Notice that there are maps of fiber sequences \[\begin{tikzcd} 
\mathcal{N}^{\geq i-t}\mathcal{O}_{\DD}\otimes \mathcal{I}_{s}\{i\} \arrow[d]\arrow[r] & \mathcal{N}^{\geq i-t}\mathcal{O}_{\DD}\{i\}\arrow[d] \arrow[r] & \mathcal{N}^{\geq i-t}\mathcal{O}_{\DD}\{i\}/\mathcal{N}^{\geq i-t}\mathcal{O}_{\DD}\otimes \mathcal{I}_{s}\{i\}\arrow[d]\\ 
\mathcal{N}^{\geq i-t-1}\mathcal{O}_{\DD}\otimes \mathcal{I}_{s}\{t\} \arrow[r] & \mathcal{N}^{\geq i-1-t}\mathcal{O}_{\DD}\{i\} \arrow[r] & \mathcal{N}^{\geq i-1-t}\mathcal{O}_{\DD}\{i\}/\mathcal{N}^{\geq i-1-t}\mathcal{O}_{\DD}\otimes \mathcal{I}_{s}\{i\} 
\end{tikzcd}\] and so there is a fiber sequence of the vertical cofibers which remains true after tensoring with $\mathcal{I}^{-\lfloor t/p\rfloor}$. The cofiber of the left and middle maps are both equivalent to $\mathcal{N}^{i-1-t}\mathcal{O}_{\DD}$ and the induced map between them is multiplication by $p^{s}$. To see this note that by Theorem~\ref{thm: qrsp are basis} it is enough to show this on quasiregular semiperfectoid rings where $\mathcal{N}^{i-1-t}\DD_S$ are naturally $\DD_S/d$-modules and $\mathcal{I}_{s}=(\phi(d)\ldots \phi^{s}(d))= (p^{s})\mod d$. In particular for all $j\in \NN$ there are fiber sequences \[\mathcal{N}^{\geq j}\mathcal{O}_{\DD}\{i\}/\mathcal{N}^{\geq j}\mathcal{O}_{\DD}\otimes \mathcal{I}_{s}\{i\}\to \mathcal{N}^{\geq j-1}\mathcal{O}_{\DD}\{i\}/\mathcal{N}^{\geq j-1}\mathcal{O}_{\DD}\otimes \mathcal{I}_{s}\{i\}\to \mathcal{N}^{j-1}\mathcal{O}_{\DD}/p^{s}\] and so inductively it is enough to show the desired statement for $j=0$.

With the exception of $j=0$ we also have that $\mathcal{N}^j\DD_A/p^s$ is concentrated in degrees $0$ and $1$ with $v_p(|H^0|)=v_p(|H^1|)$. This is because $\mathcal{N}^j\DD_A$ is concentrated in degree $1$ and so the cofiber of a map can only pick up new cohomology in degree $0$, and the fact that the cardinalities are equal follows from the long exact sequence. In particular the orders of the first homology groups $\mathcal{I}^{-\lfloor t/p\rfloor}\otimes \mathcal{N}^{\geq j}\mathcal{O}_\DD/(\mathcal{N}^{\geq j}\mathcal{O}_\DD\otimes \mathcal{I}_s)\{i\}$ are all the same with the exception of $j=0$. We also have that $N^0\DD_{A}/p^s=A/p^s$ and therefore  
\begin{align*} 
    v_p(|H^1(A; \mathcal{I}^{-\lfloor t/p\rfloor}\otimes&(\mathcal{N}^{\geq i-t}\mathcal{O}_{\widehat{\DD}}/\mathcal{N}^{\geq i-t}\mathcal{O}\otimes \mathcal{I}_{s})\{i\})|)=\\ 
    & v_p(|A/\pi|)(es)+v_p(|H^1(A; \mathcal{I}^{-\lfloor t/p\rfloor}\otimes(\mathcal{O}_{\widehat{\DD}}/ \mathcal{I}_{s})\{i\})|).
\end{align*}

From Lemma~\ref{lem: Tsaladis's algebraic theorem}, Lemma~\ref{lem: Segal conjecture with line bundles}, the proof of Theorem~\ref{thm: tc filtration simplification segal}, and since we are assuming $i-t\geq 1$, the Frobenius \[\operatorname{R\Gamma}_{\mathrm{qSyn}}(A; \mathcal{I}^{-\lfloor t'/p\rfloor}\mathcal{N}^{\geq i-t'}\mathcal{O}_{\DD}/\mathcal{N}^{\geq i-t'+1}\mathcal{O}_{\DD}\otimes \mathcal{I}_{s-1}\{i\})\to \operatorname{R\Gamma}_{\mathrm{qSyn}}(A; \mathcal{I}^{-\lfloor t/p\rfloor}\otimes \mathcal{O}_{\DD}/\mathcal{I}_{s}\mathcal{O}_{\DD}\{i\})\] is an equivalence. This satisfies the inductive hypothesis since $i-t'\geq i-t$ and so the result follows. 

\end{proof}

We have now handled the majority of the cases appearing in Theorem~\ref{thm: computation for F-curves}, all the remains are the sheaves appearing in the final product. We handle those in the following Lemma.

\begin{lem}\label{lem: only junk in odd degrees coker of v} 

    Let $i,s\in \mathbb{N}$ and $u\in J_p$ be such that $i-t(u,p,s-1,n)>0$ and $s>v_p(n)$. Then \[\operatorname{R\Gamma}_{\mathrm{qSyn}}(A; \mathcal{I}^{-\lfloor t/p\rfloor}\otimes (\mathcal{N}^{\geq i-t}\mathcal{O}_\DD/(\phi^{s-1-v_p(n)})^*\mathcal{I}_{v_p(n)})\{i\})\] is $p$-power torsion concentrated in homological degree $-1$.   

\end{lem}

\begin{proof} 

    Note that when $s-1=v_p(e)$ then this result was already proven in the proof of Lemma~\ref{lem: only junk in odd degrees}, it was the first reduction made. This will be the base case of an inductive proof.

    To this end, suppose the statement is true for some $s$. Then there are maps of fiber sequences  
\[
    \begin{tikzcd} 
\mathcal{N}^{\geq i-t}\mathcal{O}_{\DD}\otimes (\phi^{s-v_p(n)})^*\mathcal{I}_{v_p(n)}\{i\} \arrow[d]\arrow[r] & \mathcal{N}^{\geq i-t}\mathcal{O}_{\DD}\{i\}\arrow[d] \arrow[r] & \mathcal{N}^{\geq i-t}\mathcal{O}_{\DD}\{i\}/(\phi^{s-v_p(n)})^*\mathcal{I}_{v_p(n)}\{i\}\arrow[d]\\ 
\mathcal{N}^{\geq i-t-1}\mathcal{O}_{\DD}\otimes (\phi^{s-v_p(n)})^*\mathcal{I}_{v_p(n)}\{i\} \arrow[r] & \mathcal{N}^{\geq i-1-t}\mathcal{O}_{\DD}\{i\} \arrow[r] & \mathcal{N}^{\geq i-1-t}\mathcal{O}_{\DD}\{i\}/ (\phi^{s-v_p(n)})^*\mathcal{I}_{v_p(n)}\{i\} 
\end{tikzcd}\]
as in the previous proof inductively reduces us to showing that the sheaves \[\mathcal{I}^{-\lfloor t/p\rfloor}\otimes\mathcal{O}_\DD/((\phi^{s-v_p(n)})^*\mathcal{I}_{v_p(n)})\{i\}\] have cohomology concentrated in degree $1$ and are $p$-power torsion in those degrees. These sheaves are the target of the Frobenius maps of the forms \[\begin{tikzcd}\operatorname{R\Gamma}_{\mathrm{qSyn}}(A;\mathcal{I}^{-\lfloor t'/p\rfloor}\otimes\mathcal{N}^{\geq i-t'}\mathcal{O}_\DD/((\phi^{s-1-v_p(n)})^*\mathcal{I}_{v_p(n)})\{i\})\arrow[d]\\ 
\operatorname{R\Gamma}_{\mathrm{qSyn}}(A;\mathcal{I}^{-\lfloor t/p\rfloor}\otimes\mathcal{O}_\DD/((\phi^{s-1-v_p(n)})^*\mathcal{I}_{v_p(n)})\{i\})\end{tikzcd}\] which by Lemma~\ref{lem: Segal conjecture with line bundles} and the proof of Theorem~\ref{thm: tc filtration simplification segal} is an equivalence. The result follows.  

\end{proof}

We are now ready to prove Theorem~\ref{thm: AGH theorem}.

\begin{thm}\label{thm: AGH theorem redux} 

Let $n$ be a positive integer, $i$ a non-negative integer, and $A$ a CDVR of mixed characteristic $(0,p)$, uniformizer $\pi$ with Eisenstein polynomial $E(x)$, perfect residue field $k$, and ramification index $e$. Then there are isomorphisms \[K_{2i+1}(A[x]/x^n, (x);\ZZ_p)\simeq A^{n-1}.\] If in addition $k$ is finite then \[v_p(|K_{2i}(A[x]/x^n, (x);\ZZ_p)|)=ev_p(|k|)v_p((ni)!(i!)^{(n-2)})+v_p(k)v_\pi(E'(\pi))(ni-i)\] and since these groups are $p$-torsion this determines the cardinality. 

\end{thm} 

\begin{proof}[Proof of Theorem~\ref{thm: AGH theorem}] 

We will begin with the odd group computation.  Note that all of the sheaves appearing in Theorem~\ref{thm: computation for F-curves} are of the form needed for either Lemma~\ref{lem: only junk in odd degrees} or Lemma~\ref{lem: only junk in odd degrees coker of v} unless $r+1=t$ and $e\nmid up^{s-1}$. In the case when $r+1=t$ and $e\nmid ep^{s-1}$ we then have a cofiber sequence of sheaves \[\mathcal{N}^0\mathcal{O}_{\DD}\otimes \mathcal{I}_{s-1}\{r-1\}\to \mathcal{N}^{\geq 0}\mathcal{O}_{\widehat{\DD}}/\mathcal{N}^{\geq 1}\mathcal{O}_{\widehat{\DD}}\otimes \mathcal{I}_{s-1}\{r-1\}\to \mathcal{N}^{\geq 1}\mathcal{O}_{\DD}/(\mathcal{N}^{\geq 1}\mathcal{O}_\DD\otimes \mathcal{I}_{s-1})\{r-1\}\] and so a cofiber sequence after tensoring with $\mathcal{I}^{-\lfloor t/p\rfloor}$. Lemma~\ref{lem: only junk in odd degrees coker of v} applies to sheaf on the right so the only possible contribution to the rank can come from the left-hand sheaf. We also have a canonical identification of $\mathcal{N}^{0}\mathcal{O}_{\DD}$ with the ($p$-completed) structure sheaf and a canonical trivialization of all the other line bundles appearing in the left-hand term of this fiber sequence after base changing. Hence from the long exact sequence we have that we get that we get $H^0(A; \mathcal{N}^{\geq 0}\mathcal{O}_{\widehat{\DD}}/\mathcal{N}^{\geq 1}\mathcal{O}_{\widehat{\DD}}\otimes \mathcal{I}_{s-1}\{r-1\})=A$.

To finish the first part of the proof it only remains to count how many times $t=r+1$ with $e\nmid up^{s-1}$. By definition this amounts to counting the number of integers $up^{s-1}$ with $\lfloor (up^{s-1}-1)/e\rfloor =r+1$ and $e\nmid up^{s-1}$. The options for $up^{s-1}$ are $\{n(r+1)+1, n(r+1)+2, \ldots, n(r+1)+n-1\}$.

For the even groups, note that the recurrence relation of \cite[Theorem B(ii)]{AGH} is the input needed to make the arguments of \cite[Proposition 3.1(ii)]{AGH} and \cite[Proposition 3.2(ii)]{AGH} work. From Lemma~\ref{lem: only junk in odd degrees} and the proof of Lemma~\ref{lem: only junk in odd degrees coker of v} we have the correct generalization of this relation. This can be worked out, but if we take Theorem~\ref{thm: agh} as input it makes the argument easier. We take this approach below.

The recursive formulas set up in Lemma~\ref{lem: only junk in odd degrees} and Lemma~\ref{lem: only junk in odd degrees coker of v}, and therefore the formulas for $p$-adic valuation of the sizes of the groups $K_{2i}(A[x]/x^n, (x);\ZZ_p)$, do not depend on $A$ itself but only the values $|k|$, $e$, and $v_\pi(E'(\pi))$. In particular the recursion relations give that\[v_p(|K_{2i}(A[x]/x^n, (x);\ZZ_p)|)=l(ev_p(|k|), v_p(|k|)v_\pi(E'(\pi)))\] where $l(x,y)$ is some fixed homogeneous linear polynomial. From Theorem~\ref{thm: agh} \[l(1,0)=v_p((ni)!(i!)^{n-2})\] and so to determine $l(x,y)$ we need only find $l(0,1)$. By the recursive formula from Lemma~\ref{lem: only junk in odd degrees} the coefficient of $v_p(|k|)v_\pi(E'(\pi))$ in $v_p(|H^1(A; \mathcal{I}^{-\lfloor t/p\rfloor}\otimes(\mathcal{N}^{\geq (i+1)-t}\mathcal{O}_{\widehat{\DD}}/\mathcal{N}^{\geq (i+1)+1-t}\mathcal{O}\otimes \mathcal{I}_{s-1})\{i\})|)$ is $s$ unless $i+1=t$ in which case the above argument shows the coefficient is $s-1$. The only other terms appearing in Theorem~\ref{thm: computation for F-curves} which we need to account for are the sheaves from Lemma~\ref{lem: only junk in odd degrees coker of v} which by the proof of that Lemma each have coefficient $v_p(n)$. These are, except for when $i+1=t$, the values of the function $h(p,i+1,n,u)$ given in \cite[Lemma 2]{Speirs_truncated}. When $i+1=t$ then $h(p,i+1,n,u)=s$ but our coefficient is $s-1$, so we must subtract $1$ for each time this occurs to match these values up, which from the above occurs $n-1$ times. It then follows that  

\begin{align*} 
l(0,1)&=(\sum_{u\in J_p}h(p,i+1,n,u))-(n-1)\\ 
&=v_p(|\mathbb{W}_{(i+1)n}(\mathbb{F}_p)/V_{n}\mathbb{W}_{i+1}(\mathbb{F}_p)|)-(n-1)\\ 
&=n(i+1)-(i+1)-(n-1)\\ 
&= ni-i 
\end{align*} 
as desired. Here the second equality is coming from \cite[Lemma 2]{Speirs_truncated}. 
\end{proof} 
\printbibliography
\end{document}